\documentclass[11pt, a4paper, USenglish]{article}
\usepackage[margin=30mm]{geometry}
\usepackage[USenglish]{babel}
\usepackage[T1]{fontenc}
\usepackage[utf8]{inputenc}
\usepackage{lmodern}
\usepackage{relsize}
\usepackage{amsmath, amssymb, amsfonts, amsthm, mathtools, sectsty}
\usepackage{mathrsfs}
\usepackage{enumerate}
\usepackage{hyperref}
\allowdisplaybreaks

\numberwithin{equation}{section}

\newtheorem{thm}{Theorem}[section]
\newtheorem{lem}[thm]{Lemma}

\theoremstyle{definition}
\newtheorem{dfn}[thm]{Definition}

\theoremstyle{remark}
\newtheorem{rmk}[thm]{Remark}
\newtheorem*{clm*}{Claim}
\newtheorem{exm}[thm]{Example}

\DeclareMathOperator{\e}{e}
\DeclareMathOperator{\sgn}{sgn}

\DeclareMathOperator{\id}{id}
\DeclareMathOperator{\dvg}{div}
\DeclareMathOperator{\dom}{dom}
\DeclareMathOperator{\dist}{dist}
\DeclareMathOperator{\conv}{conv}
\DeclareMathOperator{\loc}{loc}
\DeclareMathOperator{\ess}{ess}

\DeclareMathOperator{\Lip}{Lip}
\DeclareMathOperator{\sdist}{sdist}

\newcommand{\real}{\ensuremath{\mathbb{R}}}
\newcommand{\nat}{\ensuremath{\mathbb{N}}}
\newcommand{\weakstar}{\ensuremath{\overset{{*}}{\rightharpoonup}}}
\newcommand{\weak}{\ensuremath{\rightharpoonup}}
\newcommand{\redbd}{\ensuremath{\partial^{*}}}
\newcommand{\hausd}{\ensuremath{\mathcal{H}^{d-1}}}
\newcommand{\rmeas}{\ensuremath{\mathcal{M}}}
\newcommand{\pp}{\ensuremath{^{\prime\prime}}}
\newcommand{\phn}{\ensuremath{\nu_{\eps}}}
\newcommand{\re}[2]{\ensuremath{\mathscr{E}\hspace{-2.5pt}\left[#1\big|#2\right]}}
\newcommand{\epsre}[2]{\ensuremath{\mathscr{E}_{\eps}\hspace{-2.5pt}\left[#1\big|#2\right]}}
\newcommand{\be}[2]{\ensuremath{\mathscr{B}\hspace{-2.5pt}\left[#1\big|#2\right]}}
\newcommand{\leb}[1]{\ensuremath{\mathcal{L}^{#1}}}
\newcommand{\mres}{\mathbin{\vrule height 1.6ex depth 0pt width
		0.13ex\vrule height 0.13ex depth 0pt width 1.3ex}}
\newcommand{\horiz}{\ensuremath{T^{\prime}}}
\newcommand{\cnght}{\ensuremath{c_{0}}}
\newcommand{\ggtv}{\ensuremath{\big|\nabla \chi\big|}}
\newcommand{\trans}{\ensuremath{^{T}}}
\newcommand{\pol}{\ensuremath{^{\circ}}}
\newcommand{\uh}[1]{\ensuremath{u_{h}^{#1}}}
\newcommand{\volint}{\ensuremath{\int_{0}^{T}\int_{\Omega}}}
\newcommand{\domint}{\ensuremath{\int_{\Omega}}}
\newcommand{\eps}{\varepsilon}
\newcommand{\tha}{\vartheta}
\newcommand{\tst}{\ensuremath{C_{c}^{\infty}}}
\newcommand{\scA}{\ensuremath{\mathscr{A}}}

\begin{document}
\title{Diffuse-interface approximation and weak-strong uniqueness of anisotropic mean curvature flow}

\author{Tim Laux\footnote{University of Bonn, \url{tim.laux@hcm.uni-bonn.de}} \hspace{1em}
Kerrek Stinson\footnote{University of Bonn, \url{kerrek.stinson@hcm.uni-bonn.de}} \hspace{1em} Clemens Ullrich\footnote{University of Erlangen-Nuremberg, \url{clemens.ullrich@fau.de}}}

\maketitle

\begin{abstract} 
The purpose of this paper is to derive anisotropic mean curvature flow as the limit of the anisotropic Allen--Cahn equation. We rely on distributional solution concepts for both the diffuse and sharp interface models, and prove convergence using relative entropy methods, which have recently proven to be a powerful tool in interface evolution problems. With the same relative entropy, we prove a weak-strong uniqueness result, which relies on the construction of gradient flow calibrations for our anisotropic energy functionals.



\end{abstract}

\tableofcontents

\setcounter{page}{1}

\section{Introduction}

We consider anisotropic mean curvature flow, a geometric evolution equation used to model microstructure in complex materials. The prototypical application is in multi-phase grain growth for polycrystals. As noted in \cite{GARCKE199887}, isotropic models, such as mean curvature flow, fail to capture phenomenological features such as the dendritic growth of phases (see also \cite{LQChen}). In chemical kinetics, phase separation experiences anisotropy due to the underlying lattice orientation of the solid host-material \cite{Bazant-Theory2013}. Similarly, many materials even display anisotropic surface tensions which are not smooth with respect to the interface orientation (see also \cite{angenentGurtin89,taylor78,TaylorCahnHandwerker}): here, one can even consider the household setting of salt (NaCl) and air. At the same time study of interface evolutions poses serious numerical and mathematical challenges and a large amount of insight has been gained by modeling such systems in terms of phase-field models, where one replaces interfaces by continuous order parameters (see, e.g., \cite{LQChen}). 

In this paper, we prove convergence of solutions of the anisotropic scalar Allen--Cahn equation, a phase-field model, to anisotropic mean curvature flow using variational methods. This may be considered as a first step to proving convergence in the physically relevant vectorial setting. Our approach generally sheds light onto anisotropic mean curvature flow, and further enables us to prove a weak-strong uniqueness result for the interface evolution.

Anisotropic mean curvature flow prescribes the evolution of an oriented hypersurface with the surface velocity determined by a weighted mean curvature. Fixing a surface tension $\sigma: \real^{d}\to \real_{\geq 0}$ and a mobility $\mu: \real^{d}\to \real_{\geq 0}$ (where one can think of extending from the sphere by one-homogeneity), we say that a time-parametrized collection of sets $\{A(t)\}_{t\in [0,T]}$ evolves by anisotropic mean curvature flow if the boundary $\Gamma(t):= \partial A(t)$ is smooth and satisfies 
\begin{equation}\label{strong-amcf}
V = -\mu(\nu)H_{\sigma} \qquad \text{on } \bigcup_{t \in [0,T]}\left(\Gamma(t)\times \{t\}\right),
\end{equation}
where $\nu$ is the outer unit normal of $A$, $V$ is the surface normal velocity and $H_{\sigma}:= \dvg_{\Gamma(t)}(D\sigma(\nu))$ is the anisotropic mean curvature with respect to $\sigma$. We note that $D\sigma(\nu)$ is typically referred to as the Cahn--Hoffman vector field, a generalization of the surface normal. In the case that $\sigma = \mu = |\cdot|$ are given by the Euclidean norm, one recovers the usual (isotropic) mean curvature flow $V=-H$.

 Following the approach of Luckhaus and Sturzenhecker \cite{LuckhausSturzenhecker}, one way to encode the motion (\ref{strong-amcf}) is through the characteristic functions $\chi(t) : = \chi_{A(t)}$. Here $A(t)$ are naturally given by sets of finite perimeter instead of smooth open sets, allowing for distributional (or $BV$) weak solutions to (\ref{strong-amcf}). In the isotropic setting, solutions were derived via a minimizing movements scheme (see also \cite{AlmgrenTaylorWang}) for the perimeter functional. This was a natural approach as mean curvature flow can formally be viewed as the gradient flow of the perimeter functional with an appropriate metric \cite{Serfaty2011,Hensel2021l}.

To carry this analogy to our setting, the curvature flow (\ref{strong-amcf}) seeks to minimize an anisotropic surface energy $E: L_{\loc}^{2}(\Omega)\to [0,\infty]$ defined by
\begin{equation}\label{surface-energy}
E[u]:=\begin{cases}
\cnght\int_{\Omega}\sigma(\nu)\ggtv &\text{if }u =\chi \in BV_{\loc}(\Omega;\{0,1\}),\\
\infty&\text{otherwise.}
\end{cases}
\end{equation}
Here, $\nu=-\frac{\nabla \chi}{|\nabla \chi|}$ is the measure-theoretic outer unit normal, and $\cnght$ is a positive constant quantifying surface energy. 
Formally speaking (see Subsection \ref{subsec:anisoMCF+surfaceE}), anisotropic mean curvature flow is a gradient flow of the anisotropic perimeter $E$ with respect to the weighted $L^2$-surface metric
\begin{equation}\label{degenerate-metric}
\left(V,W\right)_{\Gamma}:=\cnght\int_{\Gamma}\frac{1}{\mu(\nu)}V W d\hausd.
\end{equation}

To construct solutions of interface evolutions, such as (\ref{strong-amcf}), one can approximate via diffuse-interface models. The idea is that the sharp interface $\Gamma(t)$, which captures a jump discontinuity of $u =\chi$, is replaced by a diffuse interface $u_{\eps}$ forming a continuous transition between values close to $1$ and values close to $0$. Diffuse-interface or phase-field models are often used in practice and especially for numerics, where tracking of the interface is reduced to a reaction-diffusion equation.
Herein, we consider a phase-field approximation of the curvature flow (\ref{strong-amcf}) given by the anisotropic Allen--Cahn equation
\begin{equation}\label{aac}
\left.
\begin{aligned}
2g(-\nabla u_{\eps})\partial_{t}u_{\eps}&=-\dvg\left(D f(-\nabla u_{\eps})\right)-\frac{1}{\eps^{2}}W^{\prime}(u_{\eps})& \qquad\text{in } \Omega\times (0,T),\\
u_{\eps}(\cdot, 0)&=u_{\eps,0}&\qquad\text{in } \Omega,
\end{aligned}
\right.
\end{equation}
where $\Omega$ is the $d$-dimensional torus, the function $W:\real \to \real_{\geq 0}$ is a double-well potential with its wells at $0$ and $1$, and the anisotropic surface tension and mobility are encoded by the functions
\begin{equation}\label{aac-f}
f: \, \real^{d}\to \real_{\geq 0}, \qquad f(p)=\sigma^{2}(p)
\end{equation}
and
\begin{equation}\label{aac-g}
g: \, \real^{d}\to \real_{> 0}, \qquad g(p)=\frac{\sigma(p)+1}{\mu(p)+1}.
\end{equation}
For the isotropic case $\sigma = \mu = |\cdot|$, we simply obtain $f=|\cdot|^{2}$ and $g \equiv 1$.

In contrast to the approach taken in \cite{ElliottSchaetzle}, we have introduced a regularization (\ref{aac-g}) of the mobility $\mu$ allowing us to take advantage of the gradient flow structure for the anisotropic Allen--Cahn equation. We define the anisotropic Cahn--Hilliard energy $E_{\eps}: L^{2}(\Omega)\to [0,\infty]$ by
\begin{equation}\label{cahnhilliard}
E_{\eps}[u]:=\begin{cases}
\frac{1}{2}\mathlarger{\int}_{\Omega}\left(\eps f(-\nabla u)+\frac{1}{\eps}W(u)\right)dx& \text{if } u \in H^{1}(\Omega) \text { and }W \circ u \in L^{1}(\Omega),\\
\infty& \text{otherwise},
\end{cases}
\end{equation}
and introduce a weighted $L^{2}$-metric given by
\begin{equation}\label{aac-metric}
(v,w)_{u}:=\int_{\Omega}\eps g(-\nabla u)vwdx
\end{equation}
for a given point $u\in\dom(E_{\eps})$. Note that we omit the dependence on $\eps$ in the notation $(\cdot,\cdot)_{u}$ for convenience. A formal calculation shows that the Allen--Cahn equation (\ref{aac}) is equivalent to 
\begin{equation*}
\partial_{t}u_{\eps}(t) = -\nabla_{u_{\eps}(t)}E_{\eps}[u_{\eps}(t)], \qquad t\in[0,T],
\end{equation*}
where $\nabla_{u_{\eps}(t)}$ is the gradient on $\left(L^{2}(\Omega), (\cdot,\cdot)_{u_{\eps}(t)}\right)$. 
The above equation encapsulates (\ref{aac}) as a gradient flow, and this structure will be exploited for the construction of solutions to (\ref{aac}) in Section \ref{sec:aac} as well as for the sharp-interface limit in Section \ref{sec:aacToamcf}.

A first indication that (\ref{aac}) approximates (\ref{strong-amcf}) is the $\Gamma$-convergence of the associated energies (see \cite{Brad02,DalMasoBook}). With
\begin{equation}\label{cnght}
c_{0}:=\int_{0}^{1}\sqrt{W(s)}ds,
\end{equation} it was shown by Bouchitt\'{e} \cite{Bouchitte} that $E_{\eps}\overset{\Gamma}{\longrightarrow}E$ as $\eps \searrow 0$ with respect to the strong $L^{1}$-topology on the underlying space. Likewise in the spirit of Luckhaus and Modica \cite{LuckhausModica}, Cicalese et al. \cite{CicaleseNagasePisante} verified an anisotropic Gibbs--Thomson relation for the energies connecting the first variation of (\ref{cahnhilliard}) to the limiting minimal surface's anisotropic curvature.

Early results on anisotropic mean curvature flow in the special case $\mu=\sigma$ are due to Chen, Giga, and Goto \cite{ChenGigaGoto}, who proved the existence and uniqueness of viscosity solutions (up to fattening) for smooth surface tensions $\sigma$. 
Almgren, Taylor, and Wang \cite{AlmgrenTaylorWang} introduced a time discretization in the form of a minimizing movements scheme including crystalline surface tensions, yielding the so-called flat-flow solutions for anisotropic mean curvature flow. They also proved a short-time existence result for strong solutions if the surface tension $\sigma$ is smooth.
Bellettini and Paolini \cite{BellettiniPaolini} provide a thorough introduction of the anisotropic mean curvature flow equation for a Finsler metric $\sigma$, i.e., with the surface tension possibly depending on the position in space, and argue formally that the time discretization, the level-set equation proposed in \cite{ChenGigaGoto}, and the anisotropic Allen--Cahn equation lead to solutions to anisotropic mean curvature flow. 

Allowing for sufficiently regular convex surface tensions $\sigma$ and arbitrary mobilities $\mu$, Elliott and Sch\"{a}tzle \cite{ElliottSchaetzle} proved that, in the sharp-interface limit, solutions to the anisotropic Allen--Cahn equation converge to anisotropic mean curvature flow in the sense of the viscosity formulation. Unlike in the present work, Elliott and Sch\"{a}tzle used a discontinuous, non-regularized version of $g$ and, therefore, resorted to viscosity solutions of the phase-field equation. In \cite{chambolleNovaga07}, Chambolle and Novaga prove consistency of the minimizing movements scheme \cite{AlmgrenTaylorWang} and the MBO thresholding scheme for the energy (\ref{surface-energy}) with anisotropic mean curvature flow using viscosity solutions. Similarly, using viscosity solutions and distributional solutions with an energy convergence hypothesis, Chambolle et al. \cite{chambolleDeGennaroMorini} prove convergence of the minimizing movements scheme \cite{AlmgrenTaylorWang} for a translationally dependendent energy to inhomogeneous anisotropic mean curvature flow. 

In the case of non-smooth (or crystalline) surface tensions, the Cahn--Hoffman vector-field is effectively described in terms of a differential inclusion $\nu_{\sigma} \in \partial \sigma(\nu)$ and selection of the appropriate curvature can make the problem nonlocal. Recently, a variety of work has been invested in understanding crystalline curvature flow. Giga and Giga were among the first to develop a robust solution concept in the planar setting \cite{GigaGiga}. For crystalline surface tensions, existence and uniqueness was proven in dimension $d=3$ by Giga and Po{\v z}\'ar in \cite{gigaPozar16}, and the result was ultimately extended to arbitrary dimension in \cite{gigaPozar18} by the same authors.
Chambolle, Morini, and Ponsiglione \cite{ChambolleMoriniPonsiglione} introduced a novel definition of supersolutions, subsolutions, and weak solutions to anisotropic mean curvature flow that is also based on level set techniques and is particularly useful for crystalline surface tensions. They presented an existence and uniqueness result up to fattening and a comparison principle. The same authors together with Novaga \cite{ChambolleMoriniNovagaPonsiglione} extended this result from the special case $\mu=\sigma$ to arbitrary mobilities $\mu$.

Many of the above results are qualitative, but in the spirit of Chen \cite{Chen1992}, given the power of viscosity methods, quantitative rates of convergence have been derived and we refer the interested reader to \cite{bellettiniNovaga00,gigaOhtsukaSchaetzle06} and references therein.

In contrast to the above approaches, we will apply relative entropy methods to identify the limit of (\ref{aac}). A related approach regarding a localized energy excess was introduced by the first author with Otto in \cite{lauxOttoThreshold16} where they proved convergence of the MBO thresholding scheme to multi-phase mean curvature flow. A non-trivial modification of this idea---based on controlling the tilt excess of the sharp or diffuse interface with regard to a smooth approximation---has been used to prove convergence of the vectorial Allen--Cahn equation to multi-phase mean curvature flow \cite{LauxSimon} and derive an associated rate of convergence \cite{FischerMarveggio22} along with optimal quantitative convergence rates for the Allen--Cahn equation to mean curvature flow in the two-phase setting \cite{FischerLauxSimon}. A key feature of viscosity type solutions is the associated comparison principle, which automatically provides uniqueness of solutions up to the issue of fattening. However, in the multi-phase case, fundamentally different tools are needed to address uniqueness: The relative entropy method \cite{FHLS21} has been used to prove weak-strong uniqueness results for a variety of geometric evolution equations including planar multi-phase flows (see, e.g., \cite{LauxVMCF,Hensel2021l,HenselLaux,fischerHensel20}).

In this paper we prove convergence of the anisotropic Allen--Cahn equation (\ref{aac}) to anisotropic mean curvature flow (\ref{strong-amcf}) for arbitrary Lipschitz mobilities $\mu$ and $C^2$ surface tensions $\sigma$ under an energy convergence hypothesis, as is often used in application \cite{lauxOttoThreshold16,LauxStinson22,LuckhausSturzenhecker,kimMelletWu2022}. This result provides a complete proof for the result first announced in \cite{LauxKyoto}. Further, we prove weak-strong uniqueness of the associated distributional solution concept for anisotropic mean curvature flow: If a smooth solution (which we will endow with the structure of a calibrated evolution) and a $BV$ solution share the same initial data, it follows that both solutions coincide for all times in their common interval of existence. We summarize these results here, and refer to Theorems \ref{sil-thm} and \ref{wsu} for precise details.

\begin{thm}
Let $\mu$ be a Lipschitz mobility, $\sigma$ a $C^{2}$ uniformly convex surface tension, and $\Omega$ the $d$-dimensional torus. Then the following holds:
\begin{itemize}
\item Any sequence of weak solutions $u_\varepsilon$ of the anisotropic Allen--Cahn equation (\ref{aac}) with well-prepared initial conditions has a subsequence converging to some limit $u=\chi$ with $\chi :\Omega\times(0,T)\to \{0,1\}$ as $\varepsilon \to 0.$ Under an energy convergence hypothesis, $u$ is a weak solution of anisotropic mean curvature flow.

\item Let $\sigma$ and $\mu$ be smooth. If $\{\mathscr{A}(t)\}_{t \in [0,T]}$ is a strong solution of anisotropic mean curvature flow (\ref{strong-amcf}) and $ \chi:\Omega\times(0,T)\to \{0,1\}$ is a distributional solution of anisotropic mean curvature flow with the same initial condition, i.e., $\chi(\cdot, 0) = \chi_{\mathscr{A}(0)},$
then $\chi \equiv \chi_{\mathscr{A}}$ for all $t \in [0,T].$
\end{itemize}
\end{thm}

We remark that convergence of the anisotropic Allen--Cahn equation is well-studied, but as far as the authors are aware, this has exclusively been done from the perspective of viscosity solutions. Our result considers this from the distributional setting and may ultimately be amenable to tackling the multi-phase setting that is most relevant to physical applications. Further, our uniqueness result for distributional solutions shows that it may be possible to obtain quantitative convergence rates in the spirit of Fischer et al. \cite{FischerLauxSimon} for the anisotropic Allen--Cahn equation.

The structure of the paper is as follows: In Section \ref{sec:mathPrelim}, we introduce the admissible class of anisotropies and, based on the closely connected notions of anisotropic surface energy and anisotropic mean curvature, derive the anisotropic mean curvature flow equation as a formal gradient flow. Here, we further introduce our notion of a distributional solution to anisotropic mean curvature flow. As a preparation for the convergence result, Section \ref{sec:aac} discusses the anisotropic Allen--Cahn equation (\ref{aac}) and establishes the existence---via a time discretization---and regularity of weak solutions. Section \ref{sec:aacToamcf} is devoted to the proof of the conditional convergence theorem. Finally, Section \ref{sec:wsu} covers the weak-strong uniqueness theorem.

\subsection*{Notation}

Throughout the paper let $d\geq 2$ be the ambient dimension. We consider the equations with periodic boundary conditions, i.e., for the domain we will always choose the flat torus $\Omega:= \real^{d}/\mathbb{Z}^{d}$.

The $i$-th unit vector will be denoted by $e_{i}$, and the identity matrix in dimension $d$ will be written as $I_{d}\in \real^{d\times d}$. For the scalar products of vectors $a, b \in \real^{d}$ and of matrices $A, B \in \real^{d \times d}$ we write $a \cdot b := \sum_{i=1}^{d}a_{i}b_{i}$ and $A:B:= \sum_{i,j=1}^{d}(A)_{ij}(B)_{ij}$, respectively. The tensor product of $a, b \in \real^{d}$ is a matrix $a \otimes b \in \real^{d \times d}$, which is given by $(a\otimes b)_{ij}=a_{i}b_{j}$.

The symbol $\nabla$ is reserved for derivatives with respect to the space variable $x \in \Omega$. For a vector field $X: \Omega \to \real^{d}$, the derivative $\nabla X$ is a $(d\times d)$-matrix defined by $(\nabla X)_{ij}=\partial_{j}X_{i}$. Derivatives with respect to variables $p \in \real^{d}$ will be denoted by $D = D_{p}$. In contrast to $\nabla$, $D\sigma(\nu)$ will denote the column vector.

For a set $A \subset \mathbb{R}^d$, $\chi_A$ is the characteristic function taking the value $1$ on $A$ and $0$ in the complement. 
 
By a standard abuse of notation, we refer to the distributional derivative in space of a function $\chi\in BV(\Omega;\{0,1\})$ by $\nabla \chi$. In the case that $\chi \in L^1((0,T);BV(\Omega))$, we write $\nabla \chi : = \nabla \chi (\cdot ,t )\otimes \mathcal{L}^1_{(0,T)}$ in the notation of Young measures.

The divergence of a vector field $X: \Omega \to \real^{d}$ is $\dvg X = \sum_{i=1}^{d} \partial_{i}X_{i}$. The divergence of a matrix field $A: \Omega \to \real^{d \times d}$ is vector-valued, $(\dvg A)_{i}=\sum_{j=1}^{d}\partial_{j}(A)_{ij}$.

For a vector $a \in \real^{d}$ we define the differential operator $a \cdot \nabla :=\sum_{i=1}^{d}a_{i}\partial_{i}$. In particular, for a vector field $X$, we have $(a\cdot\nabla)X=(\nabla X)a$.

\section{Anisotropies and anisotropic mean curvature}\label{sec:mathPrelim}

In this section, we develop the necessary mathematical preliminaries for the rest of the paper. We introduce the notion of admissible anisotropies in Subsection \ref{subsec:admisAniso}, and provide a calculation clarifying the formal view of anisotropic mean curvature flow as gradient flow in Subsection \ref{subsec:anisoMCF+surfaceE}. Finally, we show in the the smooth setting how the anisotropic curvature can be reinterpreted via integration by parts, allowing us to introduce a distributional solution for anisotropic mean curvature flow in Subsection \ref{subsec:distSolnConcept}.

\subsection{Admissible anisotropies}\label{subsec:admisAniso}

In this subsection, we introduce a suitable class of anisotropic surface tensions and mobilities, and discuss some elementary facts about admissible anisotropic surface tensions.

\begin{dfn}\label{anisotropies}
	An admissible pair of anisotropies $(\sigma, \mu)$ is comprised of an admissible surface tension $\sigma: \real^{d} \to \real_{\geq 0}$ that satisfies
	\begin{enumerate}[\bf(S1)]
		\item positive $1$-homogeneity, i.e., $\sigma(\lambda p)=\lambda \sigma(p)$ for all $p\in \real^{d}$, $\lambda > 0$,
		\item positive definiteness, i.e., $\sigma(p)=0$ if and only if $p=0$,
		\item smoothness, i.e., $\sigma \in C^{2}(\real^{d}\setminus\{0\})$,
		\item uniform convexity, i.e., $\sigma^{2}:\real^{d}\to \real_{\geq 0}$ is strongly convex in the sense that there exists a constant $c>0$ satisfying 
		\begin{equation*}
		\sigma^{2}((1-t)p+tq)\leq (1-t)\sigma^{2}(p)+t\sigma^{2}(q)-c\frac{t(1-t)}{2}|p-q|^{2}
		\end{equation*}
		for all $p, q \in \real^{d}$, $t \in (0,1)$,
	\end{enumerate}
	and an admissible mobility $\mu : \real^{d} \to \real_{\geq 0}$ that satisfies
	\begin{enumerate}[\bf(M1)]
		\item positive $1$-homogeneity, i.e., $\mu(\lambda p)=\lambda \mu(p)$ for all $p\in \real^{d}$, $\lambda > 0$,
		\item positive definiteness, i.e., $\mu(p)=0$ if and only if $p=0$,
		\item regularity, i.e., $\mu \in C^{0,1}(\real^{d})$.
	\end{enumerate}	
	The polar norm of $\sigma$ is 
	\begin{equation}\label{polar}
	\sigma\pol: \real^{d}\to \real_{\geq 0}, \qquad \sigma\pol(q)=\sup_{p:\, \sigma(p)\leq 1}p\cdot q.
	\end{equation}
\end{dfn}

\begin{exm}\label{st-exm}
A useful class of admissible surface tensions, which is suitable to approximate arbitrary norms, is introduced by Barrett, Garcke, and N\"{u}rnberg in \cite[(1.12)]{BarrettGarckeNuernberg} as follows: Let $q \in [1, \infty)$, $L \in \nat$, and let $G_{1}, G_{2}, \ldots, G_{L} \in \real^{d \times d}$ be symmetric and positive definite. We define
		\begin{equation*}
		\sigma(p):= \left(\sum_{l=1}^{L}\sigma_{l}(p)^{q}\right)^{\frac{1}{q}}, \qquad \sigma_{l}(p):=\sqrt{p \cdot G_{l}p}.
		\end{equation*}
		One can show that $\sigma$ is an admissible surface tension in the above sense (see also \cite[Remark 1.7.5]{Giga} and \cite[Example 4.6]{BellettiniPaolini}).
\end{exm}

For convenience, let us collect several useful properties of surface tensions. 
\begin{lem}\label{surface-tensions}
	Let $\sigma$ be an admissible surface tension as in Definition \ref{anisotropies}. Then
	\begin{enumerate}[(i)]
		\item $\sigma\pol$ is an admissible surface tension and $\sigma^{\circ\circ}=\sigma$,
		\item $p\cdot q \leq \sigma(p)\sigma\pol(q)$ for all $p, q \in \real^{d}$,
		\item we have $0 < \min_{S^{d-1}}\sigma \leq \max_{S^{d-1}}\sigma < \infty$, and
		\begin{equation*}
		(\min_{S^{d-1}}\sigma)|p| \leq \sigma(p)\leq (\max_{S^{d-1}}\sigma)|p| \qquad \text{for all } p \in \real^{d},
		\end{equation*} 
		\item $D\sigma$ is positively $0$-homogeneous and $D^{2}\sigma$ is positively $(-1)$-homogeneous, i.e.,
		\begin{equation*}
		D\sigma(\lambda p) = D\sigma(p) \qquad \text{and} \qquad D^{2}\sigma(\lambda p)=\frac{1}{\lambda}D^{2}\sigma(p) \qquad \text{for all }\lambda>0,\, p \in \real^{d}\setminus\{0\},
		\end{equation*}
		\item $p\cdot D\sigma(p)=\sigma(p)$ for all $p \in \real^{d}\setminus\{0\}$,
		\item $\sigma\pol(D\sigma(p))=1$ for all $p \in \real^{d}\setminus\{0\}$,
		\item $\sigma\pol(q)D\sigma\left(D\sigma\pol(q)\right)=q$ for all $q \in \real^{d}\setminus\{0\}$.
	\end{enumerate}
\end{lem}

\begin{proof}	
	(ii) can be shown directly from the definition of $\sigma\pol$ and the positive $1$-homogeneity of $\sigma$.
	
	(iii) and (iv) follow from the positive $1$-homogeneity and continuity resp. smoothness of $\sigma$. 

	(v) is a result of (iv) and the fundamental theorem of calculus, cf. \cite[Section 1.7.2]{Giga}:
	\begin{equation*}
	\sigma(p)=\int_{0}^{1}\frac{d}{d\lambda}\sigma(\lambda p)\,d\lambda=\int_{0}^{1}D\sigma(\lambda p)\cdot p \,d\lambda=D\sigma(p)\cdot p.
	\end{equation*}
	
	Giga \cite[Section 1.7.2]{Giga} also provides the proofs for (vi), (vii), and (i) except the uniform convexity of $\sigma\pol$ and the duality statement $\sigma^{\circ\circ}=\sigma$. A proof for $\sigma^{\circ\circ}=\sigma$ can be found in \cite[Theorem 15.1 and Corollary 15.1.1]{Rockafellar}. The uniform convexity of $\sigma\pol$ can be proved in a similar way to \cite[Proposition 1.e.2]{LindenstraussTzafriri}.
\end{proof}

A useful trick with regard to the Euclidean metric $|\cdot|$ is to control quadratic errors for unit vectors via the inequality
\begin{equation*}
\left|p-p^{\prime}\right|^{2}=1+|p^{\prime}|^{2}-2\,p\cdot p^{\prime}\leq 2\left(1-p\cdot p^{\prime}\right) \qquad \text{whenever } |p|=1 \text{ and } \left|p^{\prime}\right|\leq 1,
\end{equation*}
with equality if and only if $|p^{\prime}|=1$. In order to introduce a tilt excess functional for anisotropic mean curvature flow, we are interested in an anisotropic counterpart to the above inequality.

The suitable anisotropic inequality bounds squared distances $|p-p^{\prime}|^{2}$ by a term of the form $\sigma(p)-|p^{\prime}|D\sigma(p^{\prime})\cdot p$.
However, since the mapping $p^{\prime} \mapsto |p^{\prime}|D\sigma(p^{\prime})$ is, in general, not continuously differentiable at $p^{\prime}=0$, we will use a truncated version instead. To this end, we fix a cutoff function $\psi \in C^{\infty}([0,\infty))$ satisfying
\begin{equation}\label{cutoff}
\psi(r)=0\quad \text{for all }r\leq \frac{1}{4},\qquad \psi(r)=1\quad\text{for all }r \geq \frac{1}{2}, \qquad\text{and }\psi^{\prime}\geq 0.
\end{equation} The following lemma contains two estimates featuring the truncated version. 

\begin{lem}\label{dziuk}
	Let $\sigma$ be an admissible surface tension. There exist constants $c_{\sigma}, C_{\sigma}>0$ depending only on $\sigma$ such that
	\begin{enumerate}[(i)]
		\item we have \begin{equation}\label{dziuk-lower}
		\sigma(p)-|p^{\prime}|\psi(|p^{\prime}|)D\sigma(p^{\prime})\cdot p \geq c_{\sigma}\left|p-p^{\prime}\right|^{2}
		\end{equation}
		for all $p, p^{\prime} \in \real^{d}$ such that $|p|=1$ and $\left|p^{\prime}\right|\leq 1$, and
		\item we have \begin{equation}\label{dziuk-upper}
		\sigma(p)-|p^{\prime}|\psi(|p^{\prime}|)D\sigma(p^{\prime})\cdot p\leq C_{\sigma}\left(\left|p-p^{\prime}\right|^{2}+\left(1-\left|p^{\prime}\right|\right)\right)
		\end{equation}
		for all $p, p^{\prime} \in \real^{d}$ such that $|p|=1$ and $\left|p^{\prime}\right|\leq 1$.
	\end{enumerate}
\end{lem}

\begin{proof}
	Variants of (i) and (ii) were provided by Dziuk \cite[Proposition 2.2]{Dziuk} and Laux \cite[Lemma 3.3 and Section 4.2]{LauxKyoto}, respectively. The proof is included here for the reader's convenience.

	\vspace{6pt}
	
	For inequality (i), we consider two cases with respect to $p^{\prime}$: First, if $D\sigma(p^{\prime})\cdot p<\frac{\sigma(p)}{2}$ or $p^{\prime}=0$, we use the estimate $|p-p^{\prime}|^{2}\leq 4$ to obtain
	\begin{align}\label{lower-1}
	\sigma(p)-|p^{\prime}|\psi(|p^{\prime}|)D\sigma(p^{\prime})\cdot p&\geq\sigma(p)\left(1-\frac{1}{2}|p^{\prime}|\psi(|p^{\prime}|)\right)\nonumber \\
	&\geq \frac{1}{2}\sigma(p)\nonumber \\
	&\geq \frac{1}{8}(\min_{S^{d-1}}\sigma )|p-p^{\prime}|^{2}.
	\end{align}
	
	Second, if $p^{\prime}\neq 0$ and $D\sigma(p^{\prime})\cdot p\geq\frac{\sigma(p)}{2}$, we have $\left|p+\frac{p^{\prime}}{|p^{\prime}|}\right|\geq \frac{\min_{S^{d-1}}\sigma}{\max_{S^{d-1}}|D\sigma|}>0$ since 
	\begin{equation*}
	0\leq D\sigma(p^{\prime})\cdot p \leq D\sigma(p^{\prime})\cdot\left(-\frac{p^{\prime}}{|p^{\prime}|}\right)+ (\max_{S^{d-1}}|D\sigma|)|p+p^{\prime}|\leq - (\min_{S^{d-1}}\sigma) + (\max_{S^{d-1}}|D\sigma|)\left|p+\frac{p^{\prime}}{|p^{\prime}|}\right|.
	\end{equation*}
	The uniform convexity assumption \textbf{(S4)} can equivalently be stated as follows (see \cite[Remark 1.7.5]{Giga}): There exists a constant $\underline{\sigma}>0$ such that 
	\begin{equation}\label{s4-new}
	D^{2}\sigma(p^{*})\geq \frac{\underline{\sigma}}{|p^{*}|} \qquad \text{on } \{p^{*}\}^{\perp}
	\end{equation}
	for all $p^{*}\neq 0$. We use a second-order Taylor expansion around $\frac{p^{\prime}}{|p^{\prime}|}$ and write the remainder in terms of an intermediate point $p^{*}=(1-t)p+t\frac{p^{\prime}}{|p^{\prime}|}$ with $t \in [0,1]$. Together with the convexity property (\ref{s4-new}), it follows that
	\begin{align}\label{lower-2a}
	\sigma(p)-D\sigma(p^{\prime})\cdot p &= \sigma(p)- \sigma(\frac{p^{\prime}}{|p^{\prime}|})-D\sigma\left(\frac{p^{\prime}}{|p^{\prime}|}\right)\cdot \left(p-\frac{p^{\prime}}{|p^{\prime}|}\right)\nonumber\\
	&= \frac{1}{2}\left(p-\frac{p^{\prime}}{|p^{\prime}|}\right)\cdot D^{2}\sigma(p^{*})\left(p-\frac{p^{\prime}}{|p^{\prime}|}\right)\nonumber\\
	&\geq \frac{\underline{\sigma}}{2|p^{*}|}{\left|\left(p-\frac{p^{\prime}}{|p^{\prime}|}\right) - \left(\left(p-\frac{p^{\prime}}{|p^{\prime}|}\right)\cdot \frac{p^{*}}{|p^{*}|}\right)\frac{p^{*}}{|p^{*}|}\right|}^{2}\nonumber\\
	&\geq \frac{\underline{\sigma}}{8}\left|p+\frac{p^{\prime}}{|p^{\prime}|}\right|^{2}\left|p-\frac{p^{\prime}}{|p^{\prime}|}\right|^{2}\nonumber\\
	&\geq \frac{\underline{\sigma}}{8}\frac{(\min_{S^{d-1}}\sigma)^{2}}{(\max_{S^{d-1}}|D\sigma|)^{2}}\left|p-\frac{p^{\prime}}{|p^{\prime}|}\right|^{2}.
	\end{align}
	Furthermore, the assumption $D\sigma(p^{\prime})\cdot p \geq \frac{\sigma(p)}{2}$ allows us to compute
	\begin{equation}\label{lower-2b}
	\left(1-|p^{\prime}|\psi(|p^{\prime}|)\right)D\sigma(p^{\prime})\cdot p \geq \left(1-|p^{\prime}|\right)\frac{\sigma(p)}{2}\geq \frac{\min_{S^{d-1}}\sigma}{2}{\left|\frac{p^{\prime}}{|p^{\prime}|}-p^{\prime}\right|}^{2}.	
	\end{equation}
	Finally, a combination of (\ref{lower-2a}) and (\ref{lower-2b}) together with an application of Young's inequality yields
	\begin{align}\label{lower-2}
	\sigma(p)-|p^{\prime}|\psi(|p^{\prime}|)D\sigma(p^{\prime})\cdot p &= \sigma(p)-D\sigma(p^{\prime})\cdot p + \left(1-|p^{\prime}|\psi(|p^{\prime}|)\right)D\sigma(p^{\prime})\cdot p\nonumber \\
	&\geq 2c_{\sigma}\left(\left|p-\frac{p^{\prime}}{|p^{\prime}|}\right|^{2}+ \left|\frac{p^{\prime}}{|p^{\prime}|}-p^{\prime}\right|^{2}\right)\nonumber\\
	&\geq c_{\sigma}\left|p-p^{\prime}\right|^{2},
	\end{align}
	which completes the proof of (i) in the second case.
	
	\vspace{6pt}
	
	For inequality (ii), we distinguish two cases again. First, in the easier case $\left|p-p^{\prime}\right|\geq 1$, we can estimate
	\begin{equation}\label{upper-1}
	\sigma(p)-|p^{\prime}|\psi(|p^{\prime}|)D\sigma(p^{\prime})\cdot p \leq \max_{S^{d-1}}\sigma+\max_{S^{d-1}}|D\sigma| \leq \left(\max_{S^{d-1}}\sigma+\max_{S^{d-1}}|D\sigma|\right) \left|p-p^{\prime}\right|^{2}.
	\end{equation}
	
	Second, if $|p-p^{\prime}|< 1$, then
	\begin{equation*}
	{\left|p-\frac{p^{\prime}}{|p^{\prime}|} \right|}^{2}= 2 - 2 p \cdot \frac{p^{\prime}}{|p^{\prime}|}\leq 2+\frac{1}{|p^{\prime}|}\left(|p-p^{\prime}|^{2}-1\right)<2
	\end{equation*}
	and, therefore, 
	\begin{equation*}
	{\left|(1-t)p+t\frac{p^{\prime}}{|p^{\prime}|}\right|}^{2} = 1-\left(t-t^{2}\right)\left|p-\frac{p^{\prime}}{|p^{\prime}|}\right|^{2}\geq 1-\frac{1}{4}{\left|p-\frac{p^{\prime}}{|p^{\prime}|}\right|}^{2}> \frac{1}{2}
	\end{equation*}
	for all $t \in [0,1]$. As in the proof of inequality (i) above, we introduce a second-order Taylor expansion around $\frac{p^{\prime}}{|p^{\prime}|}$ with an intermediate point $p^{*}=(1-t)p+t\frac{p^{\prime}}{|p^{\prime}|}$, where $t \in [0,1]$. Using this expansion, Lemma \ref{surface-tensions}(ii), (vi), and Young's inequality, we obtain
	\begin{align}\label{upper-2}
	\sigma(p)-|p^{\prime}|\psi(|p^{\prime}|)D\sigma(p^{\prime})\cdot p &= \sigma(p)-\sigma\left(\frac{p^{\prime}}{|p^{\prime}|}\right)-D\sigma\left(\frac{p^{\prime}}{|p^{\prime}|}\right)\cdot (p-\frac{p^{\prime}}{|p^{\prime}|})\nonumber \\
	&\quad +\left(1-|p^{\prime}|\psi(|p^{\prime}|)\right)D\sigma(p^{\prime})\cdot p\nonumber\\
	&\leq\frac{1}{2}\left(p-\frac{p^{\prime}}{|p^{\prime}|}\right)\cdot D^{2}\sigma(p^{*})\left(p-\frac{p^{\prime}}{|p^{\prime}|}\right) +\left(1-|p^{\prime}|\psi(|p^{\prime}|)\right)\sigma(p)\nonumber\\
	&\leq\frac{\sqrt{2}}{2}(\max_{S^{d-1}}|D^{2}\sigma|)\left|p-\frac{p^{\prime}}{|p^{\prime}|}\right|^{2}+2 (\max_{S^{d-1}}\sigma)\left(1-|p^{\prime}|\right)\nonumber\\
	&\leq\sqrt{2}(\max_{S^{d-1}}|D^{2}\sigma|)|p-p^{\prime}|^{2} \nonumber \\
	&\quad+ \left(\sqrt{2}(\max_{S^{d-1}}|D^{2}\sigma|)+2 (\max_{S^{d-1}}\sigma)\right)\left(1-|p^{\prime}|\right),
	\end{align}
	from which (ii) follows in the case $|p-p^{\prime}|<1$. \qedhere

\end{proof}

The following fact on the duality of $\sigma$ and $\sigma\pol$ will be used at a later point. 
\begin{lem}\label{energy-as-sup}
		Let $B \in L^{1}(\Omega)^d$, and let $\sigma$ be an admissible surface tension. Then 
		\begin{equation*}
		\int_{\Omega}\sigma(B)dx=\sup_{\eta}\int_{\Omega}B\cdot \eta \, dx,
		\end{equation*}
		where the supremum is taken over all $\eta \in C^{1}(\Omega)^{d}$ such that $\sigma\pol(\eta)\leq 1$.
\end{lem}

\begin{proof}
For $\eta \in C^{1}(\Omega)^{d}$ such that $\sigma\pol(\eta)\leq 1$, we have that
$$\int_\Omega \sigma(B) \, dx \geq \int_{\Omega}\sigma(B) \sigma\pol(\eta) \, dx \geq \int_{\Omega}B\cdot \eta \, dx$$
by Lemma \ref{surface-tensions}. To see the other bound, whenever $B \neq 0$ note that $\sigma(B) = B\cdot D\sigma(B)$ and that $\sigma\pol(D\sigma(B)) = 1$ by Lemma \ref{surface-tensions}. As $\{\nu : \sigma\pol(\nu)\leq 1\}$ is a convex set, if $\eta$ is given by a smooth mollification of $D\sigma(B)$, the inequality $\sigma\pol(\eta)\leq 1$ still holds. Approximating $D\sigma(B)$ by such $\eta$ shows that $\int_{\Omega}\sigma(B)dx\leq \sup_{\eta}\int_{\Omega}B\cdot \eta \, dx$, concluding the proof.

\end{proof}

\subsection{Anisotropic mean curvature and surface energy}\label{subsec:anisoMCF+surfaceE}

In the remainder of this section, we assume that $(\sigma,\mu)$ is an admissible pair of anisotropies according to Definition \ref{anisotropies}.

Similarly to \cite{BellettiniPaolini}, we introduce the $\sigma$-mean curvature of a hypersurface $\Gamma$ as the surface divergence of the Cahn--Hoffman vector of the outer unit normal $\nu$, i.e., $H_{\sigma}=\dvg_{\Gamma}(D\sigma(\nu))$.
For the notion of the Cahn--Hoffman vector, see \cite[Section 1.3]{Giga}. The surface divergence is given by $\dvg_{\Gamma}X := \dvg X - \nu\cdot \nabla X\, \nu$. Observe that, for every $C^{1}$-extension of the normal $\nu$ to the whole space and every $x \in \Gamma$, we have
\begin{align*}
H_{\sigma}(x)&=\dvg_{\Gamma}(D\sigma(\nu))(x)\\
&=\dvg(D\sigma(\nu))(x)-\nu(x) \cdot \nabla(D\sigma(\nu))(x)\nu(x)\\
&=\dvg(D\sigma(\nu))(x)-\nu(x) \cdot D^{2}\sigma(\nu(x))(\nabla \nu(x))\nu(x)\\
&=\dvg(D\sigma(\nu))(x),
\end{align*}
where the last step uses the fact that $D^{2}\sigma(p)p=\frac{d}{ds}\Big|_{s=0}D\sigma(\e^{s}p)=\frac{d}{ds}\Big|_{s=0}D\sigma(p)=0$ for all $p \in \real^{d}\setminus\{0\}$ by the positive $0$-homogeneity of $D\sigma$, Lemma \ref{surface-tensions}(iv).

Furthermore, we define the anisotropic surface energy $E$ as in (\ref{surface-energy}). The relation between the anisotropic mean curvature and anisotropic surface energy becomes clear from the following theorem, which deals with the first variation and direction of steepest descent of the functional $E$: 

\begin{thm}\label{variation}
	Let $A \subset \real^{d}$ be a bounded open set with $C^{2}$-boundary, and let $\nu$ denote the outer unit normal on $\partial A$. 
	Given a compactly supported vector field $B \in C_{c}^{1}(\real^{d})^{d}$, we define a one-parameter family of diffeomorphisms 
	\begin{equation*}\Psi: \real^{d}\times\left(-\|\nabla B\|_{L^{\infty}}^{-1},\|\nabla B\|_{L^{\infty}}^{-1}\right)\to \real^{d}, \qquad \Psi_{t}(x)=x+tB(x).
	\end{equation*} 
	For $|t| <\|\nabla B\|_{L^{\infty}}^{-1}$, let $A(t):= \Psi_{t}(A)$ and $u(x,t):=\chi_{\Omega(t)}(x)$. Then
	\begin{equation}\label{first-variation}
	\frac{d}{dt}E[u(\cdot,t)]\Big|_{t=0}=\cnght \int_{\partial A}H_{\sigma}\nu \cdot B d\hausd.
	\end{equation}
	Moreover, we define $\|B\|_{*}:=\left(\cnght\int_{\partial A}\sigma\pol(-\sgn(H_{\sigma})B)^{2}\frac{\sigma(\nu)^{2}}{\mu(\nu)}d\hausd \right)^{\frac{1}{2}}$, so that $\|\cdot\|_{*}$ is bounded from below and above by the $L^{2}$-norm for vector-valued functions $\partial A \to \real^{d}$. If we denote the right-hand side of (\ref{first-variation}) by the dual pairing $(dE, B)$, then a solution of the minimization problem
	\begin{equation}\label{steepest-descent}
	\min \left\{(dE,B)\big| \|B\|_{*}\leq 1 \right\}
	\end{equation}
	is given by a positive multiple of $-H_{\sigma}\frac{\mu(\nu)}{\sigma(\nu)}D\sigma(\nu)$.
\end{thm}

\begin{proof}
The identity for the first variation (\ref{first-variation}) was stated by Almgren, Taylor, and Wang \cite[Section 2.2]{AlmgrenTaylorWang} and proved by Bellettini and Paolini in the more general context of Finsler geometry \cite[Theorem 5.1]{BellettiniPaolini}. The proof of the steepest descent statement follows the idea of \cite[Proposition 5.1]{BellettiniPaolini} but incorporates the (arbitrary) mobility $\mu$:

Let $B\in L^{2}(\partial A)^{d}$, then by Lemma \ref{surface-tensions}(ii) and the Cauchy-Schwarz inequality it follows that
\begin{align*}
(dE, B)&= \cnght \int_{\partial A}H_{\sigma}\nu \cdot B d\hausd\\
&=-\cnght \int_{\partial A}|H_{\sigma}|\nu \cdot \left(-\sgn(H_{\sigma})B\right) d\hausd\\
&\geq-\cnght \int_{\partial A}|H_{\sigma}|\sigma(\nu)\,\sigma\pol(-\sgn(H_{\sigma})B)\, d\hausd\\
&\geq-\cnght \left(\int_{\partial A}|H_{\sigma}|^{2}\mu(\nu)\,d\hausd\right)^{\frac{1}{2}} \left(\int_{\partial A}\sigma\pol(-\sgn(H_{\sigma})B)^{2}\frac{\sigma(\nu)^{2}}{\mu(\nu)} d\hausd\right)^{\frac{1}{2}}\\
&=-\sqrt{\cnght} \left(\int_{\partial A}|H_{\sigma}|^{2}\mu(\nu)\,d\hausd\right)^{\frac{1}{2}} \|B\|_{*},
\end{align*}
and Lemma \ref{surface-tensions}(v), (vi) yield that the first inequality is an equality if $B=-\lambda\sgn(H_{\sigma})D\sigma(\nu)$ for some nonnegative function $\lambda: \partial A \to \real_{\geq 0}$. Therefore, if $B$ is a constant positive multiple of $-H_{\sigma}\frac{\mu(\nu)}{\sigma(\nu)}D\sigma(\nu)$, then both inequalities are equalities. The statement follows by the positive $1$-homogeneity of $\|\cdot\|_{*}$.
\end{proof}

This theorem helps us to justify the gradient flow structure for anisotropic mean curvature flow as introduced earlier via (\ref{surface-energy}) and (\ref{degenerate-metric}): Considering that tangential components of the velocity $B$ do not contribute to the variation of $A$, let us restrict ourselves to velocities of the form $B =-\lambda \sgn(H_{\sigma})D\sigma(\nu)$, where $\lambda\in C(\partial A;\real_{\geq 0})$. The normal component of such a vector field $B$ is $V=B\cdot \nu = -\lambda \sgn(H_{\sigma})\sigma(\nu)$, and the metric term becomes
\begin{equation*}
\|B\|_{*}^{2}=\cnght\int_{\partial A}\sigma\pol(\lambda D\sigma(\nu))^{2}\frac{\sigma(\nu)^{2}}{\mu(\nu)}d\hausd = \cnght\int_{\partial A}\lambda^{2}\frac{\sigma(\nu)^{2}}{\mu(\nu)}d\hausd=\cnght \int_{\partial A}\frac{1}{\mu(\nu)}V^{2}d\hausd,
\end{equation*}
which is in accordance with (\ref{degenerate-metric}). The velocity of steepest descent which was given in Theorem \ref{variation} satisfies
\begin{equation*}
V=-H_{\sigma}\frac{\mu(\nu)}{\sigma(\nu)}D\sigma(\nu)\cdot \nu = -\mu(\nu)H_{\sigma}
\end{equation*}
by Lemma \ref{surface-tensions}(v), which (formally) verifies the gradient flow structure of (\ref{strong-amcf}).

\subsection{Distributional solutions to anisotropic mean curvature flow}\label{subsec:distSolnConcept}
We introduce a distributional formulation for (\ref{strong-amcf}) that was proposed in \cite{LauxKyoto}. The idea behind this formulation is to encode the $\sigma$-mean curvature by a $(d\times d)$-matrix via an integration by parts as follows:

\begin{rmk}
Let $A \subset \real^{d}$ be a bounded open set with $C^{2}$-boundary. For every vector field $B\in C^{1}(\real^{d})^{d}$, we have 
\begin{equation}\label{ibp-for-weak-formulation}
\int_{\partial A}\left(\sigma(\nu)I_{d}-\nu\otimes D\sigma(\nu)\right):\nabla B\, d\hausd = \int_{\partial A} H_{\sigma}\,B\cdot \nu \, d\hausd.
\end{equation}

To see this, we extend the normal $\nu$ to a vector field $\nu \in C^{1}(\real^{d})^{d}$, e.g., as a truncation of the normalized gradient $\frac{\nabla \sdist}{|\nabla \sdist|}$ of the signed distance function. There exists an open neighborhood $U$ of $\partial A$ where $\nu \neq 0$, and in this neighborhood we can decompose the vector field $B$ as 
\begin{align*}
B&= \left(B-\frac{1}{\sigma(\nu)}(B\cdot\nu)D\sigma(\nu)\right) + \frac{1}{\sigma(\nu)}(B\cdot\nu)D\sigma(\nu)\\
&=: B_{1}+B_{2}.
\end{align*}

The restriction of $B_{1}$ to $\partial A$ is tangential since, by Lemma \ref{surface-tensions}(v), we have $B_{1}\cdot \nu = 0$. By the divergence theorem on the closed surface $\partial A$, we obtain
\begin{align}\label{tangential-div-thm}
0&= \int_{\partial A}\dvg_{\partial A}\left(\sigma(\nu)B_{1}\right)d\hausd\nonumber\\
&= \int_{\partial A}\left(\sigma(\nu)\dvg B_{1}+ D\sigma(\nu)\cdot \nabla \nu B_{1}-\sigma(\nu)\nu \cdot \nabla B_{1}\nu \right)d\hausd\nonumber\\
&= \int_{\partial A}\left(\sigma(\nu)\dvg B_{1}- \nu\cdot \nabla B_{1}D\sigma(\nu)+ \left(D\sigma(\nu)-\sigma(\nu)\nu\right)\cdot \nabla \nu B_{1}+\vphantom{\nu \cdot \nabla B_{1}\left(D\sigma(\nu)-\sigma(\nu)\nu\right)} \right.\nonumber\\
&\hspace{8.5pt}\left.\vphantom{\sigma(\nu)\dvg B_{1}- \nu\cdot \nabla B_{1}D\sigma(\nu)+ \left(D\sigma(\nu)-\sigma(\nu)\nu\right)\cdot \nabla \nu B_{1}}+\nu \cdot \nabla B_{1}\left(D\sigma(\nu)-\sigma(\nu)\nu\right) \right)d\hausd\nonumber\\
&= \int_{\partial A}\left(\left(\sigma(\nu)I_{d}-\nu \otimes D\sigma(\nu)\right):\nabla B_{1}+ \left(D\sigma(\nu)-\sigma(\nu)\nu\right)\cdot \nabla \nu B_{1}+\vphantom{\nu \cdot \nabla B_{1}\left(D\sigma(\nu)-\sigma(\nu)\nu\right)} \right.\nonumber\\
&\hspace{8.5pt}\left.\vphantom{\sigma(\nu)\dvg B_{1}- \nu\cdot \nabla B_{1}D\sigma(\nu)+ \left(D\sigma(\nu)-\sigma(\nu)\nu\right)\cdot \nabla \nu B_{1}}-B_{1} \cdot \nabla \nu\left(D\sigma(\nu)-\sigma(\nu)\nu\right) \right)d\hausd\nonumber\\
&= \int_{\partial A}\left(\sigma(\nu)I_{d}-\nu \otimes D\sigma(\nu)\right):\nabla B_{1}\,d\hausd.
\end{align}
In this computation, the third equality follows by adding zero and using that $\nu \cdot \nabla \nu \,B_{1} =B_{1}\cdot\nabla \left(\frac{1}{2}|\nu|^{2}\right)=0$. For the fourth equality observe that $(\nabla B_{1})\trans \nu + (\nabla \nu)\trans B_{1}=\nabla(B_{1}\cdot\nu)=0$ since $B_{1}$ is tangential. Finally, the fifth equality holds true because the second fundamental form 
\begin{equation*}
\mathrm{I\!I}_{x}:T_{x}\partial A \times T_{x}\partial A \to \real, \qquad \mathrm{I\!I}_{x}(v,w):=w\cdot (v\cdot \nabla)\nu(x)
\end{equation*}
is symmetric for all $x \in \partial A$ (see \cite[Section 1.3]{Giga}).

On the other hand, we have
\begin{align*}
\nabla B_{2}&=-\frac{1}{\sigma^{2}(\nu)}(B\cdot \nu)D\sigma(\nu)\otimes (\nabla \nu)\trans D\sigma(\nu)+\frac{1}{\sigma(\nu)}D\sigma(\nu)\otimes (\nabla \nu)\trans B \\&\quad+\frac{1}{\sigma(\nu)}D\sigma(\nu)\otimes (\nabla B)\trans \nu+\frac{1}{\sigma(\nu)}(B\cdot \nu)D^{2}\sigma(\nu)\nabla \nu,
\end{align*}
which yields
\begin{equation}\label{ch-div-thm}
\left(\sigma(\nu)I_{d}-\nu\otimes D\sigma(\nu)\right):\nabla B_{2}= \dvg(D\sigma(\nu))B\cdot \nu= H_{\sigma}B\cdot \nu.
\end{equation}
From (\ref{tangential-div-thm}) and (\ref{ch-div-thm}) we obtain (\ref{ibp-for-weak-formulation}).
\end{rmk}

Using the integration by parts (\ref{ibp-for-weak-formulation}), we arrive at a distributional formulation for anisotropic mean curvature which resembles the isotropic version due to Luckhaus and Sturzenhecker \cite{LuckhausSturzenhecker}. Following \cite{LauxKyoto}, our definition for $BV$ solutions to anisotropic mean curvature flow also includes an optimal energy dissipation inequality, which alludes to the gradient flow structure of the problem. 

\begin{dfn}[Distributional solutions to anisotropic mean curvature flow]\label{amcf-sol}
	Let $\chi_{0} \in BV(\Omega;\{0,1\})$. A distributional (or $BV$) solution to $V=-\mu(\nu) H_{\sigma}$ with initial condition $\chi_{0}$ is a function $\chi \in BV\left(\Omega\times(0,T);\left\{0,1\right\}\right)\cap C^{0,\frac{1}{2}}\left([0,T];L^{1}(\Omega)\right)\cap L^{\infty}\left(0,T;BV(\Omega)\right)$ such that 
	\begin{enumerate}[(i)]
		\item there exists a $\ggtv$-measurable function $V: \Omega \times (0,T) \to \real$ such that 
		\begin{equation}\label{velocity-integ}
		\volint V^{2}\ggtv < \infty,
		\end{equation} 
		which is the normal velocity in the sense that for all $\zeta \in C^{1}(\Omega\times [0,T])$ and $\horiz \in (0,T]$:
		\begin{align}\label{velocity-criterion}
		\int_{\Omega} \zeta(x,\horiz)\chi(x, \horiz)dx&-\int_{\Omega} \zeta(x,0)\chi_{0}(x)dx \nonumber\\&=\int_{0}^{\horiz}\int_{\Omega} \partial_{t}\zeta(x,t)\chi(x,t)dx dt + \int_{0}^{\horiz}\int_{\Omega} \zeta(x,t)V(x,t)\ggtv,
		\end{align}
		\item
		the relation $V=-\mu(\nu) H_{\sigma}$ is satisfied in a distributional sense, i.e., for all $B \in C^{1}(\Omega \times [0,T])^{d}$, we have
		\begin{equation}\label{amcf-df}
		-\volint \frac{1}{\mu(\nu)}VB\cdot \nu \ggtv = \volint \nabla B : \left(\sigma(\nu)I_{d}-\nu \otimes D\sigma(\nu) \right)\ggtv,
		\end{equation}
		and
		\item 
		the function $\chi$ satisfies the optimal energy dissipation inequality 
		\begin{equation}\label{amcf-oed}
		\int_{\Omega}\sigma(\nu)|\nabla \chi(\horiz)|+\int_{0}^{\horiz}\int_{\Omega}\frac{1}{\mu(\nu)}V^{2}\ggtv \leq \int_{\Omega}\sigma(\nu)|\nabla \chi_0|
		\end{equation}
		for every $\horiz \in [0,T]$.
	\end{enumerate}
\end{dfn}

A prime example of anisotropic mean curvature flow is the following self-similar solution, which generalizes the evolution of shrinking spheres by (isotropic) mean curvature flow. This example of motion by anisotropic mean curvature flow is derived in \cite[Section 1.7.2]{Giga}.

\begin{exm}[The Wulff shape]
	Let $\sigma$ be an admissible surface tension, and suppose that $\mu=\sigma$. The Wulff shape associated with $\sigma$ is the set
	\begin{equation}\label{wulff}
	\mathscr{W}_{\sigma}:=\left\{x \in \real^{d}\Big| \sigma\pol(x)\leq 1 \right\}.
	\end{equation}
	Then 
	\begin{itemize}
		\item the $\sigma$-mean curvature of $\partial \mathscr{W}_{\sigma}$ is $H_{\sigma}\equiv d-1$, and
		\item the family $\{\Gamma(t) \}_{t\in [0,\frac{1}{2(d-1)})}$ defined by $\Gamma(t)=\sqrt{1-2(d-1)t}\ \partial \mathscr{W}_{\sigma}$ evolves by \linebreak anisotropic mean curvature flow $V=-\sigma(\nu)H_{\sigma}$.
	\end{itemize}
\end{exm}

In Theorem \ref{sil-thm}, we will prove convergence of the anisotropic Allen--Cahn equation (\ref{aac}) to a distributional solution of anisotropic mean curvature flow in the sense of Definition \ref{amcf-sol}, but first we must introduce an appropriate notion of solution for (\ref{aac}).

\section{The anisotropic Allen--Cahn equation}\label{sec:aac}
The goal of this chapter is to construct solutions to the anisotropic Allen--Cahn equation~(\ref{aac}) and to establish spatial $H^{2}$-regularity for these solutions.

From now on, our assumptions on the double-well potential $W$ are
\begin{enumerate}[\bf(W1)]
	\item $W(0)=W(1)=0$ and $W(s)>0$ for all $s \in \real\setminus\{0, 1\}$,
	\item $W \in C^{1}(\real)$,
	\item $W$ is twice differentiable at $0$ and $1$, with $W\pp(0),\, W\pp(1)>0$,
	\item there exists $\lambda > 0$ such that $W$ is $\lambda$-convex (see Remark \ref{lambda-convexity}), and
	\item $W$ is nonincreasing on $(-\infty, 0]$ and nondecreasing on $[1,\infty)$.
	\end{enumerate}

While \textbf{(W1)} and \textbf{(W2)} are common assumptions on double-well potentials, \textbf{(W3)} is motivated by the desirable property of exponential convergence of solutions to (\ref{aac}) to the wells far away from the diffuse interface and is not necessary for the main results of existence and sharp-interface limit. This assumption was used in a similar way by Sternberg \cite{Sternberg}. A reference which makes use of the exponential convergence is \cite{FischerLauxSimon}, where Fischer, Simon, and the first author derive a convergence rate for the sharp-interface limit in the isotropic case.

Assumption \textbf{(W4)} is needed in the proof of existence of solutions to guarantee the convergence of the approximate energies. Assumption \textbf{(W5)} is dispensable if one only considers solutions $u_{\eps}$ to (\ref{aac}) that satisfy $0\leq u_{\eps}\leq 1$ almost everywhere.

Up to a linear scaling, a possible choice for $W$ is the standard double-well potential $W(s)=\frac{9}{16}\left(1-s^{2}\right)^{2}$. Here, the prefactor $\frac{9}{16}$ is chosen such that $\cnght:=\int_{-1}^{1}\sqrt{W(s)}ds=1$.

\begin{rmk}\label{lambda-convexity}
	Let $\mathcal{X}$ be a Hilbert space and $\lambda > 0$. A function $F:\mathcal{X}\to [0,\infty]$ is called $\lambda$-convex if the following equivalent conditions are satisfied:
	\begin{enumerate}[(i)]
		\item For all $x_{0}\in \mathcal{X}$, the mapping $x \mapsto F(x)+\frac{\lambda}{2}\left\|x-x_{0}\right\|_{\mathcal{X}}^{2}$ is convex,
		\item for all $x, y \in \mathcal{X}$ and $\mu \in (0,1)$, we have
		\begin{equation*}
			F\left((1-\mu)x+\mu y\right)\leq (1-\mu)F(x)+\mu F(y) + \frac{\lambda}{2}\mu(1-\mu)\left\|y-x\right\|_{\mathcal{X}}^{2}.
		\end{equation*}
	\end{enumerate}
\end{rmk}

Furthermore, we will always assume that $(\sigma, \mu)$ is an admissible pair of anisotropies in accordance with Definition \ref{anisotropies}.
The information on the surface tension $\sigma$ and mobility $\mu$ is contained in the functions $f$, $g$ as defined in (\ref{aac-f}) and (\ref{aac-g}), respectively. We remark that there is some freedom with regard to the choice of $g$: Since $\frac{\sigma(p)}{\mu(p)}$ is not defined for $p=0$, adding $+1$ in the numerator as well as the denominator (cf. (\ref{aac-g})) is a means of avoiding a singularity at $0$.
Our existence and conditional convergence statements rely merely on the following properties of $g$:
\begin{lem}\label{g-properties} Letting $g$ be as in (\ref{aac-g}), the following holds:
\begin{enumerate}[(i)]
	\item The function $g: \real^{d}\to \real_{>0}$ is continuous,
	\item we have $c_{g}:=\inf_{\real^{d}}g > 0$, and
	\item \begin{equation}\label{g-asymp}
	\lim_{|p|\to \infty}\left|g(p)- \frac{\sigma(p)}{\mu(p)} \right|=0.
	\end{equation}
\end{enumerate}	
\end{lem}

The following weak solution concept for the anisotropic Allen--Cahn equation combines an integration by parts in space (ii) with an optimal energy dissipation identity (iii), which hints at the gradient flow structure of (\ref{aac}). Prescribing the initial data as in (i) is possible because $u_{\eps}(0)$ is well-defined via the embedding $H^{1}(\Omega\times(0,T))\subset C^{0,\frac{1}{2}}\left([0,T];L^{2}(\Omega)\right)$.

\begin{dfn}[Solutions to the anisotropic Allen--Cahn equation]\label{aac-solution}
	Let $\eps > 0$ and $u_{\eps,0} \in \dom(E_{\eps})$. A function 
	\begin{equation*}u_{\eps} \in H^{1}(\Omega\times (0,T)) \cap L^{\infty}(\Omega \times(0,T)) 	\end{equation*} 
	is a solution to the anisotropic Allen--Cahn equation with initial data $u_{\eps,0}$ if 
	\begin{enumerate}[(i)]
		\item 
		\begin{equation}\label{aac-init}
		u_{\eps}(0) = u_{\eps,0},
		\end{equation}
		\item
		\begin{equation}\label{aac-df}
		\volint g(-\nabla u_{\eps})\partial_{t}u_{\eps}\varphi \, dx dt = \volint \frac{1}{2}\left(D f(-\nabla u_{\eps}))\cdot \nabla \varphi-\frac{1}{\eps^{2}}W^{\prime}(u_{\eps}) \varphi\right) dx dt
		\end{equation}
		for all $\varphi \in C^{1}(\Omega\times[0,T])$, and
		\item 
		\begin{equation}\label{aac-oed}
		E_{\eps}[u_{\eps}(\horiz)] + \int_{0}^{T^{\prime}}\int_{\Omega}\eps g(-\nabla u_{\eps})(\partial_{t}u_{\eps})^{2}dx dt = E_{\eps}[u_{\eps,0}]
		\end{equation}
		for all $\horiz \in (0,T]$.
	\end{enumerate}
\end{dfn}

We comment on the typical behavior of the anisotropic Allen--Cahn equation (\ref{aac}). 

\begin{rmk}
The reaction term $-\frac{1}{\eps^{2}}W^{\prime}(u_{\eps})$ forces the solution $u_{\eps}$ towards the wells of $W$, i.e., towards the set $\{0, 1\}$. The interplay of the reaction term with the anisotropic diffusion term $-\dvg\left(Df(-\nabla u_{\eps})\right)$ results in a transition layer of width $\mathcal{O}(\eps)$ whose shape also depends on the direction of the approximate outer normal $-\frac{\nabla u_{\eps}}{|\nabla u_{\eps}|}$. More precisely, the typical width of the transition layer is $\eps \sigma\left(-\frac{\nabla u_{\eps}}{|\nabla u_{\eps}|}\right)$.

To see this, one can choose the $1$-dimensional stationary ansatz
\begin{equation}\label{profile}
u_{\eps}: \real^{d}\times \real_{\geq 0}\to \real, \qquad u_{\eps}(x,t)=U_{\eps}(x\cdot \nu),
\end{equation}
with $\nu \in S^{d-1}$ being a fixed unit-length vector. 
Together with a monotonicity assumption on the transition, this admits the solution
\begin{equation}\label{profile-sol}
u_{\eps}(x,t)=\Theta\left(-\frac{x\cdot\nu}{\eps \sigma(\nu)}\right),
\end{equation}
where $\Theta$ is the unique solution to \begin{equation}\label{theta-ode}
\Theta^{\prime}=\sqrt{W(\Theta)}, \qquad \Theta(0)=1/2.
\end{equation}

With this in mind, we note that the weighted $L^{2}$-metric (\ref{aac-metric}) depends on $\sigma$ and $\mu$ whereas anisotropic mean curvature is a gradient flow with respect to the metric (\ref{degenerate-metric}), which depends only on $\mu$. This can be explained by the idea that, in the anisotropic Allen--Cahn equation, the metric (\ref{aac-metric}) has to compensate for the typical width of the transition layer varying with the orientation of the normal.
\end{rmk}
\subsection{Existence of weak solutions}

\begin{thm}[Existence]\label{aac-ex}
	Let $\eps > 0$ and $u_{0} \in \dom(E_{\eps})\cap L^{\infty}(\Omega)$. For any finite time horizon $T \in (0,\infty)$, there exists a solution $u$ to the anisotropic Allen--Cahn equation in the sense of Definition \ref{aac-solution} with initial data $u_{0}$. Furthermore, $u$ satisfies 
	\begin{equation}\label{aac-linfty}
	\|u-1/2\|_{L^{\infty}(\Omega\times(0,T))}\leq \max\left\{\|u_{0}-1/2\|_{L^{\infty}(\Omega)},1/2 \right\}.
	\end{equation}
\end{thm}

To prove this existence theorem, we exploit the gradient-flow structure of the equation and construct solutions via a minimizing movements scheme. To this end, we consider an approximation of the PDE (\ref{aac}) that replaces the time derivative by difference quotients: Given a time-step size $h>0$, let us look for functions $\{u_{h}^{n}\}_{n\in\nat_{0}}$ that solve
\begin{align}
2g(-\nabla \uh{n-1})\frac{\uh{n}-\uh{n-1}}{h}&=-\dvg(Df(-\nabla \uh{n})) -\frac{1}{\eps^{2}}W^{\prime}(\uh{n}), \qquad n\in\nat, \label{aac-strongele} \\
u_{h}^{0}&=u_{0}. \label{aac-ic}
\end{align}
In fact, equation (\ref{aac-strongele}) is the strong Euler--Lagrange equation of the variational problem
\begin{equation}\label{aac-mm}
\uh{n}\in \underset{u \in L^{2}(\Omega)}{\arg \min}\left\{E_{\eps}[u]+\frac{1}{2h}\left\|u-\uh{n-1}\right\|_{\uh{n-1}}^{2} \right\}, \qquad n \in \nat,
\end{equation}
where we recall the energy (\ref{cahnhilliard}) and the inner-product (\ref{aac-metric}) for which we use the notation
\begin{equation*}
\|v\|_{u}:={(v,v)_{u}}^{\frac{1}{2}},
\end{equation*}
in which the dependence on $\eps$ is suppressed again.

Equation (\ref{aac-strongele}) does not precisely resemble an implicit Euler scheme as it features the explicit term $g(-\nabla \uh{n-1})$ rather than the implicit term $g(-\nabla \uh{n})$. Instead (\ref{aac-strongele}) can be viewed an implicit-explicit splitting discretization scheme of (\ref{aac}). This is mirrored in the minimization problem (\ref{aac-mm}) by the occurrence of the metric $(\cdot,\cdot)_{\uh{n-1}}$, which is taken with respect to the constant point $\uh{n-1}$ and therefore does not depend on $u$. 
Thanks to this choice, one can immediately prove that (\ref{aac-mm}) admits a unique solution as soon as $h$ is small enough, using the direct method in the calculus of variations:

\begin{lem}\label{mmscheme-dmcov} Let $E_\varepsilon$ be defined as in (\ref{cahnhilliard}) and assumptions \textbf{(W1)}--\textbf{(W5)} hold. It follows that:
	\begin{enumerate}[(i)]
		\item The energy $E_{\eps}$ is $\frac{\lambda}{2\eps}$-convex on $L^{2}(\Omega)$,
		\item if $h < \frac{2\eps^{2}c_{g}}{\lambda}$, then $u \mapsto E_{\eps}[u]+\frac{1}{2h}\left\|u-\uh{n-1}\right\|_{\uh{n-1}}^{2}$ is strongly convex,
		\item for all $h > 0$, the mapping $u \mapsto E_{\eps}[u]+\frac{1}{2h}\left\|u-\uh{n-1}\right\|_{\uh{n-1}}^{2}$ is lower semicontinuous on $L^{2}(\Omega)$,
		\item for $h < \frac{2\eps^{2}c_{g}}{\lambda}$ and $\uh{n-1} \in \dom(E_{\eps})$, there exists a unique minimizer $\uh{n} \in \dom(E_{\eps})$ for (\ref{aac-mm}), and $\|\uh{n}\|_{L^{\infty}}\leq \max\left\{\|\uh{n-1}\|_{L^{\infty}},1 \right\}$.
	\end{enumerate}
\end{lem}

\begin{proof}
	\begin{enumerate}[(i)]
		\item Let $u, v \in L^{2}(\Omega)$ and $\mu \in [0,1]$. We will use the equivalent formulation (ii) in Remark \ref{lambda-convexity}, so that we have to show that
		\begin{equation*}
		E_{\eps}[(1-\mu)u+\mu v] \leq (1-\mu) E_{\eps}[u]+ \mu  E_{\eps}[v]+\frac{\lambda}{4\eps}\mu(1-\mu)\|v-u\|_{L^{2}}^{2}
		\end{equation*}
		for all $u,v \in L^{2}(\Omega)$ and $\mu \in (0,1)$. We can assume without loss of generality that $u, v \in \dom(E_{\eps})$. Exploiting the convexity of $f$ as well as the $\lambda$-convexity of $W$ \textbf{(W4)}, we obtain
		\begin{align}\label{mmscheme-lambdaeps}
		E_{\eps}[(1-\mu)u+\mu v] 
		&= \frac{1}{2}\int_{\Omega}\left(\eps f(-(1-\mu)\nabla u - \mu \nabla v)+\frac{1}{\eps}W((1-\mu)u+\mu v)\right)dx \nonumber\\
		&\leq \frac{1}{2}\int_{\Omega}\left(\eps (1-\mu) f(-\nabla u) +\eps\mu f(-\nabla v) \right)dx \nonumber\\
		&\quad+ \frac{1}{2}\int_{\Omega}\left(\frac{1-\mu}{\eps}W(u)+\frac{\mu}{\eps}W(v)+\frac{\lambda}{2\eps}\mu(1-\mu)|v-u|^{2}\right)dx \nonumber\\
		&\leq (1-\mu)E_{\eps}[u]+\mu E_{\eps}[v] + \frac{\lambda}{4\eps}\mu(1-\mu)\left\|v-u\right\|_{L^{2}}^{2}.
		\end{align}
		\item This follows from (i) and quantifying the strong convexity of the metric term:
		
		Given $u, v \in L^{2}(\Omega)$ and $\mu \in (0,1)$, we compute
		\begin{align}\label{mmscheme-convex}
		\frac{1}{2h}\left\|((1-\mu)u\vphantom{+\mu v)- \uh{n-1}}\right.& \left.\vphantom{((1-\mu)u}\hspace{-3pt}{+} \mu v)- \uh{n-1}\right\|_{\uh{n-1}}^{2}\nonumber\\	
		&= \frac{1-\mu}{2h}\|u-\uh{n-1}\|_{\uh{n-1}}^{2}+\frac{\mu}{2h}\|v-\uh{n-1}\|_{\uh{n-1}}^{2}\nonumber\\
		&\quad- \frac{\mu(1-\mu)}{2h}\int_{\Omega}\eps g(-\nabla \uh{n-1})|v-u|^{2}dx \nonumber \\
		&\leq \frac{1-\mu}{2h}\|u-\uh{n-1}\|_{\uh{n-1}}^{2}+\frac{\mu}{2h}\|v-\uh{n-1}\|_{\uh{n-1}}^{2}\nonumber\\
		&\quad- \frac{\eps c_{g}}{2h}\mu(1-\mu)\left\|v-u\right\|_{L^{2}}^{2}.
		\end{align}

		Combining (\ref{mmscheme-lambdaeps}) with (\ref{mmscheme-convex}) yields strong convexity for
		$u \mapsto E_{\eps}[u]+\frac{1}{2h}\left\|u-\uh{n-1}\right\|_{\uh{n-1}}^{2}$ if
		$\frac{\lambda}{4\eps}<\frac{\eps c_{g}}{2h},$
		or, equivalently, if $h<\frac{2\eps^{2}c_{g}}{\lambda}$.
		
		\item This will follow from the lower semi-continuity of the individual terms. To see that the anisotropic Dirichlet energy is lower semi-continuous, we can assume without restriction that $v_{k} \in H^{1}(\Omega)$ for all $k \in \nat$, so that it remains to prove that
		\begin{equation*}
		\frac{\eps}{2}\int_{\Omega}f(-\nabla v)dx \leq \liminf_{k \to \infty}\frac{\eps}{2}\int_{\Omega}f(-\nabla v_{k})dx.
		\end{equation*} Without loss of generality, $v_k$ converge to $v$ weakly in $H^1(\Omega),$ and lower semi-continuous follows from convexity of $f$.

		The lower semicontinuity of the potential term $\int_{\Omega} W(u) \, dx$ and the metric term $u\mapsto \frac{1}{2h}\left\|u-\uh{n-1} \right\|_{\uh{n-1}}^{2}$ follow from Fatou's lemma.
		
		\item The functional to be minimized is proper.		
		Let $\left\{v_{k}\right\}_{k \in \nat}$ be a minimizing sequence, i.e.,
		\begin{equation*}
		\lim_{k \to \infty} \left(E_{\eps}[v_{k}]+\frac{1}{2h}\left\|v_{k}-\uh{n-1}\right\|_{\uh{n-1}}^{2}\right) =\inf_{u \in L^{2}(\Omega)}\left\{E_{\eps}[u]+\frac{1}{2h}\left\|u-\uh{n-1}\right\|_{\uh{n-1}}^{2}\right\}.
		\end{equation*}
		
		Every minimizing sequence is bounded in $H^{1}(\Omega)$ since
		\begin{align}\label{mmscheme-ex}
		\left\|v_{k}\right\|_{H^{1}}^{2}&\leq 2\left\|v_{k}-\uh{n-1}\right\|_{L^{2}}^{2} + 2 \left\|\uh{n-1}\right\|_{L^{2}}^{2}+ \left\|\nabla v_{k}\right\|_{L^{2}}^{2}\nonumber\\
		&\leq \frac{2}{\eps c_{g}}\left\|v_{k}-\uh{n-1}\right\|_{\uh{n-1}}^{2} + 2 \left\|\uh{n-1}\right\|_{L^{2}}^{2}+ \frac{1}{(\min_{S^{d-1}}\sigma)^{2}}\int_{\Omega}f(-\nabla v_{k})dx\nonumber\\
		&\leq 2 \left\|\uh{n-1}\right\|_{L^{2}}^{2}+\left(\frac{4h}{\eps c_{g}}+\frac{2}{\eps(\min_{S^{d-1}}\sigma)^{2}} \right)\left(E_{\eps}[v_{k}]+\frac{1}{2h}\left\|v_{k}-\uh{n-1}\right\|_{\uh{n-1}}^{2}\right).
		\end{align}
		
		By Rellich's theorem, there exists a subsequence converging in $L^{2}(\Omega)$ to a limit function $v \in L^{2}(\Omega)$. The lower semicontinuity (iii) then shows that
		\begin{equation*}
		v \in \underset{u \in L^{2}(\Omega)}{\arg\min}\left\{E_{\eps}[u]+ \frac{1}{2h}\left\|u-\uh{n-1}\right\|_{\uh{n-1}}^{2} \right\}.
		\end{equation*}
		The minimizer $v$ is unique due to the strong convexity (ii).

		For the $L^{\infty}$-bound we define a truncated version of $\uh{n}$ by
		\begin{equation*}
		v(x):=\begin{cases}
		-\alpha^{n-1}+\tfrac12 &\text{if } \uh{n}(x)<-\alpha^{n-1}+\tfrac12,\\
		\uh{n}(x)&\text{if } |\uh{n}(x)-\tfrac12|\leq \alpha^{n-1},\\
		\alpha^{n-1}+\tfrac12 &\text{if } \uh{n}(x)>\alpha^{n-1}+\tfrac12,
		\end{cases}
		\end{equation*}
		where $\alpha^{n-1} := \max\{\|\uh{n-1}-\tfrac12\|_{L^{\infty}},\tfrac12\}.$
		Letting $A:= \{x\in \Omega\,\big|\, v(x)\neq \uh{n}(x) \}$, it follows from $\uh{n} \in \dom(E_{\eps})\subseteq H^{1}(\Omega)$ and a general fact about Sobolev spaces that $v \in H^{1}(\Omega)$ and 
		\begin{equation}\label{mmscheme-contender}
		\nabla v(x) =\begin{cases} \nabla \uh{n}(x) &\text{for almost every } x \in \Omega\setminus A,\\
		0&\text{for almost every } x\in A.\end{cases}
		\end{equation}
		Then, using (\ref{mmscheme-contender}), the positive definiteness of $\sigma$, the monotonicity assumption \textbf{(W5)}, and the pointwise inequality $|v-\uh{n-1}|\leq |\uh{n}-\uh{n-1}|$, one obtains
		\begin{align*}
		E_{\eps}[v]+\frac{1}{2h}\left\|v-\uh{n-1}\right\|_{\uh{n-1}}^{2}&\leq\frac{\eps}{2}\int_{\Omega}f(-\nabla \uh{n})dx+\frac{1}{2\eps}\int_{\Omega}W(\uh{n})dx\\
		&\quad+\frac{\eps}{2h}\int_{\Omega}g(-\nabla \uh{n-1})|\uh{n}-\uh{n-1}|^{2}dx\\
		&=E_{\eps}[\uh{n}]+\frac{1}{2h}\left\|\uh{n}-\uh{n-1}\right\|_{\uh{n-1}}^{2}.
		\end{align*}
		By the uniqueness of minimizers, this implies that $v=\uh{n}$ and, therefore,
		\begin{equation*}
		\|\uh{n}-\tfrac12\|_{L^{\infty}}=\|v-\tfrac12\|_{L^{\infty}}\leq \alpha^{n-1} = \max\left\{\|\uh{n-1}-\tfrac12\|_{L^{\infty}},\tfrac12 \right\}.\qedhere
		\end{equation*}
	\end{enumerate}
\end{proof}

In order to pass to the limit $h \searrow 0$, we define one affine and two piecewise constant interpolations: For $n \in \nat$ such that $t \in [(n-1)h, nh)$ let 
\begin{align}\label{mmscheme-interpol}
u_{h}(t)&:= \frac{t-(n-1)h}{h}\uh{n}+\frac{nh-t}{h}\uh{n-1}, \nonumber\\
\overline{u}_{h}(t)&:= \uh{n}, \nonumber\\
\underline{u}_{h}(t)&:= \uh{n-1},
\end{align}
where $\uh{0} \in \dom(E)$ and $\uh{n}$ are defined inductively via (\ref{aac-mm}) for $n \in \nat$.

Before we turn to the limiting function $u$, let us prove the following inequality, which is a consequence of the $\frac{\lambda}{2\eps}$-convexity of $E_{\eps}$ and the minimizing property of the time steps $\uh{n}$ (\ref{aac-mm}):

\begin{lem}\label{mmscheme-conv-ineq}
	Let $h<\frac{2\eps^{2}c_{g}}{\lambda}$ and $w \in L^{2}(\Omega\times (0,T))$. Then
	\begin{equation}\label{mmscheme-conv-ineq-2}
	E_{\eps}[w(t)]\geq E_{\eps}[\overline{u}_{h}(t)] - \int_{\Omega}\eps g(-\nabla \underline{u}_{h}(t))\partial_{t}u_{h}(t)\left(w(t)-\overline{u}_{h}(t)\right)dx - \frac{\lambda}{4\eps}\int_{\Omega}\left|w(t)-\overline{u}_{h}(t)\right|^{2}dx
	\end{equation}
	for almost every $t \in (0,T)$.
\end{lem}

\begin{proof}
	It suffices to prove the corresponding inequality involving the time steps $\uh{n}$, $n \in \nat$:
	Let $h<\frac{\eps^{2}c_{g}}{2\lambda}$ and $n \in \nat$. Then for all $v \in L^{2}(\Omega)$, we have
	\begin{equation}\label{mmscheme-conv-ineq-1}
		E_{\eps}[v]\geq E_{\eps}[\uh{n}]-\int_{\Omega} \eps g(-\nabla \uh{n-1})\frac{\uh{n}-\uh{n-1}}{h}\left(v-\uh{n}\right)dx - \frac{\lambda}{4\eps} \int_{\Omega} \left|v-\uh{n}\right|^{2}dx.
	\end{equation}

	Clearly, one can write the interpolations $u_{h}$, $\overline{u}_{h}$, and $\underline{u}_{h}$ on the left-hand side of (\ref{mmscheme-conv-ineq-2}) in terms of the time steps, choosing $n=n(t) \in \nat$ such that $t \in [(n(t)-1)h, n(t)h)$. We also observe that the piecewise affine interpolation $u_{h} \in C\left([0,T];L^{2}(\Omega)\right)$ has a weak time derivative $\partial_{t}u_{h}$ given by 
	\begin{equation*}
	\partial_{t}u_{h}(t)=\frac{\overline{u}_{h}(t)-\underline{u}_{h}(t)}{h} = \frac{\uh{n(t)}-\uh{n(t)-1}}{h}
	\end{equation*} 
	for almost every $t \in (0,T)$. Therefore, (\ref{mmscheme-conv-ineq-2}) follows from (\ref{mmscheme-conv-ineq-1}) by reformulating the interpolation in terms of the time steps and using $v=w(t)$ for fixed $t \in (0,T)$.

	To prove (\ref{mmscheme-conv-ineq-1}), let $v \in L^{2}(\Omega)$ and $\delta \in (0,1)$. Again, we use the equivalent characterization of $\frac{\lambda}{2\eps}$-convexity in Remark \ref{lambda-convexity}(ii). For computational ease, first note, that
\begin{equation}\nonumber
|\delta v + (1-\delta)u^{n}_h - u^{n-1}_h|^2 = \delta^2 |v - u^{n}_h|^2 + 2\delta(v-u^{n}_h)(u^n_h - u^{n-1}_h) + |u^n_h-u^{n-1}_h|^2 .
\end{equation}	
Subtracting $(1-\delta)|u^n-u^{n-1}|^2$ from both sides of the above equality, we may directly compute
	\begin{align}
	E_{\eps}[v]&\geq \frac{1}{\delta}\left(E_{\eps}[\delta v + (1-\delta)\uh{n}]-(1-\delta)E_{\eps}[\uh{n}]\right)-(1-\delta)\frac{\lambda}{4\eps}\int_{\Omega}|v-\uh{n}|^{2}dx \nonumber \\
	&=\frac{1}{\delta}\left(E_{\eps}[\delta v + (1-\delta)\uh{n}]+\frac{1}{2h}\left\|\delta v +(1-\delta)\uh{n}-\uh{n-1}\right\|_{\uh{n-1}}^{2}
	\vphantom{-(1-\delta)E_{\eps}[\uh{n}]-(1-\delta)\frac{1}{2h}\left\|\uh{n}-\uh{n-1}\right\|_{\uh{n-1}}^{2}}\right.\nonumber\\
	&\hspace{8.5pt}\left.\vphantom{E_{\eps}[\delta v + (1-\delta)\uh{n}]+\frac{1}{2h}\left\|\delta v +(1-\delta)\uh{n}-\uh{n-1}\right\|_{\uh{n-1}}^{2}}-(1-\delta)E_{\eps}[\uh{n}]-(1-\delta)\frac{1}{2h}\left\|\uh{n}-\uh{n-1}\right\|_{\uh{n-1}}^{2}\right)\nonumber\\
	&\quad-\frac{1}{2h}\int_{\Omega}\eps g(-\nabla \uh{n-1}) \left(2(\uh{n}-\uh{n-1})(v-\uh{n})+(\uh{n}-\uh{n-1})^{2}+\delta(\uh{n}-v)^{2}\right)dx\nonumber\\
	&\quad -(1-\delta)\frac{\lambda}{4\eps}\int_{\Omega}|v-\uh{n}|^{2}dx \nonumber \\
	&\geq\frac{1}{\delta}\left(E_{\eps}[\uh{n}]+\frac{1}{2h}\left\| \uh{n}-\uh{n-1}\right\|_{\uh{n-1}}^{2} \vphantom{-(1-\delta)E_{\eps}[\uh{n}]-(1-\delta)\frac{1}{2h}\left\|\uh{n}-\uh{n-1}\right\|^{2}}\right.\nonumber\\
	&\hspace{8.5pt}\left.\vphantom{E_{\eps}[\delta v + (1-\delta)\uh{n}]+\frac{1}{2h}\left\|\delta v +(1-\delta)\uh{n}-\uh{n-1}\right\|_{\uh{n-1}}^{2}}-(1-\delta)E_{\eps}[\uh{n}]-(1-\delta)\frac{1}{2h}\left\|\uh{n}-\uh{n-1}\right\|_{\uh{n-1}}^{2}\right)\nonumber\\
	&\quad-\frac{1}{2h}\int_{\Omega}\eps g(-\nabla \uh{n-1}) \left(2(\uh{n}-\uh{n-1})(v-\uh{n})+(\uh{n}-\uh{n-1})^{2}+\delta(\uh{n}-v)^{2}\right)dx\nonumber\\
	&\quad -(1-\delta)\frac{\lambda}{4\eps}\int_{\Omega}|v-\uh{n}|^{2}dx \nonumber \\
	&=E_{\eps}[\uh{n}]- \int_{\Omega}\eps g(-\nabla \uh{n-1}) \frac{\uh{n}-\uh{n-1}}{h}(v-\uh{n})dx-\frac{\delta}{2h}\int_{\Omega}\eps g(-\nabla \uh{n-1}) (\uh{n}-v)^{2}dx\nonumber\\
	&\quad -(1-\delta)\frac{\lambda}{4\eps}\int_{\Omega}|v-\uh{n}|^{2}dx,
	\end{align}
	where the third step is an application of (\ref{aac-mm}), with $\delta v + (1-\delta)\uh{n}$ acting as a contender for the minimization problem. Taking the limit $\delta \searrow 0$ finally yields (\ref{mmscheme-conv-ineq-1}).
\end{proof}

An application of the Arzel\`{a}--Ascoli theorem, as performed in the lemma below, shows that a subsequence of $\{u_{h} \}_{h}$ converges strongly, and the derivatives in time and space will turn out to converge weakly in $L^{2}$. However, due to the nonlinear term $g(-\nabla u)$ in the anisotropic Allen--Cahn equation (\ref{aac}), weak convergence of the gradients is not sufficient to show that the limit function solves the equation. We will therefore use a convexity argument to also prove strong convergence of the gradients.

\begin{lem}\label{mmscheme-compact}
	Let $\{\uh{0}\}_{h \in (0, \frac{2\eps^{2}c_{g}}{\lambda})}\subset \dom(E)$ such that $\limsup_{h\searrow 0}E_{\eps}[\uh{0}] < \infty$ and \linebreak $\limsup_{h\searrow 0}\|\uh{0}\|_{L^{\infty}(\Omega)}<\infty$. Then there exist a sequence $h\searrow 0$, denoted without relabeling, and a limit function $u \in H^{1}(\Omega\times(0,T))$ such that 
	\begin{enumerate}[(i)] 
		\item 
		\begin{equation*}
		u_{h}(t), \:\overline{u}_{h}(t), \:\underline{u}_{h}(t)\longrightarrow u(t) \qquad \text{in } L^{2}(\Omega)\text{ uniformly in } t\in [0,T],
		\end{equation*}
		\item 
		\begin{equation*}\partial_{t}u_{h}\weak \partial_{t}u \qquad \text{in } L^{2}(\Omega\times(0,T)),
		\end{equation*}
		\item
		\begin{equation*}
		\nabla u_{h},\: \nabla \overline{u}_{h}, \:\nabla\underline{u}_{h}\longrightarrow \nabla u \qquad \text{in } L^{2}(\Omega\times (0,T))^{d}
		\end{equation*}
	\end{enumerate}
	as $h\searrow 0$, where $u_{h}$, $\overline{u}_{h}$, and $\underline{u}_{h}$ are constructed as above. Furthermore, the mapping $u: [0,T]\to L^{2}(\Omega)$ is $\frac{1}{2}$-Hölder continuous and satisfies the $L^{\infty}$-bound $\|u-\tfrac12\|_{L^{\infty}(\Omega\times(0,T))}\leq \max\left\{\limsup_{h\searrow 0}\|\uh{0}-\tfrac12\|_{L^{\infty}(\Omega)},\tfrac12 \right\}$.
\end{lem}
\begin{proof}
	\begin{enumerate}[(i)]
		\item The first step is to prove a $H^{1}$-bound for the minimizers $\uh{n}$, $n \in \nat_{0}$:
		\begin{align}\label{mmscheme-constant-h1}
		\left\|\uh{n}\right\|_{H^{1}(\Omega)}^{2}&=\|\uh{n}\|_{L^{2}(\Omega)}^{2}+\left\|\nabla\uh{n}\right\|_{L^{2}(\Omega)}^{2}\nonumber \\
		&\leq \left\|\uh{n}\right\|_{L^{\infty}(\Omega)}^{2}+ \frac{1}{\min_{S^{d-1}}\sigma^{2}}\int_{\Omega}f(-\nabla \uh{n})dx\nonumber\\
		&\leq \left\|\uh{n}\right\|_{L^{\infty}(\Omega)}^{2}+ \frac{2}{\eps\min_{S^{d-1}}\sigma^{2}}E_{\eps}[\uh{n}]\nonumber\\
		&\leq \max\left\{\left\|\uh{0}\right\|_{L^{\infty}(\Omega)}^{2},1\right\}+ \frac{2}{\eps\min_{S^{d-1}}\sigma^{2}}E_{\eps}[\uh{0}],
		\end{align}
		where the last inequality uses Lemma \ref{mmscheme-dmcov}(iv) and (\ref{aac-mm}) repeatedly.
		
		For any $t \in [0,T]$, the function $u_{h}(t)$ is a convex combination of two minimizers $\uh{n-1}$, $\uh{n}$ of (\ref{aac-mm}). Applying the triangle inequality yields
		\begin{align}\label{mmscheme-h1}
		\sup_{t\in[0,T]}\limsup_{h \searrow 0}\left\|u_{h}(t)\right\|_{H^{1}(\Omega)}&\leq \left(\max\left\{\limsup_{h\searrow 0}\left\|\uh{n}\right\|_{L^{\infty}(\Omega)}^{2},1\right\}\vphantom{+\frac{2}{\eps \min_{S^{d-1}}\sigma^{2}}\limsup_{h\searrow 0}E_{\eps}[\uh{0}]}\right.\nonumber\\&\hspace{8.5pt}{\left.
		\vphantom{\max\left\{\limsup_{h\searrow 0}\left\|\uh{n}\right\|_{L^{\infty}(\Omega)}^{2},1\right\}}
		+\frac{2}{\eps \min_{S^{d-1}}\sigma^{2}}\limsup_{h\searrow 0}E_{\eps}[\uh{0}]\right)}^{\frac{1}{2}}\nonumber\\&<\infty.
		\end{align}
		
		It follows from Rellich's compact embedding theorem that any sequence $\{u_{h_{k}}(t)\}_{k \in \nat}$ with $t \in [0,T]$ and $\lim_{k\to \infty}h_{k}=0$ has a subsequence converging in $L^{2}(\Omega)$. 
		
		To prove equicontinuity in time, which is the second requirement for the Arzel\`{a}--Ascoli theorem, we will---as is standard with minimizing movement approximations---show the stronger statement that $u_{h} \in C\left([0,T];L^{2}(\Omega)\right)$ are $\frac{1}{2}$-H\"{o}lder continuous and the H\"{o}lder constants are uniformly bounded as $h \searrow 0$. Indeed, for $n \in \nat$, we use the definition of the metric $(\cdot,\cdot)_{\uh{n-1}}$ and the minimization property (\ref{aac-mm}) to show that
		\begin{align}\label{mmscheme-stepest}
		\left\|\uh{n}-\uh{n-1} \right\|_{L^{2}(\Omega)}^{2} &\leq \frac{1}{\eps c_{g}}\left\|\uh{n}-\uh{n-1}\right\|_{\uh{n-1}}^{2} \nonumber \\
		&= \frac{2h}{\eps c_{g}}\left(E_{\eps}[\uh{n}]+\frac{1}{2h}\left\|\uh{n}-\uh{n-1}\right\|_{\uh{n-1}}^{2} - E_{\eps}[\uh{n}] \right) \nonumber \\
		&\leq \frac{2h}{\eps c_{g}}\left(E_{\eps}[\uh{n-1}]-E_{\eps}[\uh{n}] \right).
		\end{align}
		
		Inequality (\ref{mmscheme-stepest}) allows us to bound the $L^{2}$-norms of the weak time derivates $\partial_{t}u_{h}$ as
		\begin{align}\label{mmscheme-timeder}
		\|\partial_{t}u_{h}\|_{L^{2}\left(0,T;L^{2}(\Omega)\right)}^{2}&\leq \sum_{n=1}^{\left\lceil \frac{T}{h} \right\rceil} \int_{(n-1)h}^{nh}\|\partial_{t}u_{h}(t) \|_{L^{2}(\Omega)}^{2}dt \nonumber \\
		&= \sum_{n=1}^{\left\lceil \frac{T}{h} \right\rceil} \int_{(n-1)h}^{nh} \frac{\|\uh{n}-\uh{n-1}\|_{L^{2}(\Omega)}^{2}}{h^{2}}dt\nonumber \\
		&\leq \sum_{n=1}^{\left\lceil \frac{T}{h} \right\rceil} \frac{2}{\eps c_{g}}\left(E_{\eps}[\uh{n-1}]-E_{\eps}[\uh{n}] \right) \nonumber \\
		&\leq \frac{2}{\eps c_{g}}E_{\eps}[\uh{0}].
		\end{align}
		
		The fundamental theorem of calculus for vector-valued functions together with the Cauchy--Schwarz inequality now yields
		\begin{align}\label{mmscheme-holder}
		\|u_{h}(t)-u_{h}(s)\|_{L^{2}(\Omega)}&\leq\int_{(s,t)}\|\partial_{t}u_{h}(r)\|_{L^{2}(\Omega)}dr \nonumber \\&\leq \|\partial_{t}u_{h}\|_{L^{2}\left(s,t;L^{2}(\Omega)\right)}\sqrt{|t-s|}\nonumber \\
		&\leq \sqrt{\frac{2E_{\eps}[\uh{0}]}{\eps c_{g}}}\sqrt{|t-s|}
		\end{align}
		for all $s,t \in [0,T]$.
		This is the one-dimensional case of Morrey's inequality. 
		
		By (\ref{mmscheme-holder}) and the assumption on the initial data, the H\"{o}lder constants are bounded as $h \searrow 0$. Families of functions with bounded H\"{o}lder constants are equicontinuous. These equicontinuity and precompactness statements allow us to apply the Arzel\`{a}--Ascoli theorem, which states that there exists a subsequence $h \searrow 0$ and a function $u \in C\left([0,T];L^{2}(\Omega)\right)$ such that $u_{h} \to u$ in $C\left([0,T];L^{2}(\Omega)\right)$. Furthermore, the H\"{o}lder continuity carries over to the limit function, i.e.,
		\begin{equation}\label{mmscheme-limholder}
		\|u(t)-u(s)\|_{L^{2}(\Omega)}\leq \sqrt{\frac{2\limsup_{h\searrow 0}E_{\eps}[\uh{0}]}{\eps c_{g}}}\sqrt{|t-s|} \qquad \text{for all }s,t \in [0,T].
		\end{equation}
		One can extract a further subsequence $\{u_{h_{k}} \}_{k\in\nat}$ such that $h_{k}\searrow 0$ and $u_{h_{k}}(x,t)\to u(x,t)$ almost everywhere in $\Omega\times(0,T)$. This pointwise convergence together with the triangle inequality yields
		\begin{align*}
		\|u-\tfrac12\|_{L^{\infty}(\Omega\times(0,T))}&\leq \limsup_{k\to \infty}\|u_{h_{k}}-\tfrac12\|_{L^{\infty}(\Omega\times(0,T))}\\
		&\leq \limsup_{h\searrow 0}\|u_{h}-\tfrac12\|_{L^{\infty}(\Omega\times(0,T))}\\
		&\leq \limsup_{h\searrow 0}\sup_{n\in\nat}\|\uh{n}-\tfrac12\|_{L^{\infty}(\Omega)}\\
		&\leq \max\left\{\limsup_{h\searrow 0}\|\uh{0}-\tfrac12\|_{L^{\infty}(\Omega)},\tfrac12\right\}.
		\end{align*}
		
		In a last step, we argue briefly that the piecewise constant interpolations $\overline{u}_{h}$ and $\underline{u}_{h}$ converge uniformly.
		
		Given $t \in [0,T]$, we find $n \in \nat$ such that $t \in [(n-1)h,nh)$. Then
		\begin{equation*}
		\|\overline{u}_{h}(t)-u_{h}(t)\|_{L^{2}(\Omega)}=\|u_{h}(nh)-u_{h}(t)\|_{L^{2}(\Omega)}\leq\sqrt{\frac{2E_{\eps}[\uh{0}]}{\eps c_{g}}}\sqrt{nh-t}\leq\sqrt{\frac{2E_{\eps}[\uh{0}]}{\eps c_{g}}}\sqrt{h},
		\end{equation*}
		and a similar argument shows that 
		\begin{equation*}
		\|\underline{u}_{h}(t)-u_{h}(t)\|_{L^{2}(\Omega)}\leq\sqrt{\frac{2E_{\eps}[\uh{0}]}{\eps c_{g}}}\sqrt{h}.
		\end{equation*}
		
		Combining this bound with the uniform convergence $u_{h}\to u$ in $C\left([0,T];L^{2}(\Omega)\right)$, we obtain
		\begin{align*}
		\limsup_{h\searrow 0}\sup_{t\in [0,T]}\max&\left\{\|\overline{u}_{h}(t)-u(t)\|_{L^{2}(\Omega)},\|\underline{u}_{h}(t)-u(t)\|_{L^{2}(\Omega)} \right\}\\
		&\leq \limsup_{h\searrow 0}\sup_{t \in [0,T]}\|u_{h}(t)-u(t)\|_{L^{2}(\Omega)}+ \lim_{h\searrow 0}\sqrt{\frac{2E_{\eps}[\uh{0}]}{\eps c_{g}}}\sqrt{h} \\
		&=0.
		\end{align*}

		\item It follows from (i) that $u_{h}\to u$ as $h \searrow 0$ in $L^{2}\left(0,T;L^{2}(\Omega)\right)\cong L^{2}(\Omega \times (0,T))$. Furthermore, by (\ref{mmscheme-timeder}), the time derivatives $\{\partial_{t}u_{h} \}$ are bounded in $L^{2}(\Omega \times (0,T))$ as $h \searrow 0$. By, e.g., \cite[Prop. 2.5(b)]{AmbrosioFuscoPallara}, one concludes that $u$ has a weak time derivative $\partial_{t}u \in L^{2}(\Omega\times(0,T))$, and that $\partial_{t}u_{h}\weak \partial_{t}u$ in $L^{2}(\Omega\times (0,T))$ as $h \searrow 0$.
		\item We will prove the desired convergence result for the piecewise constant interpolation $\overline{u}_{h}$ first. Integrating (\ref{mmscheme-constant-h1}) from $0$ to $T$ shows that the gradients $\{\nabla \overline{u}_{h} \}$ are bounded in $L^{2}(\Omega \times (0,T))^{d}$ as $h \searrow 0$. Just like in (ii), it then follows from the strong convergence $\overline{u}_{h}\to u$ in $L^{2}(\Omega\times(0,T))$ that the limit function $u$ has a weak gradient $\nabla u \in L^{2}(\Omega\times(0,T))^{d}$, and that $\nabla \overline{u}_{h}\weak \nabla u$ as $h \searrow 0$ in $L^{2}(\Omega\times(0,T))^{d}$.

		The next step is to upgrade this weak convergence statement for the gradients to strong convergence in $L^{2}$. The key ingredient for the argument will be the energy convergence
		\begin{equation}\label{mmscheme-energy-conv}
		E_{\eps}[\overline{u}_{h}(\cdot)]\weak E_{\eps}[u(\cdot)] \qquad\text{in } L^{1}(0,T).
		\end{equation}
		
		Let us first argue how (\ref{mmscheme-energy-conv}) implies strong convergence of the gradients. Since the Cahn--Hilliard energy $E_{\eps}$ is made up of a Dirichlet energy and a nonconvex term involving $W$, we will use the convergence result (\ref{mmscheme-energy-conv}) for the Cahn--Hilliard energies and a liminf inequality for the nonconvex part to derive a limsup inequality for the Dirichlet energy, which will allow us to prove the strong convergence statement.

		Similarly to the proof of Lemma \ref{mmscheme-dmcov}(iii), an application of 
		Fatou's lemma yields
		\begin{equation}\label{mmscheme-reaction-liminf}
		\liminf_{h\searrow 0}\volint W(\overline{u}_{h})dxdt
		\geq \volint W(u)dxdt.
		\end{equation}
		
		One can now test (\ref{mmscheme-energy-conv}) with a constant test function and combine this convergence statement with (\ref{mmscheme-reaction-liminf}), which yields
		\begin{align}\label{mmscheme-diffusion-limsup}
		\limsup_{h\searrow 0}\volint f(-\nabla \overline{u}_{h})dx dt&=\frac{2}{\eps}\lim_{h\searrow 0}\int_{0}^{T}E_{\eps}[\overline{u}_{h}(t)]dt-\frac{1}{\eps^{2}}\liminf_{h\searrow 0}\volint W(\overline{u}_{h})dxdt \nonumber \\
		&\leq\frac{2}{\eps}\int_{0}^{T}E_{\eps}[u(t)]dt-\frac{1}{\eps^{2}}\volint W(u)dxdt \nonumber \\
		&=\volint f(-\nabla u)dx dt.
		\end{align}
		
		The function $f$ is strongly convex, i.e., there exists a constant $c>0$ such that $f(p^{\prime})\geq f(p)+Df(p)\cdot(p^{\prime}-p)+\frac{c}{2}|p^{\prime}-p|^{2}$ for all $p, p^{\prime} \in \real^{d}$. Thus, if (\ref{mmscheme-diffusion-limsup}) holds true, then
		\begin{align*}
		\limsup_{h\searrow 0}\frac{c}{2}\volint& \left|\nabla \overline{u}_{h} - \nabla u\right|^{2}dx dt \\&\leq \limsup_{h \searrow 0}\left( \volint f(-\nabla \overline{u}_{h})dx dt- \volint f(-\nabla u)dxdt \vphantom{-\volint Df(-\nabla u)\cdot(\nabla u - \nabla u_{h}) dxdt}\right.\\
		&\hspace{8.5pt}\left. \vphantom{\volint f(-\nabla u_{h})dx dt- \volint f(-\nabla u)dxdt} -\volint Df(-\nabla u)\cdot(\nabla u - \nabla \overline{u}_{h})dx dt\right)\\
		&\leq \limsup_{h \searrow 0} \volint f(-\nabla \overline{u}_{h})dx dt - \volint f(-\nabla u)dxdt\\
		&\quad -\liminf_{h \searrow 0} \volint Df(-\nabla u)\cdot(\nabla u - \nabla \overline{u}_{h}) dxdt\\
		&\leq 0,
		\end{align*}
		where the last step also uses the fact that $\nabla u_{h}\weak \nabla u$ in $L^{2}(\Omega\times(0,T))^{d}$. This computation implies that $\nabla \overline{u}_{h}\to \nabla u$ strongly in $L^{2}(\Omega\times(0,T))$.

		In order to prove (\ref{mmscheme-energy-conv}), it suffices to consider nonnegative test functions \linebreak $\zeta \in L^{\infty}\left(0,T;\real_{\geq 0}\right)$. We will prove a liminf inequality and a limsup inequality separately.
		
		On the one hand, by Fatou's lemma, the lower semicontinuity of $E_{\eps}$, and (i), we obtain
		\begin{equation}\label{mmscheme-energy-liminf}
		\liminf_{h\searrow 0}\int_{0}^{T}\zeta(t)E_{\eps}[\overline{u}_{h}(t)]dt\geq \int_{0}^{T}\zeta(t)\liminf_{h\searrow 0}E_{\eps}[\overline{u}_{h}(t)]dt\geq \int_{0}^{T}\zeta(t)E_{\eps}[u(t)]dt.
		\end{equation}
		
		On the other hand, choosing $w(x,t):=u(x,t)$ in (\ref{mmscheme-conv-ineq-2}) and integrating against $\zeta$ yields
		\begin{align*}
		\int_{0}^{T}\zeta(t)E_{\eps}[u(t)]dt& \geq\int_{0}^{T}\zeta(t)E_{\eps}[\overline{u}_{h}(t)]dt\\
		& \quad-\int_{0}^{T}\zeta(t)\int_{\Omega}\left(\eps g(-\nabla \underline{u}_{h})\partial_{t}u_{h}(u-\overline{u}_{h}) +\frac{\lambda}{4\eps}|u-\overline{u}_{h}|^{2}\right)dx dt.
		\end{align*}
		
		In the limit $h\searrow 0$, the second integral on the right-hand side vanishes since $\overline{u}_{h}\to u$ in $L^{2}(\Omega\times(0,T))$. Thus,
		\begin{equation}\label{mmscheme-energy-limsup}
		\int_{0}^{T}\zeta(t)E_{\eps}[u(t)]dt\geq \limsup_{h\searrow 0}\int_{0}^{T}\zeta(t)E_{\eps}[\overline{u}_{h}(t)]dt.
		\end{equation}
		Combining (\ref{mmscheme-energy-liminf}) and (\ref{mmscheme-energy-limsup}) yields (\ref{mmscheme-energy-conv}).

		We have shown that $\nabla \overline{u}_{h}\to \nabla u$ in $L^{2}(\Omega\times (0,T))$ as $h \searrow 0$. As for $\nabla \underline{u}_{h}$, one computes 		
		\begin{align*}
		\|\nabla\underline{u}_{h}-&\nabla u\|_{L^{2}(\Omega\times(0,T))}^{2}\\
		&=\int_{0}^{h}\int_{\Omega}|\nabla\underline{u}_{h}-\nabla u|^{2}dx dt + \int_{h}^{T}\int_{\Omega}|\nabla\underline{u}_{h}-\nabla u|^{2}dx dt \\
		&=\int_{0}^{h}\int_{\Omega}|\nabla\uh{0}-\nabla u|^{2}dx dt + \int_{h}^{T}\int_{\Omega}|\nabla\overline{u}_{h}(x,t-h)-\nabla u(x,t)|^{2}dx dt\\
		&\leq 2h\int_{\Omega}|\nabla \uh{0}|^{2}dx + 2 \int_{0}^{h}\int_{\Omega}|\nabla u|^{2}dxdt\\
		&\quad+2\int_{0}^{T-h}\int_{\Omega}|\nabla \overline{u}_{h}(x,t)-\nabla u(x,t)|^{2}dx dt\\
		&\quad + 2\int_{0}^{T-h}\int_{\Omega}|\nabla u(x,t)-\nabla u(x, t+h)|^{2}dx dt\\
		&\longrightarrow 0 \qquad \text{as } h \searrow 0,
		\end{align*}
		where the last integral vanishes as $h \searrow 0$ due to the continuity of translation.

		Lastly, $\nabla u_{h}(x,t)$ is a convex combination of $\nabla \overline{u}_{h}(x,t)$ and $\nabla\underline{u}_{h}(x,t)$ for almost all $(x,t)\in \Omega\times (0,T)$. Hence, it follows from $\nabla \overline{u}_{h}\to \nabla u$ and $\nabla \underline{u}_{h}\to \nabla u$ that, as $h\searrow 0$, we have $\nabla u_{h}\to \nabla u$ in $L^{2}(\Omega\times (0,T))$ as well.
	\end{enumerate}
\end{proof}

We are now ready to pass to the limit $h\searrow 0$ in (\ref{mmscheme-conv-ineq-2}), which will eventually allow us to prove Theorem \ref{aac-ex}. This is the content of the following lemma: 

\begin{lem}\label{mmscheme-convlimit}
	Let $w \in L^{2}(\Omega\times (0,T))$ and $\zeta \in L^{\infty}(0,T)$ such that $\zeta \geq 0$ almost everywhere. Then
	\begin{equation}\label{mmscheme-conv-ineq-3}
	\int_{0}^{T}\zeta(t)\left(E_{\eps}[w(t)]-E_{\eps}[u(t)]+\int_{\Omega}\eps g(-\nabla u)\partial_{t}u(w-u)dx + \frac{\lambda}{4\eps}\int_{\Omega}|w-u|^{2}dx\right)dt\geq 0.
	\end{equation}
\end{lem}

In other words, for every $w \in L^{2}(\Omega\times (0,T))$, the inequality $E_{\eps}[w(t)]\geq E_{\eps}[u(t)]-\int_{\Omega}\eps g(-\nabla u)\partial_{t}u(w-u)dx - \frac{\lambda}{4\eps}\int_{\Omega}|w-u|^{2}dx$ holds true for almost every $t \in (0,T)$.

\begin{proof}
First, by the weak convergence (\ref{mmscheme-energy-conv}), it follows that
\begin{equation}\label{mmscheme-limit-1}
\lim_{h\searrow 0}\int_{0}^{T}\zeta(t)E_{\eps}[\overline{u}_{h}(t)]dt = \int_{0}^{T}\zeta(t)E_{\eps}[u(t)]dt
\end{equation}
for any function $\zeta \in L^{\infty}(0,T)$ such that $\zeta \geq 0$.

Second, for the term involving the time derivative, we know from Lemma \ref{mmscheme-compact}(ii) that $\partial_{t}u_{h}\weak \partial_{t}u$ in $L^{2}(\Omega\times (0,T))$. Thus, to derive the convergence of the integrals, it suffices to show that $g(-\nabla \underline{u}_{h})(w-\overline{u}_{h}) \to g(-\nabla u)(w-u)$ (strongly) in $L^{2}(\Omega\times (0,T))$.

By Lemma \ref{mmscheme-compact}(iii), we know that $\nabla \overline{u}_{h},\nabla \underline{u}_{h}\to \nabla u$ in $L^{2}(\Omega\times(0,T))$. Since $g$ is a continuous function, every subsequence of $g(-\nabla \underline{u}_{h})$ as $h\searrow 0$ has a further subsequence that converges pointwise almost everywhere to $g(-\nabla u)$. By the uniform boundedness of $g$ and the dominated convergence theorem, it follows that
\begin{equation*}
g(-\nabla \underline{u}_{h}(t))(w(t)-\overline{u}_{h}(t)) \longrightarrow g(-\nabla u(t))(w(t)-u(t)) \qquad \text{in } L^{2}(\Omega) \text{ for almost all } t\in (0,T).
\end{equation*}

It remains to show the $L^{2}$-convergence on the product space $\Omega\times(0,T)$. Using the generalized dominated convergence theorem with
\begin{align*}
\left\|g(-\nabla \underline{u}_{h}(t))(w(t)-\overline{u}_{h}(t)) \right\|_{L^{2}(\Omega)}&\leq (\sup_{\real^{d}}g) \left(\int_{\Omega}(w(t)-\overline{u}_{h}(t))^{2}dx\right)^{\frac{1}{2}} \\&\longrightarrow (\sup_{\real^{d}}g) \left(\int_{\Omega}(w(t)-u(t))^{2}dx\right)^{\frac{1}{2}} \qquad \text{in } L^{2}(0,T),
\end{align*}
it follows that 
\begin{equation*}
\int_{0}^{T}\left\|g(-\nabla \underline{u}_{h}(t))(w(t)-\overline{u}_{h}(t))-g(-\nabla u(t))(w(t)-u(t)) \right\|_{L^{2}(\Omega)}^{2}dt\longrightarrow 0,
\end{equation*}
i.e., by Fubini's theorem, $g(-\nabla \underline{u}_{h})(w-\overline{u}_{h}) \to g(-\nabla u)(w-u)$ in $L^{2}(\Omega\times (0,T))$. In total, we obtain
\begin{align}\label{mmscheme-limit-2}
\lim_{h\searrow 0}\int_{0}^{T}\zeta(t)&\int_{\Omega}\eps g(-\nabla \underline{u}_{h}(t))\partial_{t}u_{h}(t)(w(t)-\overline{u}_{h}(t))dx dt\nonumber \\
& =\int_{0}^{T}\zeta(t)\int_{\Omega}\eps g(-\nabla u(t))\partial_{t}u(t)(w(t)-u(t))dx dt.
\end{align}

Third, it follows from Lemma \ref{mmscheme-compact}(i) that
\begin{equation}\label{mmscheme-limit-3}
\lim_{h\searrow 0}\frac{\lambda}{4\eps}\int_{0}^{T}\zeta(t)\int_{\Omega}|w(t)-\overline{u}_{h}(t)|^{2}dxdt = \frac{\lambda}{4\eps}\int_{0}^{T}\zeta(t)\int_{\Omega}|w(t)-u(t)|^{2}dxdt.
\end{equation}

Integrating (\ref{mmscheme-conv-ineq-2}) against $\zeta$, taking the limit $h \searrow 0$, and plugging in (\ref{mmscheme-limit-1})--(\ref{mmscheme-limit-3}) proves the desired inequality.
\end{proof}

With the help of Lemma \ref{mmscheme-convlimit}, it can be seen that $-\partial_{t}u(t)$ lies in the $u(t)$-subdifferential $\partial_{u(t)} E_{\eps}[u(t)]$ for almost every $t \in (0,T)$. Thus, it seems plausible that the limiting trajectory $u(t)$ is a gradient flow for $E_{\eps}$.

\begin{proof}[Proof of Theorem \ref{aac-ex}]
Choosing $\uh{0}=u_{0}$, one can construct the functions $u_{h}$, $\overline{u}_{h}$, and $\underline{u}_{h}$ by Lemma \ref{mmscheme-dmcov} and (\ref{mmscheme-interpol}). Invoking Lemma \ref{mmscheme-compact} yields a subsequence $h \searrow 0$ as well as a function $u \in H^{1}(\Omega\times (0,T))\cap L^{\infty}(\Omega\times(0,T))$ such that $\|u-\tfrac12\|_{L^{\infty}(\Omega\times(0,T))}\leq\max\left\{\|u_{0}-\tfrac12\|_{L^{\infty}(\Omega)},\tfrac12 \right\}$, and such that the three convergence results in Lemma \ref{mmscheme-compact}(i)--(iii) hold true. It remains to be shown that $u$ is a weak solution to the anisotropic Allen--Cahn equation with initial data $u_{0}$ in the sense of Definition \ref{aac-solution}.

As for the initial condition, it follows from the uniform convergence in Lemma \ref{mmscheme-compact}(i) and the choice $\uh{0}=u_{0}$ that 
\begin{equation*}
u(0)=\lim_{h\searrow 0}u_{h}(0)=\lim_{h\searrow 0}\uh{0}=u_{0},
\end{equation*}
where all limits are in $L^{2}(\Omega)$.

For the optimal energy dissipation relation (\ref{aac-oed}) let $\horiz \in (0,T]$ be a fixed time horizon. We define an extension of $u:\Omega\times [0,\horiz]\to \real$ to $\Omega \times \real$ by
\begin{equation*}
\tilde{u}(x,t)=\begin{cases}
u(x,t)&\text{if } 0\leq t \leq \horiz,\\
u(x,0)&\text{if } t < 0,\\
u(x,\horiz)&\text{if } t> \horiz.
\end{cases}
\end{equation*}

Then $\tilde{u} \in H_{\loc}^{1}(\Omega\times \real)$, with the weak time derivative being given by $\partial_{t}\tilde{u}(x,t)=\chi_{\Omega\times (0,\horiz)}\partial_{t}u(x,t)$.
Taking $a>0$, $\zeta:=\frac{1}{a}\chi_{(0,\horiz)}$, and $w(x,t):=\tilde{u}(x, t+a)$ in (\ref{mmscheme-conv-ineq-3}) yields
\begin{align}\label{mmscheme-oed-ineq-1}
\int_{0}^{\horiz}\int_{\Omega}\eps g(-\nabla u(t))\partial_{t}u(t)& \frac{\tilde{u}(t+a)-\tilde{u}(t)}{a}dx dt + a\int_{0}^{\horiz}\int_{\Omega}\frac{\lambda}{4\eps}\left(\frac{\tilde{u}(t+a)-\tilde{u}(t)}{a}\right)^{2}dxdt\nonumber\\&\geq-\int_{0}^{\horiz}\frac{E_{\eps}[\tilde{u}(t+a)]-E_{\eps}[\tilde{u}(t)]}{a}dt\nonumber\\
&=\frac{1}{a}\int_{0}^{a}E_{\eps}[\tilde{u}(t)]dt - \frac{1}{a}\int_{\horiz}^{\horiz+a}E_{\eps}[\tilde{u}(t)]dt\nonumber \\
&=\frac{1}{a}\int_{0}^{a}E_{\eps}[u(t)]dt - E_{\eps}[u(\horiz)],
\end{align}
where we have replaced $u$ by $\tilde{u}$ on $\Omega\times(0,\horiz)$ several times and used the definition of $\tilde{u}$ in the last line. 

Let us now take $a \searrow 0$ in the inequality above. As a consequence of the lower semicontinuity of $E_{\eps}$ and the $L^{2}$-continuity of the map $t\mapsto u(t)$, we have \begin{equation*}\liminf_{a \searrow 0} \frac{1}{a}\int_{0}^{a}E_{\eps}[u(t)]dt \geq E_{\eps}[u(0)]=E_{\eps}[u_{0}].
\end{equation*}
Furthermore, it follows from a general fact about Sobolev functions that \begin{equation*}\frac{\tilde{u}(\cdot+a)-\tilde{u}(\cdot)}{a}\weak \partial_{t}\tilde{u}(=\partial_{t}u)\qquad \text{in } L^{2}(\Omega\times(0,T))\text{ as }a \searrow 0\end{equation*} (see \cite[Lemma 7.23 and proof of Lemma 7.24]{GilbargTrudinger}). In particular, the integral\newline $\int_{0}^{\horiz}\int_{\Omega}\frac{\lambda}{2\eps}\left(\frac{\tilde{u}(t+a)-\tilde{u}(t)}{a}\right)^{2}dxdt$ is bounded as $a \searrow 0$.

Therefore, the limiting inequality of (\ref{mmscheme-oed-ineq-1}) as $a \searrow 0$ reads
\begin{equation}\label{mmscheme-oed-1}
\int_{0}^{\horiz}\int_{\Omega}\eps g(-\nabla u(t))(\partial_{t}u(t))^{2}dx dt \geq E_{\eps}[u_{0}]-E_{\eps}[u(\horiz)],
\end{equation}
which is one inequality in the optimal energy dissipation relation (\ref{aac-oed}).

For the converse inequality we proceed analogously by taking $w(x,t):=\tilde{u}(x,t-a)$ and $a \searrow 0$. Using firstly the fact that $\frac{\tilde{u}(\cdot-a)-\tilde{u}(\cdot)}{a}\weak -\partial_{t}u$ in $L^{2}(\Omega\times(0,\horiz))$, secondly the inequality $\liminf_{a \searrow 0}\frac{1}{a}\int_{\horiz - a}^{\horiz}E_{\eps}[u(t)]dt\geq E_{\eps}[u(\horiz)]$, we obtain 
\begin{equation}\label{mmscheme-oed-2}
\int_{0}^{\horiz}\int_{\Omega}\eps g(-\nabla u(t))(\partial_{t}u(t))^{2}dx dt \leq E_{\eps}[u_{0}]-E_{\eps}[u(\horiz)].
\end{equation}
The combination of (\ref{mmscheme-oed-1}) and (\ref{mmscheme-oed-2}) is precisely the optimal energy dissipation identity (\ref{aac-oed}).

For the distributional formulation of the PDE (\ref{aac-df}) let $\varphi\in C^{1}(\Omega\times[0,T])$ and $s >0$. By plugging in $\zeta \equiv 1$ and $w=u+s\varphi$ into (\ref{mmscheme-conv-ineq-3}), we obtain
\begin{align*}
\int_{0}^{T}&\left(E_{\eps}[u(t)+s\varphi(t)]- E_{\eps}[u(t)] + \int_{\Omega}\eps g(-\nabla u(t))\partial_{t}u(t)s\varphi(t)dx \vphantom{ + \frac{\lambda}{4\eps}\int_{\Omega}s^{2}\varphi(t)^{2}dx}\right.\\&\qquad\left.\vphantom{ E_{\eps}[u(t)+s\varphi(t)]- E_{\eps}[u(t)] + \int_{\Omega}\eps g(-\nabla u(t))\partial_{t}u(t)s\varphi(t)\, dx} + \frac{\lambda}{4\eps}\int_{\Omega}s^{2}\varphi(t)^{2}dx\right)dt\geq0.
\end{align*}
Dividing by $s$ and taking $s \searrow 0$ leads to
\begin{align*}
\int_{0}^{T}\int_{\Omega}\eps g(-\nabla u)&\partial_{t}u \varphi\, dx dt \geq -\liminf_{s\searrow 0}\int_{0}^{T}\frac{E_{\eps}[u(t)+s\varphi(t)]- E_{\eps}[u(t)]}{s}dt\\
&=-\frac{1}{2}\liminf_{s \searrow 0} \left(\eps\int_{0}^{T}\int_{\Omega}\frac{f(-\nabla u - s\nabla \varphi)-f(-\nabla u)}{s}dx dt\vphantom{ + \frac{1}{\eps}\int_{0}^{T}\int_{\Omega}\frac{W(u+s\varphi)-W(u)}{s}dx dt}\right. \\
&\hspace{8.5pt}\left.\vphantom{\eps\int_{0}^{T}\int_{\Omega}\frac{f(-\nabla u - s\nabla \varphi)-f(-\nabla u)}{s}dx dt} + \frac{1}{\eps}\int_{0}^{T}\int_{\Omega}\frac{W(u+s\varphi)-W(u)}{s}dx dt\right)\\
&=\frac{\eps}{2}\int_{0}^{T}\int_{\Omega}Df(-\nabla u)\cdot \nabla \varphi\, dx dt-\frac{1}{2\eps}\int_{0}^{T}\int_{\Omega}W^{\prime}(u)\varphi\, dxdt,
\end{align*}
where the limit and the integrals can be interchanged in the last step because $u$ is essentially bounded. 

We divide by $\eps$ to obtain one inequality in (\ref{aac-df}). The converse inequality follows by taking $s < 0$, $s \nearrow 0$.
\end{proof}

\subsection{Regularity of weak solutions}

The remainder of this section is devoted to proving a regularity result which states that solutions to the anisotropic Allen--Cahn equation have weak second derivatives in space. We apply a difference quotient method for elliptic regularity. The idea for this proof is taken from \cite{LeoniBook}. The fact that $f$ is, in general, not twice continuously differentiable on $\real^{d}$ due to a singularity at $0$ will not pose a problem when proving the existence of second derivatives, but it prevents us from deriving an additional PDE for the second derivatives.

\begin{thm}[Spatial regularity of distributional solutions]\label{aac-regular}
	Let $\eps > 0$. If $u$ is a solution to (\ref{aac}) in the sense of Definition \ref{aac-solution}, then $u \in L^{2}\left(0,T;H^{2}(\Omega)\right)$.
\end{thm}

It suffices to prove the following

\begin{clm*}Let $u \in H^{1}(\Omega)$, and suppose that $u$ is a weak solution to $-\dvg(Df(-\nabla u))=h$ for some function $h \in L^{2}(\Omega)$. More precisely, we assume that 
\begin{equation*}
\domint Df(-\nabla u)\cdot \nabla w \,dx = \domint h w\, dx
\end{equation*}
for all $w \in H^{1}(\Omega)$. Then $u$ has a weak second derivative in space $D^{2}u \in L^{2}(\Omega)^{d\times d}$ and $\|D^2 u\|_{L^2(\Omega)} \leq C \|h\|_{L^2(\Omega)}$.
\end{clm*}

Theorem \ref{aac-regular} follows from this claim by choosing $h=2g(-\nabla u)\partial_{t}u+\frac{1}{\eps^{2}}W^{\prime}(u)$ and slicing in time. Then we have $h(\cdot,t) \in L^{2}(\Omega)$ for almost every $t \in (0,T)$ since $\partial_{t}u \in L^{2}(\Omega\times(0,T))$, $u\in L^{\infty}(\Omega\times(0,T))$, and $W \in C^{1}(\real)$. The fact that (\ref{aac-df}) holds true for test functions $w \in L^{2}\left(0,T;H^{1}(\Omega)\right)$ instead of just $C^{1}(\Omega\times[0,T])$ can be shown with a density argument.

\begin{proof}[Proof of the claim]

	Given a function $w$ as in the claim, we define the difference quotient \begin{equation*}D_{l}^{s}w(x):=\frac{w(x+se_{l})-w(x)}{s}
	\end{equation*}
	for $s\neq 0$ and $l=1,\ldots, d$. Then $D_{l}^{s}w \in L^{2}\left(0,T;H^{1}(\Omega)\right)$. The weak formulation of the PDE and the substitution $y=x+se_{l}$ yield
	\begin{align}\label{reg-test}
	\domint h D_{l}^{s}w\, dx  &= \domint Df(-\nabla u)\cdot \nabla D_{l}^{s}w \,dx \nonumber\\
	&= \domint Df(-\nabla u(x,t))\cdot \frac{\nabla w(x+se_{l})-\nabla w(x)}{s}\,dx \nonumber\\
	&= \domint \frac{Df(-\nabla u(x-se_{l}))-Df(-\nabla u(x))}{s} \cdot \nabla w(x)\,dx .
	\end{align}
	
	In order to deal with the right-hand side, we will use the following identity for $i=1,2, \ldots, d$ and for almost every $x\in \Omega$:
	\begin{align}\label{reg-ftc}
	\partial_{\xi_{i}}f(-\nabla u(x-se_{l}))\,{-}\,&\partial_{\xi_{i}}f(-\nabla u(x))\nonumber \\&=-\sum_{j=1}^{d}\int_{0}^{1}\partial_{\xi_{i}\xi_{j}}^{2}f\left(-r\nabla u(x-se_{l})-(1-r)\nabla u(x)\right)dr\nonumber \\
	&\quad\quad\quad\quad \cdot\left(\partial_{j}u(x-se_{l})-\partial_{j}u(x)\right)\nonumber \\
	&=\sum_{j=1}^{d}A_{ij}^{s}(x)\left(\partial_{j}u(x)-\partial_{j}u(x-se_{l}) \right), 
	\end{align}
	where the shorthand notation in the last line is to be read as  \begin{equation*}A_{ij}^{s}(x):=\int_{0}^{1}\partial_{\xi_{i}\xi_{j}}^{2}f\left(-r\nabla u(x-se_{l})-(1-r)\nabla u(x)\right)dr.
	\end{equation*}
	
	If the line segment between $\nabla u(x-se_{l})$ and $\nabla u(x)$ does not contain the origin, then $f$ is twice continuously differentiable in a neighborhood of the line segment, so that (\ref{reg-ftc}) follows immediately from the fundamental theorem of calculus. On the other hand, if $0 \in \conv\left(\left\{\nabla u(x-se_{l}),\nabla u(x,t)\right\}\right)$, then (\ref{reg-ftc}) can be deduced by decomposing the line segment into two parts and applying the fundamental theorem of calculus on each part.

	By the strong convexity of $f$, there exists a constant $c>0$ such that
	\begin{equation*}
	\xi \cdot D^{2}f(p)\xi \geq c|\xi|^{2} 
	\end{equation*}
	for all $\xi\in \real^{d}$, $p \in \real^{d}\setminus\{0\}$.
	Using the definition of $A_{ij}^{s}$, we obtain
	
	 \begin{align}\label{reg-rhsbound}
	\nabla D^{-s}_l u \cdot A^s(x) \nabla D^{-s}_l u\geq c|\nabla D^{-s}_l u|^{2}
	\end{align}
	for almost every $x\in \Omega$.
	
	One can now choose $w = D_{l}^{-s}u$ in (\ref{reg-test}) and use (\ref{reg-ftc}) with (\ref{reg-rhsbound}) as is standard (see, e.g., \cite[Lemma 7.23]{GilbargTrudinger}) to conclude
	\begin{equation}\nonumber
	\int_{\Omega}\left|\nabla D_{l}^{-s}u\right|^{2}dx \leq C\domint h^{2}\, dx  < \infty.
	\end{equation}
	As $\nabla D_{l}^{-s}u \rightharpoonup \partial_{l}\nabla u$ as $s\to 0$, this concludes the claim.
\end{proof}

\section[Convergence of the anisotropic Allen--Cahn equation to anisotropic mean curvature flow]{Convergence of the anisotropic Allen--Cahn equation to anisotropic mean curvature flow}\label{sec:aacToamcf}

The goal of this section is to prove a conditional convergence result for the sharp-interface limit of the anisotropic Allen--Cahn equation. In order to conclude that the limit is a $BV$ solution to (\ref{strong-amcf}) in the sense of Definition \ref{amcf-sol}, a critical assumption is the convergence of the time-integrated energies. A statement of the result and a sketch of the proof are included in the notes \cite{LauxKyoto}. Here, we complete the proof and correct a mistake with respect to the phase-field equation (\ref{aac}): In \cite{LauxKyoto}, the function $g$ was defined as $g(p)=\frac{|p|}{\mu(p)}$, thereby missing the effect of the anisotropic surface tension $\sigma$ on $g$.

\begin{thm}\label{sil-thm}
	Let $(\sigma, \mu)$ be an admissible pair of anisotropies, and let $f$, $g$ be given by equations (\ref{aac-f}), (\ref{aac-g}), respectively. Suppose that $W$ is a double-well potential satisfying \textbf{(W1)}--\textbf{(W5)}. Let $\{u_{\eps}\}_{\eps > 0}$ be solutions to the anisotropic Allen--Cahn equation in the sense of Definition \ref{aac-solution} with initial data $u_{\eps,0} \in \dom(E_{\eps})$ and such that $\|u_{\eps}-\tfrac12\|_{L^{\infty}(\Omega\times(0,T))}\leq \max\{\|u_{\eps,0}-\tfrac12\|_{L^{\infty}(\Omega)},\tfrac12 \}$. Assume that 
	\begin{itemize}
		\item
		the initial conditions are well-prepared,
		\begin{equation}\label{init-conv}
		u_{\eps,0}\longrightarrow u_{0} \qquad \text{in } L^{2}(\Omega) \qquad \text{and} \qquad E_{\eps}[u_{\eps,0}]\longrightarrow E[u_{0}]=:E_{0}< \infty
		\end{equation}
		as $\eps \searrow 0$, and that
		\item 
		there is a uniform $L^{\infty}$-bound on the initial conditions,
		\begin{equation}\label{uniform-linfty-bd}
		\limsup_{\eps \searrow 0}\left\|u_{\eps,0}\right\|_{L^{\infty}(\Omega)}<\infty.
		\end{equation}
	\end{itemize}
	Then there exists a subsequence $\eps \searrow 0$ as well as a function \begin{equation*}u=\chi \in BV\left(\Omega\times(0,T);\left\{-1,1\right\}\right)\cap C^{0,\frac{1}{2}}\left([0,T];L^{1}(\Omega)\right)\cap L^{\infty}\left(0,T;BV(\Omega)\right)
	\end{equation*} such that $u_{\eps} \to u$ in $L^{2}(\Omega\times (0,T))$ as $\eps \searrow 0$.
	
	If, furthermore,
	\begin{equation}\label{energy-conv}
	\lim_{\eps \searrow 0}\int_{0}^{T}E_{\eps}[u_{\eps}(t)]dt= \int_{0}^{T}E[u(t)]dt,
	\end{equation}
	then $u$ is a distributional solution to anisotropic mean curvature flow with initial condition $u(0)=u_{0}$ in the sense of Definition \ref{amcf-sol}.
\end{thm}

Similarly to \cite{LuckhausSturzenhecker}, where a minimizing movements construction for isotropic mean curvature is performed, the intuitive meaning of the assumption of energy convergence (\ref{energy-conv}) is that no surface area is lost in the limit.

We will first argue by compactness that a limiting function $u$ exists. Assumption (\ref{energy-conv}) then allows us to prove an equipartition of energy between the anisotropic Dirichlet energy and the nonconvex potential energy as $\eps \searrow 0$. In addition, this assumption guarantees the existence of a normal velocity $V$ in the sense of (\ref{velocity-criterion}). We introduce a relative entropy functional which serves as a tilt excess. Using this tilt excess to bound the occurring error terms, we will derive the weak formulation (\ref{amcf-df}) from the distributional formulation (\ref{aac-df}) and the optimal energy dissipation inequality (\ref{amcf-oed}) from (\ref{aac-oed}).

\subsection{Compactness}

To prove the compactness statement for the solutions $\{u_{\eps} \}_{\eps > 0}$ to the anisotropic Allen--Cahn equation from Theorem \ref{aac-ex}, we will first show a $W^{1,1}$-bound for the compositions of a suitable continuous function with $u_{\eps}$. This argument was performed, for example, by Fonseca and Tartar \cite{FonsecaTartar} in the isotropic case (see also \cite{LeoniMM}). However, some simplifications are possible since we assume an $L^{\infty}$-bound on the functions $\{u_{\eps} \}$ instead of prescribing growth conditions on the double-well potential $W$. 

\begin{lem}\label{compact-lem}
	Let $u_\eps$ be solutions of the anisotropic Allen--Cahn equation as in Theorem \ref{aac-ex}. Under the assumptions of Theorem \ref{sil-thm} except (\ref{energy-conv}), there exist a subsequence $\eps \searrow 0$ and a function $u=\chi \in BV\left(\Omega\times(0,T);\left\{0,1\right\}\right)\cap C^{0,\frac{1}{2}}\left([0,T];L^{1}(\Omega)\right)\cap L^{\infty}\left(0,T;BV(\Omega)\right)$  such that
	\begin{itemize}
		\item $u_{\eps}\to u$ in $L^{2}(\Omega\times(0,T))$,
		\item $\phi \circ u_{\eps} \weakstar \phi \circ u = \cnght \chi$ in $BV(\Omega\times(0,T))$, and
		\item $\phi \circ u_{\eps}(t) \weakstar \cnght \chi(t)$ in $BV(\Omega)$ for all $t\in [0,T]$
	\end{itemize}
	as $\eps \searrow 0$, where the function $\phi:\real\to \real$ is defined via 
	\begin{equation}\label{phi}
	\phi(z):=\int_{0}^{z}\sqrt{W(u)}du.
	\end{equation}
\end{lem}

The first step of the proof is to show that
\begin{equation}\label{w11-on-product-space}
\limsup_{\eps \searrow 0}\left\|\phi\circ u_{\eps}\right\|_{W^{1,1}(\Omega\times(0,T))}<\infty.
\end{equation}

For this step, we observe that, by assumption (\ref{uniform-linfty-bd}), there exists a finite constant $R > 0$ such that $u_{\eps}(x,t)\in [-R,R]$ for almost every $(x,t)\in \Omega\times(0,T)$ and for all sufficiently small $\eps > 0$. Using the continuity of $\phi$, we first estimate
\begin{equation}\label{compact-function}
\volint \left|\phi(u_{\eps})\right|dx dt \leq T \sup_{z \in [-R,R]}\left|\phi(z)\right| < \infty
\end{equation}
for $\eps > 0$ small enough.

Furthermore, the chain rule for Sobolev functions is applicable to the compositions $\phi \circ u_{\eps}$ even though $\phi$ need not be globally Lipschitz continuous: By virtue of its definition as a primitive function, $\phi$ is continuously differentiable. In particular, $\phi^{\prime}\big|_{[-R,R]}$ is bounded, and we can assume without restriction that $\phi^{\prime}$ is bounded. Thus, one can estimate 
\begin{align}\label{compact-gradient}
\volint \left|\nabla(\phi \circ u_{\eps})\right|dx dt&= \volint \phi^{\prime}( u_{\eps})\left| \nabla u_{\eps}\right|dx dt \nonumber \\
&\leq \frac{1}{\min_{S^{d-1}}\sigma}\volint \phi^{\prime}( u_{\eps})\sigma(- \nabla u_{\eps})dx dt \nonumber \\
&\leq \frac{1}{2\min_{S^{d-1}}\sigma}\volint \left(\eps\sigma(- \nabla u_{\eps})^{2}+\frac{1}{\eps}\phi^{\prime}( u_{\eps})^{2}\right)dx dt \nonumber \\
&= \frac{1}{2\min_{S^{d-1}}\sigma}\volint \left(\eps f(- \nabla u_{\eps}) +\frac{1}{\eps}W( u_{\eps})\right)dx dt \nonumber \\
&= \frac{1}{\min_{S^{d-1}}\sigma}\int_{0}^{T} E_{\eps}[u_{\eps}(t)] dt
\end{align}
for $\eps > 0$ sufficiently small.

Lastly, a similar argument for small $\eps$ yields
\begin{align}\label{compact-timeder}
\volint \left| \partial_{t}(\phi\circ u_{\eps})\right| dx dt &= \volint \phi^{\prime}(u_{\eps})\left| \partial_{t} u_{\eps}\right| dx dt \nonumber \\
&\leq \left( \volint \frac{1}{\eps}W(u_{\eps}) dx dt\right)^{\frac{1}{2}}
\left(\int_{0}^{T}\int_{\Omega}\eps|\partial_{t} u_{\eps}|^{2} dx dt\right)^{\frac{1}{2}} \nonumber \\
&\leq \left(2\int_{0}^{T}E_{\eps}[u_{\eps}(t)]dt\right)^{\frac{1}{2}} \left( \frac{1}{ c_{g}} \volint\eps g(-\nabla u_{\eps})|\partial_{t} u_{\eps}|^{2} dx dt \right)^{\frac{1}{2}}\nonumber \\
&\leq \sqrt{\frac{2}{c_{g}}}E_{\eps}[u_{\eps,0}] \sqrt{T},
\end{align}
where the fourth inequality uses the optimal energy dissipation identity (\ref{aac-oed}) in the first and in the second factor.

It is known from assumption (\ref{init-conv}) that 
\begin{equation*}\limsup_{\eps \searrow 0} E_{\eps}[u_{\eps,0}] = E_{0} < \infty \quad \text{and} \quad \limsup_{\eps\searrow 0}\int_{0}^{T}E_{\eps}[u_{\eps}(t)]dt \leq \lim_{\eps\searrow 0}TE_{\eps}[u_{\eps,0}] =TE_{0} < \infty, \end{equation*}
where the second estimate makes use of the fact that $t\mapsto E_{\eps}[u_{\eps}(t)]$ is nonincreasing by (\ref{aac-oed}). 
Thus, the estimates (\ref{compact-function})--(\ref{compact-timeder}) suffice to prove (\ref{w11-on-product-space}).

In addition to (\ref{w11-on-product-space}), we want to derive a uniform $W^{1,1}$-estimate for fixed times $t\in (0,T)$, namely
\begin{equation}\label{w11-for-fixed-time}
\limsup_{\eps \searrow 0}\,\underset{t\in(0,T)}{\ess\sup}\left\|\phi\circ u_{\eps}(t)\right\|_{W^{1,1}(\Omega)}<\infty.
\end{equation}

Indeed, it follows by Fubini's theorem that $\| u_{\eps}(t)\|_{L^{\infty}}\leq R$ for almost every $t \in (0,T)$ if $\eps >0$ is small enough, and we can argue similarly as in (\ref{compact-function}) and (\ref{compact-gradient}) to show that \begin{equation*}
\int_{\Omega}\left|\phi\circ u_{\eps}(x,t)\right|dx\leq \sup_{z\in[-R,R]}|\phi(z)|<\infty
\end{equation*}
and 
\begin{equation*}\int_{\Omega}\left|\nabla(\phi\circ u_{\eps})(x,t)\right|dx\leq \frac{1}{\min_{S^{d-1}}\sigma}E_{\eps}[u_{\eps}(t)]\leq  \frac{1}{\min_{S^{d-1}}\sigma}E_{\eps}[u_{\eps,0}]<\infty
\end{equation*} for a.e. $t\in(0,T)$ and all $\eps >0$ sufficiently small. These two observations prove (\ref{w11-for-fixed-time}).

Bounded sequences in $BV(\Omega\times(0,T))$ admit a weakly-* convergent subsequence by a version of \cite[Theorem 3.23]{AmbrosioFuscoPallara}. Thus, by (\ref{w11-on-product-space}), we can find a subsequence $\eps \searrow 0$ such that $\phi\circ u_{\eps}$ converge weakly-* in $BV(\Omega\times(0,T))$ as $\eps \searrow 0$. Let us denote the limiting function by $v \in BV(\Omega\times(0,T))$. Then, in particular, $\phi \circ u_{\eps} \to v$ in $L^{1}(\Omega\times(0,T))$, and we can assume without restriction that 
\begin{equation}\label{pae-full}
\phi\circ u_{\eps}(x,t)\longrightarrow v(x,t)\qquad \text{for a.e. } (x,t)\in \Omega\times (0,T)
\end{equation}
and
\begin{equation}\label{pae-fubini}
\phi\circ u_{\eps}(t)\longrightarrow v(t)\qquad\text{in } L^{1}(\Omega)\qquad \text{for a.e. } t\in (0,T),
\end{equation}
which can be accomplished by taking a further subsequence. It follows from (\ref{w11-for-fixed-time}), (\ref{pae-fubini}), and the characterization \cite[Proposition 3.13]{AmbrosioFuscoPallara} of weak-* convergence in $BV$ that $\phi\circ u_{\eps}(t)\weakstar v(t)$ in $BV(\Omega)$ for a.e. $t\in(0,T)$.

From now on, the limit $\eps \searrow 0$ is to be understood as the limit along a subsequence $\{\eps_{j}\}_{j \in \nat}$ such that $\eps_{j} \searrow 0$ and $\phi \circ u_{\eps_{j}}\weakstar v$ in $BV(\Omega\times(0,T))$ as $j\to\infty$, and such that the pointwise convergence properties as in (\ref{pae-full}) and (\ref{pae-fubini}) hold true.

The following argument for the H\"{o}lder continuity $v \in C^{0,\frac{1}{2}}\left([0,T];L^{1}(\Omega)\right)$ is adapted from \cite[Lemma 2]{HenselLaux}, where Hensel and the first author deal with the H\"{o}lder continuity in the isotropic case: Let $0 \leq s \leq t \leq T$. If we proceed as in (\ref{compact-timeder}), but only integrate from $s$ to $t$, we find
\begin{align}\label{eps-holder}
\int_{\Omega}\left|\phi(u_{\eps}(x,t))-\phi(u_{\eps}(x,s)) \right|dx &\leq \int_{s}^{t}\int_{\Omega}\left|\partial_{t}\left(\phi\circ u_{\eps}\right) \right|dx dt\nonumber\\&\leq \sqrt{\frac{2}{c_{g}}}E_{\eps}[u_{\eps,0}] \sqrt{t-s}.
\end{align}
There exists a null set $N\subset (0,T)$ such that, for all $s,t \in (0,T)\setminus N$, we can pass to the limit $\eps \searrow 0$ to obtain
\begin{equation*}
\int_{\Omega}\left|v(x,t)-v(x,s)\right|dx \leq\sqrt{\frac{2}{c_{g}}}E_{0}\sqrt{t-s}.
\end{equation*}
This allows us to redefine $v$ on the null set $\Omega\times N$ so that $v\in C^{0,\frac{1}{2}}\left([0,T];L^{1}(\Omega)\right)$.

We identify $\phi\circ u_{\eps}$ with their H\"{o}lder continuous representatives due to (\ref{eps-holder}). If a sequence of uniformly H\"{o}lder continuous functions converges pointwise almost everywhere, it follows that the sequence converges pointwise. In particular, we can upgrade the $L^{1}(\Omega)$-convergence for a.e. $t\in (0,T)$ in (\ref{pae-fubini}) to $L^{1}(\Omega)$-convergence for \textit{all} $t\in [0,T]$.

Using the continuity $\phi \circ u_{\eps}\in C^{0,\frac{1}{2}}\left([0,T];L^{1}(\Omega)\right)$ and the lower semicontinuity of the variation \cite[Remark 3.5]{AmbrosioFuscoPallara}, we conclude from (\ref{w11-for-fixed-time}) that
\begin{equation*}
\limsup_{\eps \searrow 0}\sup_{t\in[0,T]}\left\|\phi\circ u_{\eps}(t)\right\|_{BV(\Omega)}<\infty,
\end{equation*}
i.e., we can also upgrade the essential boundedness of the $BV$-norm to uniform boundedness in time. From this, one finds that $\phi \circ u_{\eps}(t)\weakstar v(t)$ in $BV(\Omega)$ for \textit{all} $t \in [0,T]$, and that $v \in L^{\infty}\left(0,T;BV(\Omega)\right)$. (This statement includes the measurability of the measure-valued map $t \mapsto \nabla v(t)$ in the sense of \cite[Definition 2.25]{AmbrosioFuscoPallara}, which follows from the fact that we have $\sup_{t \in [0,T]}\int_{\Omega}|\nabla v(t)|<\infty$ and $v(t)\to v(t_{0})$ in $L^{1}(\Omega)$ as $t \to t_{0}$, and therefore $\nabla v(t)\weakstar \nabla v(t_{0})$ in $\rmeas(\Omega)$ as $t \to t_{0}$.)

Let us turn to the convergence result for the solutions $u_{\eps}$. As the function $\sqrt{W}$ is strictly positive except at two isolated points, the primitive function $\phi$ is strictly increasing. Therefore, there exists a continuous inverse $\phi^{-1}$, and we obtain from (\ref{pae-full}) that
\begin{equation*}
u_{\eps}(x,t)\longrightarrow u(x,t):=\phi^{-1}(v(x,t)) \qquad\text{as } \eps\searrow 0 \qquad \text{for almost all } (x, t)\in \Omega\times (0,T).
\end{equation*}
By (\ref{uniform-linfty-bd}) and the dominated convergence theorem, we may conclude $u_{\eps}\to u$ in $L^{2}(\Omega\times(0,T))$ as $\eps \searrow 0$.

To see that the limiting function $u$ takes values in $\{0,1\}$ it suffices to use Fatou's lemma:
$$\int_0^T\int_{\Omega} W(u)\, dx\, dt \leq \liminf_{\varepsilon\to 0} \int_0^T\int_{\Omega} W(u_\varepsilon) \, dx \, dt \leq \liminf \varepsilon\int_0^T E_\varepsilon[u_\varepsilon(\cdot,t)]\, dt = 0.$$ Consequently, we have $W(u)=0$, i.e., $u \in \{0,1\}$ almost everywhere in $\Omega\times (0,T)$.

The representation $v = \phi \circ u = \cnght \chi$ now follows from $\phi(0)=0$ and $\phi(1)=\cnght$.
Since $v \in BV\left(\Omega\times(0,T);\mathbb{R}\right)\cap C^{0,\frac{1}{2}}\left([0,T];L^{1}(\Omega)\right)\cap L^{\infty}\left(0,T;BV(\Omega)\right)$, it follows that the same holds true for $u=\frac{1}{\cnght}v$.

\subsection{Equipartition of energy}

The following theorem states that, asymptotically as $\eps \searrow 0$, the anisotropic Dirichlet energy $\frac{1}{2}\volint\eps f(-\nabla u_{\eps})dx dt$ and the nonconvex term $\frac{1}{2}\volint \frac{1}{\eps}W(u_{\eps})dx dt$ contribute equally to the Cahn-Hilliard energy. The equipartition of energy also holds true in a localized form. The argument relies crucially on the energy convergence assumption (\ref{energy-conv}) and uses a trick introduced by Modica--Mortola for Gamma-convergence of the energies \cite{ModicaMortola} (see also \cite{bogomolnyi}). Equipartition results of this kind have been used to prove conditional convergence in the static or dynamic case since \cite{LuckhausModica,Modica}. 

Statements (i) and (ii) can be viewed primarily as preparatory results for the equipartition statements (iii)--(v).

\begin{thm}\label{equi-thm}
	Under the assumptions of Theorem \ref{sil-thm}, including the energy convergence (\ref{energy-conv}), the following convergence statements hold true in the limit $\eps \searrow 0$:
	\begin{enumerate}[(i)]
		\item
		\begin{equation}\label{equi1}
		\frac{1}{2}\left(\eps f(-\nabla u_{\eps})+\frac{1}{\eps}W(u_{\eps}) \right)\weakstar \cnght \sigma(\nu)\ggtv \qquad \text{in } \rmeas(\Omega\times[0,T]),
		\end{equation}
		\item
		\begin{equation}\label{equi2}
		\sigma(-\nabla(\phi \circ u_{\eps})) \weakstar \cnght \sigma(\nu)\ggtv \qquad \text{in } \rmeas(\Omega\times[0,T]),
		\end{equation}
		\item
		\begin{equation}\label{equi3}
		\sqrt{\eps f(-\nabla u_{\eps})}-\sqrt{\frac{1}{\eps}W(u_{\eps})} \longrightarrow 0 \qquad \text{in } L^{2}(\Omega\times(0,T)),
		\end{equation}
		\item
		\begin{equation}\label{equi4}
		\frac{1}{2}\left(\eps f(-\nabla u_{\eps})-\frac{1}{\eps}W(u_{\eps}) \right)\longrightarrow 0 \qquad \text{in } L^{1}(\Omega\times(0,T)),
		\end{equation}
		\item
		\begin{equation}\label{equi5}
		\eps f(-\nabla u_{\eps}) \weakstar \cnght \sigma(\nu)\ggtv \quad\text{and} \quad \frac{1}{\eps}W(u_{\eps}) \weakstar \cnght \sigma(\nu)\ggtv \qquad \text{in } \rmeas(\Omega\times[0,T]).
		\end{equation}
	\end{enumerate}
\end{thm}

\begin{proof}	
	\begin{enumerate}[(i)]
		\item
		It suffices to show that 
		\begin{equation*}\frac{1}{2}\volint \left(\eps f(-\nabla u_{\eps})+\frac{1}{\eps}W(u_{\eps}) \right) \zeta dx dt \to \cnght \volint \sigma(\nu)\zeta \ggtv 
		\end{equation*}
		for all $\zeta \in C^{1}(\Omega\times [0,T])$ with $0\leq \zeta \leq 1$. By a density and linearity argument, the same convergence then holds true for all $\zeta \in C(\Omega\times [0,T])$.
				
		If $\zeta \in C^{1}(\Omega\times [0,T])$ and $0\leq \zeta \leq 1$, we can use Young's inequality and the chain rule for Sobolev functions to estimate
		\begin{align*}
		\frac{1}{2}\volint \left(\eps f(-\nabla u_{\eps})+\frac{1}{\eps}W(u_{\eps}) \right) \zeta dx dt
		&\geq \volint \sqrt{f(-\nabla u_{\eps})W(u_{\eps}) } \,\zeta dx dt\\
		&= \volint \sigma(-\phi^{\prime}(u_{\eps})\nabla u_{\eps}) \zeta dx dt\\
		&= \volint \sigma(-\nabla(\phi\circ u_{\eps}))\zeta dx dt.
		\end{align*}
		
		By Lemma \ref{energy-as-sup}, we have that 
		\begin{equation}\label{pointwise-dual}
		\volint \sigma(B)\,dxdt=\sup_{\eta \in C^{1}(\Omega\times[0,T])\atop \sigma\pol(\eta)\leq 1}\volint B\cdot \eta\,dx dt
		\end{equation}
		for all $B\in L^{1}(\Omega\times(0,T))^{d}$. 
		
		We apply (\ref{pointwise-dual}) and the $L^{1}$-convergence $\phi\circ u_{\eps}\to \cnght \chi$ as $\eps \searrow 0$ to the above estimate, which, recalling Lemma \ref{surface-tensions}, yields
		\begin{align}\label{modicamortola-geq}
		\liminf_{\eps \searrow 0} \frac{1}{2}\volint &\left(\eps f(-\nabla u_{\eps})+\frac{1}{\eps}W(u_{\eps}) \right) \zeta dx dt\nonumber \\
		&\geq \liminf_{\eps \searrow 0} \volint \sigma(-\nabla(\phi\circ u_{\eps}))\zeta dx dt \nonumber\\
		&= \liminf_{\eps \searrow 0} \sup_{\eta}\volint \eta \cdot (-\nabla(\phi\circ u_{\eps}))\zeta dx dt \nonumber \\
		&= \liminf_{\eps \searrow 0} \sup_{\eta}\volint(\phi\circ u_{\eps})\dvg(\zeta \eta) dx dt \nonumber \\
		&\geq \cnght \sup_{\eta}\volint\chi\dvg(\zeta \eta) dx dt \nonumber \\
		&= -\cnght \sup_{\eta}\volint \zeta \eta \cdot \frac{\nabla \chi}{\ggtv}\ggtv \nonumber \\
		&= \cnght \sup_{\eta}\volint \zeta \eta \cdot \nu\ggtv\nonumber \\
		&= \cnght \volint \sigma(\nu)\zeta\ggtv, 
		\end{align}
		where the supremum is taken over all $\eta \in C^{1}(\Omega\times [0,T])^{d}$ such that $\sigma\pol(\eta)\leq 1$.
		
		To prove the estimate from above, we observe that (\ref{modicamortola-geq}) also applies to the function $1-\zeta$ instead of $\zeta$, and use the energy convergence assumption (\ref{energy-conv}). Indeed,
		\begin{align}
		\limsup_{\eps \searrow 0} \frac{1}{2}\volint &\left(\eps f(-\nabla u_{\eps})+\frac{1}{\eps}W(u_{\eps}) \right) \zeta dx dt\nonumber \\
		&=\limsup_{\eps \searrow 0} \left(\frac{1}{2}\volint \left(\eps f(-\nabla u_{\eps})+\frac{1}{\eps}W(u_{\eps}) \right) dx dt \vphantom{- \frac{1}{2}\volint \left(\eps f(-\nabla u_{\eps})+\frac{1}{\eps}W(u_{\eps}) \right) (1-\zeta) dx dt} \right.\nonumber \\
		&\hspace{8.5pt}\left.\vphantom{\frac{1}{2}\volint \left(\eps f(-\nabla u_{\eps})+\frac{1}{\eps}W(u_{\eps}) \right) dx dt} - \frac{1}{2}\volint \left(\eps f(-\nabla u_{\eps})+\frac{1}{\eps}W(u_{\eps}) \right) (1-\zeta) dx dt \right)\nonumber \\
		&=\lim_{\eps \searrow 0} \int_{0}^{T}E_{\eps}[u_{\eps}(t)]dt\nonumber \\
		&\quad {- \liminf_{\eps \searrow 0}\frac{1}{2}}\volint \left(\eps f(-\nabla u_{\eps})+\frac{1}{\eps}W(u_{\eps}) \right) (1-\zeta) dx dt\nonumber \\
		&\leq \int_{0}^{T}E[u(t)]dt - \cnght \volint \sigma(\nu) (1-\zeta) \ggtv \nonumber \\
		&= \cnght \volint \sigma(\nu) \zeta \ggtv . \label{modicamortola-leq}
		\end{align}
		The combination of the two inequalities (\ref{modicamortola-geq}) and (\ref{modicamortola-leq}) yields the first claim.
		\item In the same manner as in (i), it suffices to show that 
		\begin{equation*}\volint \sigma(-\nabla(\phi \circ u_{\eps})) \zeta dx dt \to \cnght \volint \sigma(\nu)\zeta \ggtv 
		\end{equation*}
		whenever $\zeta \in C^{1}(\Omega\times [0,T])$ and $0 \leq \zeta \leq 1$.
		
		The $\liminf$ inequality follows from the chain of inequalities in (\ref{modicamortola-geq}).
		
		For the $\limsup$ inequality we can use Young's inequality and (i) to compute
		\begin{align*}
		\limsup_{\eps \searrow 0}\volint \sigma(-\nabla(\phi \circ u_{\eps})) \zeta dx dt &\leq \lim_{\eps \searrow 0}\frac{1}{2}\volint \left(\eps f(-\nabla u_{\eps})+\frac{1}{\eps}W(u_{\eps}) \right) \zeta dx dt\\
		&=\cnght \volint \sigma(\nu) \zeta \ggtv.
		\end{align*}
		\item By taking $\zeta \equiv 1$ as test functions for the weak-* limits in (i) and (ii), we see that 
		\begin{align}\label{modicamortola-eq}
		\lim_{\eps \searrow 0}\frac{1}{2}\volint \left(\eps f(-\nabla u_{\eps})+\frac{1}{\eps}W(u_{\eps}) \right) &dx dt = \lim_{\eps \searrow 0}\volint \sigma(-\nabla(\phi \circ u_{\eps})) dx dt \nonumber\\& = \cnght\volint \sigma(\nu)\ggtv.
		\end{align} 
		This observation allows us to compute
		\begin{align*}
		\volint& \left(\sqrt{\eps f(-\nabla u_{\eps})}-\sqrt{\frac{1}{\eps}W(u_{\eps})} \right)^{2} dx dt\\
		&=\volint\left(\eps f(-\nabla u_{\eps})+\frac{1}{\eps}W(u_{\eps}) \right) dx dt - 2\volint \sqrt{f(-\nabla u_{\eps})W(u_{\eps})}dx dt\\
		&=\volint\left(\eps f(-\nabla u_{\eps})+\frac{1}{\eps}W(u_{\eps}) \right) dx dt - 2\volint \sigma(-\nabla(\phi \circ u_{\eps}))dx dt\\
		&\longrightarrow 0 \qquad \text{as } \eps \searrow 0.
		\end{align*}
		\item Young's inequality yields
		\begin{align*}
		\left|\eps f(-\nabla u_{\eps})-\frac{1}{\eps}W(u_{\eps})\right|&\leq \delta \left(\sqrt{\eps f(-\nabla u_{\eps})}+\sqrt{\frac{1}{\eps}W(u_{\eps})} \right)^{2} \\&\quad+ \frac{1}{4\delta}\left(\sqrt{\eps f(-\nabla u_{\eps})}-\sqrt{\frac{1}{\eps}W(u_{\eps})} \right)^{2}
		\end{align*}
		for all $\delta > 0$. Taking the limit $\eps \searrow 0$ and using (iii) as well as (\ref{modicamortola-eq}), we obtain
		\begin{align*}
		\limsup_{\eps \searrow 0} \volint&\left|\eps f(-\nabla u_{\eps})-\frac{1}{\eps}W(u_{\eps})\right| dx dt \\ 
		&\leq \delta \limsup_{\eps \searrow 0} \volint \left(\sqrt{\eps f(-\nabla u_{\eps})}+\sqrt{\frac{1}{\eps}W(u_{\eps})} \right)^{2} dx dt\\
		&= \delta \limsup_{\eps \searrow 0} \volint \left(\eps f(-\nabla u_{\eps})+\frac{1}{\eps}W(u_{\eps}) + 2 \sigma(-\nabla(\phi \circ u_{\eps})) \right) dx dt\\
		&= 4 \delta \cnght\volint \sigma(\nu)\ggtv.
		\end{align*}
		Since $\delta > 0$ is arbitrary, it follows that
		\begin{equation*}
		\lim_{\eps \searrow 0} \volint\left|\eps f(-\nabla u_{\eps})-\frac{1}{\eps}W(u_{\eps})\right| dx dt=0.
		\end{equation*}
		
		\item The two convergence results follow from adding and subtracting (i) and (iv), respectively.
	\end{enumerate}
	
\end{proof}

\subsection{Construction of the normal velocity}

This construction follows the argument in the first step of \cite[Proposition 2.10]{LauxSimon}, where it was carried out by Simon and the first author for the (multiphase) isotropic case. The idea is to introduce the velocity as a Radon--Nikod\'{y}m density $V:=\frac{\partial_{t}\chi}{\ggtv}$.

For simplicity of notation, we denote the energy density of the anisotropic Cahn--Hilliard energy by
\begin{equation*}
e_{\eps}(u,x):=\frac{\eps}{2}f(-\nabla u(x))+\frac{1}{2\eps}W(u(x)).
\end{equation*}

\begin{lem}\label{vel-lem}
	Under the assumptions of Theorem \ref{sil-thm}, including the energy convergence (\ref{energy-conv}), there exists a $|\nabla\chi|$-measurable normal velocity $V:\Omega\times(0,T)\to \real$ satisfying (\ref{velocity-integ}) and (\ref{velocity-criterion}).
\end{lem}

The first step of the proof is to show that the Radon--Nikod\'{y}m theorem is applicable: Given a smooth test function $\zeta \in \tst(\Omega\times (0,T))$, the definition of distributional derivatives and an application of the Cauchy--Schwarz inequality yield
\begin{align}\label{velocity-cs}
\left|\int_{0}^{T}\int_{\Omega} (\phi \circ u_{\eps}) \partial_{t}\zeta dx dt\right| &=
\left|\int_{0}^{T}\int_{\Omega} \partial_{t}(\phi \circ u_{\eps})\zeta dx dt\right| \nonumber\\
&= \left|\int_{0}^{T}\int_{\Omega} \phi^{\prime}(u_{\eps})\partial_{t}u_{\eps}\zeta dx dt \right|\nonumber\\
&\leq\left(\int_{0}^{T}\int_{\Omega} \eps \big|\partial_{t}u_{\eps}\big|^{2}dx dt\right)^{\frac{1}{2}} \left(\int_{0}^{T}\int_{\Omega}\frac{1}{\eps}W(u_{\eps})\zeta^{2}dx dt\right)^{\frac{1}{2}}\nonumber\\
&\leq \sqrt{2} \left(\int_{0}^{T}\int_{\Omega}\eps \big|\partial_{t}u_{\eps} \big|^{2} dx dt\right)^{\frac{1}{2}}\left(\int_{0}^{T}\int_{\Omega}e_{\eps}(u_{\eps}(t), x)\zeta^{2}dx dt\right)^{\frac{1}{2}}.
\end{align}

It follows from (\ref{equi1}) that 
\begin{equation}
\lim_{\eps \to 0}\int_{0}^{T}\int_{\Omega}e_{\eps}(u_{\eps}(t), x)\zeta^{2}dx dt = \cnght \int_{0}^{T}\int_{\Omega}\zeta^{2}\sigma(\nu) \ggtv.
\end{equation}
Thus, by taking the limit inferior on both sides of (\ref{velocity-cs}) and recalling that $\phi \circ u_{\eps} \to \cnght \chi$ in $L^{1}(\Omega\times (0,T))$ as $\eps \searrow 0$, we find
\begin{align}
\cnght \left|\volint \chi\partial_{t}\zeta dx dt\right|&\leq \sqrt{2} \liminf_{\eps \to 0}\left(\volint \eps \big|\partial_{t}u_{\eps} \big|^{2}dx dt\right)^{\frac{1}{2}}\left(\cnght \volint \zeta^{2}\sigma(\nu)\ggtv \right)^{\frac{1}{2}}\nonumber \\
&\leq \left(\frac{2E_{0}}{c_{g}} \right)^{\frac{1}{2}}\left(\cnght \volint \zeta^{2}\sigma(\nu)\ggtv \right)^{\frac{1}{2}}, \label{velocity-functineq}
\end{align} 
where the second inequality follows from the optimal energy dissipation inequality (\ref{aac-oed}).

It is desirable to reformulate (\ref{velocity-functineq}) in terms of open sets instead of test functions: Given an open set $A\subseteq \Omega \times (0,T)$, let us maximize the left-hand side of (\ref{velocity-functineq}) over all $\zeta \in \tst(A)$ with $\big|\zeta\big|\leq 1$. This provides us with the inequality 
\begin{align}
\left|\partial_{t}\chi\right| (A)&\leq \left(\frac{2 E_{0}}{\cnght c_{g}} \right)^{\frac{1}{2}} \sup_{\zeta \in \tst(A) \atop |\zeta|\leq 1} \left(\volint \zeta^{2}\sigma(\nu)\ggtv \right)^{\frac{1}{2}}\nonumber \\
&\leq \left(\frac{2 E_{0} \max_{S^{d-1}}\sigma}{\cnght c_{g}} \right)^{\frac{1}{2}}  \sqrt{\ggtv(A)}. \label{velocity-setineq}
\end{align}

Making use of the outer regularity of Radon measures, we see that
\begin{equation*}
\left|\partial_{t}\chi\right| (E) = \inf_{A \supseteq E \atop A \text{ open}} \left|\partial_{t}\chi\right| (A) \leq \left(\frac{2 E_{0} \max_{S^{d-1}}\sigma}{\cnght c_{g}} \right)^{\frac{1}{2}} \inf_{A \supseteq E \atop A \text{ open}} \sqrt{\ggtv(A)} = 0
\end{equation*}
for all Borel sets $E \subseteq \Omega \times(0,T)$ such that $\ggtv(E)=0$, i.e., $\partial_{t}\chi \ll \ggtv$. By the Radon--Nikod\'{y}m theorem, there exists a $\ggtv$-measurable function $V$ such that $\partial_{t}\chi = V\ggtv$ on the open set $\Omega\times(0,T)$.

In order to prove the square integrability (\ref{velocity-integ}) of $V$, we go back to (\ref{velocity-functineq}). Let us first fix a finite number $M>0$ and find a sequence $\{\zeta_{k}\}_{k \in \nat}$ of smooth test functions such that $\zeta_{k} \to V\chi_{\{|V|\leq M\}}$ in $L^{2}(\ggtv)$ and $|\zeta_{k}|\leq M$ for all $k \in \nat$. Then it follows by dominated convergence that $V\zeta_{k} \to V^{2}\chi_{\{|V|\leq M\}}$ in $L^{1}(\ggtv)$ and, therefore,
\begin{equation*}
\volint \chi\partial_{t}\zeta_{k} dx dt = -\volint\zeta_{k}\, \partial_{t}\chi = -\volint V\zeta_{k}\,\ggtv \longrightarrow -\volint V^{2}\chi_{\{|V|\leq M\}}\ggtv,
\end{equation*}
whereas the $L^{2}$-convergence gives
\begin{equation*}
\volint \zeta_{k}^{2}\sigma(\nu) \ggtv \longrightarrow \volint V^{2}\chi_{\{|V|\leq M\}}\sigma(\nu)\ggtv.
\end{equation*}
Plugging in $\zeta_{k}$ and taking the limit $k \to \infty$ in (\ref{velocity-functineq}), we now obtain
\begin{align*}
\cnght \volint V^{2}\chi_{\{|V|\leq M\}}\ggtv&\leq \left(\frac{2E_{0}}{c_{g}} \right)^{\frac{1}{2}}\left(\cnght \volint V^{2}\chi_{\{|V|\leq M\}}\sigma(\nu)\ggtv \right)^{\frac{1}{2}}\\
&\leq \left(\frac{2E_{0}\max_{S^{d-1}}\sigma}{c_{g}} \right)^{\frac{1}{2}}\left(\cnght \volint V^{2}\chi_{\{|V|\leq M\}}\ggtv \right)^{\frac{1}{2}}.
\end{align*}

By rearranging this inequality and taking $M \to \infty$ with the help of the monotone convergence theorem, we find the desired integrability statement, namely
\begin{equation}
\volint V^{2}\ggtv \leq \frac{2E_{0}\max_{S^{d-1}}\sigma}{\cnght c_{g}} <\infty.
\end{equation}

The final step is to prove the identity (\ref{velocity-criterion}) (see also \cite[Lemma 7]{HenselLaux}). Given a test function $\zeta \in C^{1}(\Omega\times [0,T])$ and an intermediate time $\horiz \in (0,T]$, we introduce a cutoff $\eta_{\alpha}$ in time such that
\begin{equation*}
\eta_{\alpha} \in C^{1}(\real;[0,1]), \qquad \eta_{\alpha}\equiv 0 \quad \text{on } (-\infty, \frac{\alpha}{2}]\cup[\horiz-\frac{\alpha}{2}, \infty), \qquad \eta_{\alpha}\equiv 1 \quad \text{on } [\alpha, \horiz-\alpha], 
\end{equation*}
and set $\zeta_{\alpha}(x,t):=\zeta(x,t)\eta_{\alpha}(t)$. Then, in particular, we have $\zeta_{\alpha}\in C_{c}^{1}(\Omega\times(0,T))$ for sufficiently small $\alpha > 0$. The definition of $V$ as a Radon--Nikod\'{y}m density yields
\begin{align}\label{compact-velocity-criterion}
\int_{0}^{\horiz}\int_{\Omega}\zeta_{\alpha}(x,t)V(x,t)\ggtv&=-\int_{0}^{\horiz}\int_{\Omega}\chi(x,t)\partial_{t}\zeta_{\alpha}(x,t)\,dxdt\nonumber \\
&=-\int_{0}^{\horiz}\int_{\Omega}\chi(x,t)\eta_{\alpha}(t)\partial_{t}\zeta(x,t)\,dxdt\nonumber\\
&\quad-\int_{0}^{\horiz}\int_{\Omega}\chi(x,t)\zeta(x,t)\,dx\,\eta_{\alpha}^{\prime}(t)\,dt.
\end{align}

In the limit $\alpha \searrow 0$, the left-hand side integral and the first right-hand side integral converge due to the dominated convergence theorem. For the last integral, we observe that $\eta_{\alpha}^{\prime}\weakstar \delta_{0}-\delta_{\horiz}$ in $\rmeas(\real)$ as $\alpha \searrow 0$. Furthermore, we have $\chi=u\in C^{0,\frac{1}{2}}\left([0,T];L^{1}(\Omega)\right)$ by Lemma \ref{compact-lem}, and $\zeta$ is uniformly continuous, so that the map $t \mapsto \int_{\Omega}\chi(x,t)\zeta(x,t)dx$ is continuous on $[0,T]$. We can therefore compute the limits of all three integrals and obtain
\begin{align*}
\int_{0}^{\horiz}\int_{\Omega}\zeta(x,t)V(x,t)\ggtv&=-\int_{0}^{\horiz}\int_{\Omega}\chi(x,t)\partial_{t}\zeta(x,t)\,dxdt -\int_{\Omega}\chi_{0}(x)\zeta(x,0)\,dx \\&\quad+ \int_{\Omega}\chi(x,\horiz)\zeta(x, \horiz)\,dx,
\end{align*}
which proves (\ref{velocity-criterion}).

\subsection{Relative entropies}

A key ingredient to derive the sharp interface limit are the following notions of relative entropies.

\begin{dfn}\label{refunctionals}
\begin{enumerate}[(i)]
	\item Let $u=\chi \in BV(\Omega)$ take values in $\{0,1\}$ almost everywhere. As in the definition of the anisotropic surface energy, let $\nu:=-\frac{\nabla\chi}{|\nabla\chi|}$ be the measure-theoretic outer unit normal. The relative entropy of $u$ with respect to a vector field $\xi \in C(\Omega)^{d}$ is
	\begin{equation}\label{relentropy}
	\re{u}{\xi}:=\cnght\int_{\Omega}\left(\sigma(\nu)-|\xi|\psi(|\xi|)D\sigma(\xi)\cdot\nu \right) \ggtv,
	\end{equation}
	where $\psi \in C^{\infty}([0,\infty)$ is a cutoff function such that the three properties in (\ref{cutoff}) hold true.  
	\item Let $\eps > 0$ and $u_{\eps} \in H^{1}(\Omega)\cap L^{\infty}(\Omega)$. The $\eps$-relative entropy of $u_{\eps}$ with respect to a vector field $\xi \in C(\Omega)^{d}$ is
	\begin{equation}\label{epsrelentropy}
	\epsre{u_{\eps}}{\xi}:=\int_{\Omega}\left(\sigma(-\nabla u_{\eps})+|\xi|\psi(|\xi|)D\sigma(\xi)\cdot\nabla u_{\eps} \right)\sqrt{W(u_{\eps})} dx.
	\end{equation}
\end{enumerate}
\end{dfn}

\begin{rmk}
	\begin{enumerate}[(i)]
		\item For $u = \chi$ as in Definition \ref{refunctionals}(i), we can define the set of finite perimeter $A:=\left\{x \in \Omega\, \big|\,u(x)=1 \right\}$ (up to a null set). By De Giorgi's structure theorem (see \cite[Theorem 3.59]{AmbrosioFuscoPallara}), one can replace the total variation measure $|\nabla \chi|$ with the $(d-1)$-dimensional Hausdorff measure and write 
		\begin{equation*}
		\re{u}{\xi}= \cnght \int_{\redbd A} \left(\sigma(\nu)-|\xi|\psi(|\xi|)D\sigma(\xi)\cdot\nu \right)d\hausd,
		\end{equation*}
		with $\redbd$ denoting the reduced boundary of a set of finite perimeter.
		\item One can also show that $|\nabla \chi(t)|\otimes \leb{1}\mres (0,T) = |\nabla \chi|$ as Radon measures on $\Omega\times(0,T)$. Thus, integrating the relative entropy over time yields
		\begin{align*}
		\int_{0}^{T}\re{u(t)}{\xi(t)}dt & = \cnght \volint \left(\sigma(\nu)-|\xi|\psi(|\xi|)D\sigma(\xi)\cdot\nu \right)|\nabla\chi(t)|dt\\
		&=\cnght \volint \left(\sigma(\nu)-|\xi|\psi(|\xi|)D\sigma(\xi)\cdot\nu \right) \ggtv.
		\end{align*}
	\end{enumerate}
\end{rmk}

If $\xi$ is chosen to be a smooth approximation of the normal, then these relative entropies serve as tilt excesses and can be used to control quadratic error terms: Let $\xi \in C\left(\Omega\times [0,T]\right)^{d}$ such that $|\xi|\leq 1$ on $\Omega\times [0,T]$. By Lemma \ref{dziuk}(i), we have
\begin{equation}\label{eqn:normalEntropyControl}
\int_{0}^{T}\int_{\Omega}|\xi-\nu|^{2}\ggtv \leq \frac{1}{\cnght c_{\sigma}} \int_{0}^{T}\re{u(t)}{\xi(t)}dt.
\end{equation}

The link between the relative entropy $\mathscr{E}$ and the phase-field version $\mathscr{E_{\eps}}$ is the following convergence statement:
\begin{lem}\label{re-control}
	Under the assumptions of Theorem \ref{sil-thm}, including the energy convergence (\ref{energy-conv}), we have 
	\begin{equation*}
	\lim_{\eps \searrow 0}\int_{0}^{T}\epsre{u_{\eps}(t)}{\xi(t)}dt = \int_{0}^{T}\re{u(t)}{\xi(t)}dt
	\end{equation*}
	for every vector field $\xi \in C\left(\Omega\times [0,T]\right)^{d}$ such that $|\xi|\leq 1$ on $\Omega\times [0,T]$, where $u=\chi$ is the limiting function constructed in Lemma \ref{compact-lem}.
\end{lem}

\begin{proof}
	A direct computation yields
	\begin{align}
	\int_{0}^{T}\epsre{u_{\eps}(t)}{\xi(t)}dt&= \int_{0}^{T}\int_{\Omega}\left(\sigma(-\nabla u_{\eps})+|\xi|\psi(|\xi|)D\sigma(\xi)\cdot\nabla u_{\eps} \right)\sqrt{W(u_{\eps})}dxdt\nonumber\\
	&=\int_{0}^{T}\int_{\Omega}\sigma(-\nabla \left(\phi \circ  u_{\eps})\right)dxdt\nonumber \\&\quad+\int_{0}^{T}\int_{\Omega}|\xi|\psi(|\xi|)D\sigma(\xi)\cdot\nabla (\phi \circ u_{\eps})dxdt\nonumber\\
	&\to \cnght \int_{0}^{T}\sigma(\nu)\ggtv + \cnght\int_{0}^{T}\int_{\Omega}|\xi|\psi(|\xi|)D\sigma(\xi)\cdot\nabla \chi\nonumber\\
	&=\int_{0}^{T}\re{u(t)}{\xi(t)}dt, 
	\end{align}
	where the first term converges due to Theorem \ref{equi-thm}(ii). The convergence of the second term follows from the observation that $\nabla (\phi \circ u_{\eps})\weakstar \cnght \nabla \chi$ in $\rmeas(\Omega\times[0,T])$, which is a consequence of the bound in (\ref{compact-gradient}) and the $L^{1}$-convergence $\phi \circ u_{\eps}\to \cnght \chi$.
\end{proof}

While we have only used the first inequality in Lemma \ref{dziuk} so far, the second inequality helps us to show that the tilt excess can be made arbitrarily small by approximating the normal $\nu$ with suitable vector fields $\xi$.
\begin{lem}\label{smooth-xi-approx}
	For every $\delta > 0$, there exists a smooth vector field $\xi \in C^{1}(\Omega\times[0,T])^{d}$ such that $|\xi|\leq 1$ in $\Omega\times[0,T]$ and
	\begin{equation*}
	\int_{0}^{T}\re{u(t)}{\xi(t)}dt < \delta.
	\end{equation*}
\end{lem}

\begin{proof}
	The outer unit normal $\nu$ is $\ggtv$-measurable, and the estimate
	\begin{equation*}
	\int_{0}^{T}\int_{\Omega}|\nu|^{2}\ggtv = \int_{0}^{T}\int_{\Omega}\ggtv \leq \frac{1}{c_{0}\min_{S^{d-1}}\sigma}\int_{0}^{T}E[u(t)]dt<\infty
	\end{equation*}
	shows that $\nu \in L^{2}(\ggtv)$. Since $\ggtv$ is a Radon measure on $\Omega\times [0,T]$, there exists an approximating sequence $\left\{\xi_{n} \right\}_{n\in\nat}\subset C^{\infty}(\Omega\times[0,T])$ such that $\xi_{n}\to\nu$ in $L^{2}(\ggtv)$ as $n \to \infty$.
	
	Clearly, $\nu$ takes values in the closed convex set $\overline{B}_{1}=\left\{p \in \real^{d}\,\Big|\,|p|\leq 1 \right\}$ almost everywhere with respect to $\ggtv$, so we can choose $\xi_{n}$ in a way that ensures that $|\xi_{n}|\leq 1$ on $\Omega\times [0,T]$ for all $n \in \nat$.
	
	By Lemma \ref{dziuk}(ii) and the Cauchy--Schwarz inequality, we obtain
	\begin{align*}
	\int_{0}^{T}\re{u(t)}{\xi_{n}(t)}dt &\leq \cnght\, C_{\sigma}\int_{0}^{T}\int_{\Omega}\left(|\nu-\xi_{n}|^{2}+(|\nu|-|\xi_{n}|)\right)\ggtv\\
	&\leq \cnght\, C_{\sigma}\int_{0}^{T}\int_{\Omega}|\nu-\xi_{n}|^{2}\ggtv\\
	&\quad+\cnght\, C_{\sigma}\left(\int_{0}^{T}\int_{\Omega}\ggtv\right)^{\frac{1}{2}}\left(\int_{0}^{T}\int_{\Omega}\left|\nu -\xi_{n}\right|^{2}\ggtv\right)^{\frac{1}{2}}\\
	&\leq \cnght\, C_{\sigma}\left\|\nu-\xi_{n}\right\|_{L^{2}(|\nabla \chi|)}^{2} \\
	&\quad+\sqrt{\frac{\cnght}{\min_{S^{d-1}}\sigma}} C_{\sigma}\left(\int_{0}^{T}E[u(t)]dt\right)^{\frac{1}{2}}\left\|\nu-\xi_{n}\right\|_{L^{2}(|\nabla \chi|)}\\
	&\longrightarrow 0
	\end{align*}
	as $n \to \infty$.
\end{proof}

\subsection{Convergence of the curvature term}

The following two subsections are dedicated to proving the distributional law of anisotropic mean curvature flow (\ref{amcf-df}) for the limit function $u$. The strategy is to use $B\cdot \eps \nabla u_{\eps}$ as a test function in the distributional formulation of the anisotropic Allen--Cahn equation (\ref{aac-df}), where $B \in C^{1}(\Omega\times[0,T])^{d}$, and then pass to the limit as $\eps \searrow 0$ on both sides separately.

Here and in the following subsections, $C$ denotes a generic positive constant that may depend on the pair of anisotropies $(\sigma,\mu)$, the double-well potential $W$, the time horizon $T$, and the cutoff function $\psi$. The constant $C$ is not necessarily the same on every occurrence.

\begin{thm}\label{curvature-thm}
	Under the assumptions of Theorem \ref{sil-thm}, including the energy convergence (\ref{energy-conv}), we have  
	\begin{align}\label{curvature-statement}
	\lim_{\eps \searrow 0}\volint \frac{1}{2}&\left(Df(-\nabla u_{\eps})\cdot \nabla(B\cdot \eps \nabla u_{\eps})-\frac{1}{\eps^{2}}W^{\prime}(u_{\eps})B \cdot \eps \nabla u_{\eps}\right)dx dt \nonumber \\
	&= \cnght \volint \nabla B : \left(\sigma(\nu)I_{d}-\nu \otimes D\sigma(\nu) \right)\ggtv
	\end{align}
	for all $B\in C^{1}(\Omega\times[0,T])^{d}$.
\end{thm}

This theorem was first proved by Cicalese, Nagase, and Pisante in \cite[Theorem 3.3]{CicaleseNagasePisante}. However, we proceed with the alternative strategy of proof proposed in \cite[Proposition 4.5]{LauxKyoto}.

We define the energy-stress tensor $T_{\eps}$ by
\begin{equation}\label{curvature-aest}
T_{\eps}:=\frac{1}{2}\left(\eps f(-\nabla u_{\eps})+\frac{1}{\eps}W(u_{\eps}) \right)I_{d}+\eps \nabla u_{\eps}\otimes \frac{1}{2}Df(-\nabla u_{\eps}).
\end{equation}

With this definition, the statement of Theorem \ref{curvature-thm} can be rewritten as a weak-* convergence claim for the energy-stress tensor: Indeed, an integration by parts on the flat torus $\Omega$ yields

\begin{align}\label{curvature-ibp}
\volint \nabla B : T_{\eps}\,dx dt &= -\volint B \cdot \dvg T_{\eps}\,dx dt \nonumber\\
&= -\frac{1}{2}\volint B\cdot \nabla \left(\eps f(-\nabla u_{\eps})+\frac{1}{\eps}W(u_{\eps}) \right)dx dt \nonumber \\
&\quad -\frac{1}{2}\volint \dvg (Df(-\nabla u_{\eps})) B \cdot \eps \nabla u_{\eps}\,dx dt\nonumber \\
&\quad -\frac{1}{2}\volint B \cdot \eps D^{2}u_{\eps}Df(-\nabla u_{\eps})dx dt\nonumber \\
&= -\frac{1}{2}\volint B\cdot \frac{1}{\eps}\nabla \left(W(u_{\eps}) \right)dx dt \nonumber \\
&\quad -\frac{1}{2}\volint \dvg (Df(-\nabla u_{\eps})) B \cdot \eps\nabla u_{\eps}\,dx dt\nonumber \\
&= \volint\frac{1}{2}\left(Df(-\nabla u_{\eps})\cdot \nabla(B\cdot \eps \nabla u_{\eps})-\frac{1}{\eps^{2}}W^{\prime}(u_{\eps})B\cdot \eps \nabla u_{\eps} \right)dx dt.
\end{align}
We observe that the right-hand side is exactly the term appearing in (\ref{curvature-statement}). Thus, in order to prove Theorem \ref{curvature-thm}, it suffices to show that $T_{\eps}\weakstar \cnght \left(\sigma(\nu)I_{d}-\nu \otimes D\sigma(\nu)\right)\ggtv$ as $\real^{d\times d}$-valued Radon measures on $\Omega\times [0,T]$.

As for the first summand of the energy-stress tensor $T_{\eps}$, it follows immediately from the equipartition of energy statement in Theorem \ref{equi-thm}(i) that \begin{equation*}\frac{1}{2}\left(\eps f(-\nabla u_{\eps})+\frac{1}{\eps}W(u_{\eps}) \right)I_{d}\weakstar \cnght \sigma(\nu)I_{d}\ggtv.\end{equation*} 

The following lemma covers the convergence of the second summand of $T_{\eps}$, thereby completing the proof of Theorem \ref{curvature-thm}.

\begin{lem}\label{curvature-athm}
	Under the assumptions of Theorem \ref{sil-thm}, including the energy convergence (\ref{energy-conv}), we have
	\begin{equation}\label{curvature-astate}
	\eps \nabla u_{\eps}\otimes \frac{1}{2}Df(-\nabla u_{\eps}) \weakstar -\cnght \nu \otimes D\sigma(\nu) \ggtv \qquad \text{in } \rmeas(\Omega\times[0,T])^{d\times d}
	\end{equation}
	as $\eps \searrow 0$.
\end{lem} 

To prove Lemma \ref{curvature-athm}, let us first fix a vector field $\xi \in C(\Omega\times [0,T])^{d}$ such that $|\xi|\leq 1$ in $\Omega\times[0,T]$. This vector field will serve as an approximation for the measure-theoretic normal $\nu$. Furthermore let $A \in C(\Omega \times [0,T])^{d \times d}$ be a test function.

For shorter notation, we introduce the approximate outer unit normal $\nu_{\eps}$ via
\begin{equation}\label{eqn:nuEpsDef}
\nu_{\eps}(x):=\begin{cases} -\frac{\nabla u_{\eps}}{|\nabla u_{\eps}|}(x)&\text{if }\nabla u_{\eps}(x)\neq 0,\\e_{1}&\text{if }\nabla u_{\eps}(x)=0.\end{cases}
\end{equation}

Applying the product rule to $f=\sigma^{2}$ and exploiting the positive $0$-homogeneity of $D\sigma$, one can rewrite the left-hand side of (\ref{curvature-astate}) as
\begin{equation}\label{curvature-prodrule}
\eps \nabla u_{\eps}\otimes \frac{1}{2}Df(-\nabla u_{\eps}) = \eps \sigma(-\nabla u_{\eps})\nabla u_{\eps}\otimes D\sigma(-\nabla u_{\eps}) = \eps \sigma(-\nabla u_{\eps})\nabla u_{\eps}\otimes D\sigma(\nu_{\eps}),
\end{equation} 
where one can now conveniently insert the vector field $\xi$ at the cost of two errors controlled by the tilt excess $\mathscr{E}$: Let us add zero twice and compute
\begin{align}\label{curvature-add0}
\volint &\left(\eps \nabla u_{\eps}\otimes \frac{1}{2}Df(-\nabla u_{\eps})\right):A\, dx dt + \cnght \volint \left(\nu \otimes D\sigma(\nu)\right) : A\, \ggtv \nonumber \\
&= \volint \left(\eps \sigma(-\nabla u_{\eps})\nabla u_{\eps}\otimes \left(D\sigma(\nu_{\eps})-\psi(|\xi|)D\sigma(\xi)\right)\right):A\, dx dt \nonumber \\
&\quad + \volint \left(\eps \sigma(-\nabla u_{\eps})\nabla u_{\eps}\otimes \psi(|\xi|)D\sigma(\xi)\right):A\, dx dt \nonumber \\
&\quad + \cnght \volint \left(\nu \otimes \psi(|\xi|)D\sigma(\xi)\right) : A\, \ggtv\nonumber \\
&\quad + \cnght \volint \left(\nu \otimes \left(D\sigma(\nu)-\psi(|\xi|)D\sigma(\xi)\right)\right) : A\, \ggtv.
\end{align}

As a consequence of the smoothness and homogeneity of $\sigma$ and since $\psi\equiv 0$ in a neighborhood of zero, the map $p \mapsto \psi(|p|)D\sigma(p)$ is Lipschitz continuous, i.e.,
\begin{equation}\label{dsigma-quasi-lip}
|\psi(|q|)D\sigma(q)-\psi(|p|)D\sigma(p)|\leq C|q-p| \qquad \text{for all }p, q \in \real^{d}.
\end{equation}

The first right-hand side integral of (\ref{curvature-add0}) can now be estimated as follows:
\begin{align*}
\left|\volint \vphantom{\left(\eps \sigma(-\nabla u_{\eps})\nabla u_{\eps}\otimes \left(D\sigma(\nu_{\eps})-\psi(|\xi|)D\sigma(\xi)\right)\right):A\, dx dt}\right.&\left. \vphantom{\volint}\left(\eps \sigma(-\nabla u_{\eps})\nabla u_{\eps}\otimes \left(D\sigma(\nu_{\eps})-\psi(|\xi|)D\sigma(\xi)\right)\right):A\, dx dt\right|\\
&\leq C\|A\|_{L^{\infty}}\volint \eps \sigma(-\nabla u_{\eps})|\nabla u_{\eps}| |\nu_{\eps}-\xi| dx dt\\
&\leq \frac{\sqrt{2}\,C\|A\|_{L^{\infty}}}{\min_{S^{d-1}}\sigma}\left(\volint \frac{\eps}{2} f(-\nabla u_{\eps}) dx dt\right)^{\frac{1}{2}} \left(\volint \eps f(-\nabla u_{\eps})|\nu_{\eps}-\xi|^{2} dx dt\right)^{\frac{1}{2}}\\
&\leq \frac{\sqrt{2}\,C\|A\|_{L^{\infty}}}{\min_{S^{d-1}}\sigma}{\int_{0}^{T} E_{\eps}[u_{\eps}(t)]dt}^{\frac{1}{2}} \left(\volint \eps f(-\nabla u_{\eps})|\nu_{\eps}-\xi|^{2} dx dt\right)^{\frac{1}{2}},
\end{align*}
and with the help of the energy convergence assumption (\ref{energy-conv}) and Lemma \ref{control-by-epsre} below, we obtain
\begin{align}\label{curvature-first-rhs}
\limsup_{\eps \searrow 0}\left|\volint \vphantom{\left(\eps \sigma(-\nabla u_{\eps})\nabla u_{\eps}\otimes \left(D\sigma(\nu_{\eps})-\psi(|\xi|)D\sigma(\xi)\right)\right):A\, dx dt}\right.&\left. \vphantom{\volint}\left(\eps \sigma(-\nabla u_{\eps})\nabla u_{\eps}\otimes \left(D\sigma(\nu_{\eps})-\psi(|\xi|)D\sigma(\xi)\right)\right):A\, dx dt\right|\nonumber\\
&\leq C \|A\|_{L^{\infty}}\left(\int_{0}^{T} E[u(t)]dt\right)^{\frac{1}{2}}\left(\int_{0}^{T} \re{u(t)}{\xi(t)}dt\right)^{\frac{1}{2}}.
\end{align}

Similarly recalling (\ref{eqn:normalEntropyControl}), the last integral on the right-hand side of (\ref{curvature-add0}) can be controlled by the tilt excess via 
\begin{align}\label{curvature-last-rhs}
\left|\cnght \volint \vphantom{\left(\nu \otimes \left(D\sigma(\nu)-\psi(|\xi|)D\sigma(\xi)\right)\right) : A\, \ggtv dt}\right.&\left.\vphantom{\cnght \volint} \left(\nu \otimes \left(D\sigma(\nu)-\psi(|\xi|)D\sigma(\xi)\right)\right) : A\, \ggtv \right|\nonumber\\
&\leq \cnght\, C\|A\|_{L^{\infty}}\volint |\nu-\xi|\ggtv \nonumber\\
&\leq \cnght \frac{C\|A\|_{L^{\infty}}}{\sqrt{\min_{S^{d-1}}\sigma}}\left(\volint\sigma(\nu)\ggtv\right)^{\frac{1}{2}}\left(\volint |\nu-\xi|^{2}\ggtv\right)^{\frac{1}{2}}\nonumber\\
&\leq \frac{C\|A\|_{L^{\infty}}}{\sqrt{c_{\sigma}\min_{S^{d-1}}\sigma}}\left(\int_{0}^{T} E[u(t)] dt\right)^{\frac{1}{2}}\left(\int_{0}^{T}\re{u(t)}{\xi(t)}dt\right)^{\frac{1}{2}}.
\end{align}

As for the remaining terms in (\ref{curvature-add0}), we have to show that 
\begin{align*}
\volint &\left(\eps \sigma(-\nabla u_{\eps})\nabla u_{\eps}\otimes \psi(|\xi|) D\sigma(\xi)\right):A\, dx dt \\&+ \cnght \volint \left(\nu \otimes \psi(|\xi|) D\sigma(\xi)\right) : A\, \ggtv \longrightarrow 0
\end{align*}
as $\eps \searrow 0$. Since $\psi(|\xi|) A\,D\sigma(\xi)$ is a continuous vector field on $\Omega\times[0,T]$, this problem reduces to proving that $\eps \sigma(-\nabla u_{\eps})\nabla u_{\eps} \weakstar - \cnght \nu \ggtv$ in $\rmeas(\Omega\times [0,T])^{d}$.

To prove this weak-* convergence statement, we rewrite
\begin{align*}
\eps \sigma(-\nabla u_{\eps})\nabla u_{\eps}&= \left(\eps \sigma(-\nabla u_{\eps})-\sqrt{W(u_{\eps})}\right)\nabla u_{\eps}+\sqrt{W(u_{\eps})}\nabla u_{\eps}\\
&= \left(\sqrt{\eps f(-\nabla u_{\eps})}-\sqrt{\frac{1}{\eps}W(u_{\eps})}\right)\sqrt{\eps}\nabla u_{\eps}+\nabla(\phi \circ u_{\eps}).
\end{align*}

The first summand converges strongly to $0$ in $L^{1}(\Omega\times(0,T))$: This follows by the Cauchy-Schwarz inequality if we recall that $\sqrt{\eps}\nabla u_{\eps}$ is bounded in $L^{2}(\Omega\times(0,T))$ as $\eps \searrow 0$ and $\sqrt{\eps f(-\nabla u_{\eps})}-\sqrt{\frac{1}{\eps}W(u_{\eps})} \to 0$ in $L^{2}(\Omega\times (0,T))$ by Theorem \ref{equi-thm}(iii). 

It has been shown in (\ref{compact-function})--(\ref{compact-timeder}) that $\phi \circ u_{\eps}$ is bounded in $W^{1,1}$ as $\eps \searrow 0$. In particular, the total variation $\int\int |\nabla u_{\eps}|$ is uniformly bounded as $\eps \searrow 0$. Since  $\phi\circ u_{\eps}\to \cnght \chi$ in $L^{1}(\Omega\times(0,T))$, it follows that $\nabla(\phi \circ u_{\eps}) \weakstar \cnght \nabla \chi = -\cnght \nu\ggtv$ in $\rmeas(\Omega\times[0,T])^{d}$.

In total, we have
\begin{equation}\label{curvature-middle-rhs}
\eps \sigma(-\nabla u_{\eps})\nabla u_{\eps} \weakstar - \cnght \nu \ggtv \qquad \text{in } \rmeas(\Omega\times[0,T])^{d}.
\end{equation}

By plugging in (\ref{curvature-first-rhs}), (\ref{curvature-last-rhs}), and (\ref{curvature-middle-rhs}) into (\ref{curvature-add0}), one arrives at
\begin{align*}
\limsup_{\eps \searrow 0}\left|\volint\vphantom{\left(\eps \nabla u_{\eps}\otimes \frac{1}{2}Df(-\nabla u_{\eps})\right):A\, dx dt + \cnght \volint \left(\nu \otimes D\sigma(\nu)\right) : A\, \ggtv}\right. &\left.\vphantom{\volint}\left(\eps \nabla u_{\eps}\otimes \frac{1}{2}Df(-\nabla u_{\eps})\right):A\, dx dt + \cnght \volint \left(\nu \otimes D\sigma(\nu)\right) : A\, \ggtv\right| \\
&\leq C\|A\|_{L^{\infty}}\left(\int_{0}^{T}E[u(t)]dt\right)^{\frac{1}{2}}\left(\int_{0}^{T}\re{u(t)}{\xi(t)}dt\right)^{\frac{1}{2}}.
\end{align*}
This concludes the proof of Lemma \ref{curvature-athm} since, appealing to Lemma \ref{smooth-xi-approx}, the time-integrated relative entropy can be made arbitrarily small.

\begin{lem}\label{control-by-epsre}
	Under the assumptions of Theorem \ref{sil-thm}, including the energy convergence (\ref{energy-conv}), and with $\nu_\varepsilon$ defined in (\ref{eqn:nuEpsDef}), we have
	\begin{equation}\label{eps-re-dirichlet}
	\limsup_{\eps\searrow 0}\volint \eps f(-\nabla u_{\eps})|\nu_{\eps}-\xi|^{2}dxdt \leq \frac{\max_{S^{d-1}}\sigma}{c_{\sigma}} \int_{0}^{T}\re{u(t)}{\xi(t)}dt.
	\end{equation}
\end{lem}

\begin{proof}
We apply Lemma \ref{dziuk}(i), the estimate $|\nu_{\eps}-\xi|\leq 2$, and the Cauchy-Schwarz inequality in order to control the occurring integrals by the anisotropic Cahn-Hilliard energy and the $\eps$-relative entropy. Precisely,
\begin{align*}
\volint &\eps f(-\nabla u_{\eps})|\nu_{\eps}-\xi|^{2}dxdt\\&
=\volint \sqrt{W(u_{\eps})}\sigma(-\nabla u_{\eps})|\nu_{\eps}-\xi|^{2}dxdt\\
&\quad+\volint \left(\sqrt{\eps f(-\nabla u_{\eps})}-\sqrt{\frac{1}{\eps}W(u_{\eps})}\right)\sqrt{\eps}\sigma(-\nabla u_{\eps})|\nu_{\eps}-\xi|^{2}dxdt\\
&\leq\frac{\max_{S^{d-1}}\sigma}{c_{\sigma}}\volint \sqrt{W(u_{\eps})}|\nabla u_{\eps}|\left(\sigma(\nu_{\eps})-|\xi|\psi(|\xi|)D\sigma(\xi)\cdot \nu_{\eps} \right)dxdt\\
&\quad+4\left(\volint \eps f(-\nabla u_{\eps})\,dxdt\right)^{\frac{1}{2}}\left(\volint \left(\sqrt{\eps f(-\nabla u_{\eps})}-\sqrt{\frac{1}{\eps}W(u_{\eps})}\right)^{2}dxdt\right)^{\frac{1}{2}}\\
&\leq\frac{\max_{S^{d-1}}\sigma}{c_{\sigma}}\int_{0}^{T}\epsre{u_{\eps}(t)}{\xi(t)}dt\\
&\quad+4\sqrt{2}\left(\int_{0}^{T}E_{\eps}[u_{\eps}(t)]dt\right)^{\frac{1}{2}}\left(\volint \left(\sqrt{\eps f(-\nabla u_{\eps})}-\sqrt{\frac{1}{\eps}W(u_{\eps})}\right)^{2}dxdt\right)^{\frac{1}{2}}.
\end{align*}

One can now take the limit superior on both sides and note that, by assumption, we have $\limsup_{\eps \searrow 0}\int_{0}^{T}E_{\eps}[u_{\eps}(t)]dt<\infty$. Thus, applying Lemma \ref{re-control} in the first term and Theorem \ref{equi-thm}(iii) in the second term yields
\begin{equation*}
\limsup_{\eps \searrow 0}\volint\eps f(-\nabla u_{\eps})|\nu_{\eps}-\xi|^{2}dxdt\leq \frac{\max_{S^{d-1}}\sigma}{c_{\sigma}}\int_{0}^{T}\re{u(t)}{\xi(t)}dt.\qedhere
\end{equation*}
\end{proof}

\subsection{Convergence of the velocity term}

The analogue of Theorem \ref{curvature-thm} for the left-hand side terms is
\begin{thm}\label{velocity-thm}
	Under the assumptions of Theorem \ref{sil-thm}, including the energy convergence (\ref{energy-conv}), we have  
	\begin{equation}\label{velocity-statement}
	\lim_{\eps \searrow 0}\volint g(-\nabla u_{\eps})\partial_{t}u_{\eps} B\cdot \eps \nabla u_{\eps}\,dx dt =- \cnght \volint \frac{1}{\mu(\nu)}VB\cdot \nu\ggtv
	\end{equation}
	for all $B\in C^{1}(\Omega\times[0,T])^{d}$.
\end{thm}

Similarly to the argument for the curvature term, we replace the approximate outer normal $\phn$ by a vector field $\xi \in C(\Omega\times[0,T])^{d}$ such that $|\xi|\leq 1$ on $\Omega\times [0,T]$. The resulting error terms can be controlled by the relative entropy. An additional error term arises as we replace the term $g(-\nabla u_{\eps})$ by its `asymptotic' version $\frac{\sigma(\phn)}{\mu(\phn)}$: Expanding both sides of (\ref{velocity-statement}), we obtain
\begin{align}\label{velocity-add-zero}
\volint & g(-\nabla u_{\eps})\partial_{t}u_{\eps} B\cdot \eps \nabla u_{\eps}\,dx dt+ \cnght \volint \frac{1}{\mu(\nu)}VB\cdot \nu\ggtv\nonumber \\
&=\volint \left(g(-\nabla u_{\eps})-\frac{\sigma(\nu_{\eps})}{\mu(\phn)}\right)\partial_{t}u_{\eps} B\cdot \eps \nabla u_{\eps}\,dx dt\nonumber \\
&\quad - \volint \partial_{t}u_{\eps} B\cdot \eps \sigma(-\nabla u_{\eps})\left(\frac{1}{\mu(\phn)}\phn-\frac{\psi(|\xi|)}{\mu(\xi)}\xi\right)dx dt\nonumber \\
&\quad - \volint \partial_{t}u_{\eps} B\cdot \eps \sigma(-\nabla u_{\eps})\frac{\psi(|\xi|)}{\mu(\xi)}\xi \,dx dt\nonumber + \cnght \volint V B \cdot \frac{\psi(|\xi|)}{\mu(\xi)}\xi\ggtv\nonumber \\
&\quad +\cnght \volint V B \cdot \left(\frac{1}{\mu(\nu)}\nu- \frac{\psi(|\xi|)}{\mu(\xi)}\xi\right)\ggtv,
\end{align}
where we have used the identity $-\sigma(\phn)\nabla u_{\eps}=\sigma(\phn)|\nabla u_{\eps}|\phn = \sigma(-\nabla u_{\eps})\phn$.

To see that the third line on the right-hand side converges to zero as $\eps \searrow 0$, it suffices to show that 
\begin{align}\label{velocity-xi-convergence}
\volint \eps \partial_{t}u_{\eps} \sigma(-\nabla u_{\eps}) \zeta\,dxdt \longrightarrow \cnght \volint V\zeta \ggtv
\end{align}
as $\eps \searrow 0$ for all $\zeta \in C^{1}(\Omega\times[0,T])$. A density argument then yields the same statement for all $\zeta \in C(\Omega\times[0,T])$, so that one can choose $\zeta = B\cdot\frac{\psi(|\xi|)}{\mu(\xi)}\xi$.

In order to prove (\ref{velocity-xi-convergence}), we integrate by parts and write
\begin{align*}
\volint \eps \partial_{t}u_{\eps} \sigma(-\nabla u_{\eps}) \zeta\,dxdt
&=\volint \sqrt{\eps} \partial_{t}u_{\eps} \left(\sqrt{\eps}\sigma(-\nabla u_{\eps})-\sqrt{\frac{1}{\eps}W(u_{\eps})}\right) \zeta\,dxdt\\
&\quad - \volint (\phi \circ u_{\eps})\,\partial_{t}\zeta \, dxdt\\
&\quad +\int_{\Omega}\phi(u_{\eps}(x,T))\zeta(x,T)\,dx - \int_{\Omega}\phi(u_{\eps,0}(x))\zeta(x,0)\,dx.
\end{align*}
The first integral converges to zero as $\eps \searrow 0$ due to the equipartiton of energy, Theorem \ref{equi-thm}(iii), and the optimal dissipation. For the other three integrals, we apply the convergence statements for the subsequence $\eps \searrow 0$ given in Lemma \ref{compact-lem}, which leads to
\begin{align*}
\lim_{\eps \searrow 0}\volint &\eps \partial_{t}u_{\eps} \sigma(-\nabla u_{\eps}) \zeta\,dxdt\\
&=- \cnght\volint \chi\partial_{t}\zeta \, dxdt +\cnght \int_{\Omega}\chi(x,T)\zeta(x,T)\,dx - \cnght\int_{\Omega}\chi_{0}(x)\zeta(x,0)\,dx.
\end{align*}
The distributional criterion (\ref{velocity-criterion}) for the normal velocity yields (\ref{velocity-xi-convergence}), and so it follows that the third line on the right-hand side of (\ref{velocity-add-zero}) converges to zero.

It remains to estimate the ``error'' terms in (\ref{velocity-add-zero}). For the first integral on the right-hand side, one uses the asymptotic description of $g$ provided by Lemma \ref{g-properties}(iii): For every $\delta > 0$, there exists a positive constant $R$ such that $\left|g(p)-\frac{\sigma(p)}{\mu(p)}\right|<\delta$ whenever $|p|\geq R$. Therefore, we obtain
\begin{align}
\left|\volint \vphantom{\left(g(-\nabla u_{\eps})-\frac{\sigma(\nu_{\eps})}{\mu(\phn)}\right)\partial_{t}u_{\eps} B\cdot \eps \nabla u_{\eps}\,dx dt}\right.&\left.\vphantom{\volint}\left(g(-\nabla u_{\eps})-\frac{\sigma(\nu_{\eps})}{\mu(\phn)}\right)\partial_{t}u_{\eps} B\cdot \eps \nabla u_{\eps}\,dx dt\right| \nonumber \\
&\leq\volint \chi_{\{|\nabla u_{\eps}|\leq R \}}\left|g(-\nabla u_{\eps})-\frac{\sigma(\nu_{\eps})}{\mu(\phn)}\right||\partial_{t}u_{\eps}|\left| B \cdot \eps \nabla u_{\eps}\right|dx dt \nonumber \\
&\quad + \volint \chi_{\{|\nabla u_{\eps}|> R\}} \left|g(-\nabla u_{\eps})-\frac{\sigma(\nu_{\eps})}{\mu(\phn)}\right||\partial_{t}u_{\eps}| \left|B\cdot \eps \nabla u_{\eps}\right|dx dt \nonumber \\
&\leq 2R\,(\sup_{\real^{d}}g )\|B\|_{L^{\infty}}\volint \eps|\partial_{t}u_{\eps}|\,dx dt \nonumber \\
&\quad + \delta \|B\|_{L^{\infty}} \volint |\partial_{t}u_{\eps}| \eps|\nabla u_{\eps}|\,dx dt \nonumber \\
&\leq 2R\,(\sup_{\real^{d}}g )\|B\|_{L^{\infty}}\sqrt{\eps T}\left(\volint \eps(\partial_{t}u_{\eps})^{2}dx dt\right)^{\frac{1}{2}}  \nonumber \\
&\quad + \delta \|B\|_{L^{\infty}} \left(\volint \eps(\partial_{t}u_{\eps})^{2} dx dt\right)^{\frac{1}{2}}\left(\volint  \eps|\nabla u_{\eps}|^{2}dx dt\right)^{\frac{1}{2}} \nonumber \\
&\leq \sqrt{\eps}\frac{2R\sqrt{T}\,\sup_{\real^{d}}g }{\sqrt{c_{g}}}\|B\|_{L^{\infty}}E_{\eps}[u_{\eps,0}]^{\frac{1}{2}} \nonumber \\
&\quad + \delta \frac{\sqrt{2}}{\sqrt{c_{g}}\min_{S^{d-1}}\sigma}\|B\|_{L^{\infty}} E_{\eps}[u_{\eps,0}]^{\frac{1}{2}}\left(\int_{0}^{T}E_{\eps}[u_{\eps}(t)] dt\right)^{\frac{1}{2}}, \label{eqn:ErrEstimate}
\end{align}
from which it follows that 
\begin{align}\label{velocity-g-to-mu}
\limsup_{\eps \searrow 0}\left|\volint \vphantom{\left(g(-\nabla u_{\eps})-\frac{\sigma(\nu_{\eps})}{\mu(\phn)}\right)\partial_{t}u_{\eps} B\cdot \eps \nabla u_{\eps}\,dx dt}\right.&\left.\vphantom{\volint}\left(g(-\nabla u_{\eps})-\frac{\sigma(\nu_{\eps})}{\mu(\phn)}\right)\partial_{t}u_{\eps} B\cdot \eps \nabla u_{\eps}\,dx dt\right|\nonumber\\
&\leq \delta C \|B\|_{L^{\infty}}E_{0}^{\frac{1}{2}}\left(\int_{0}^{T}E[u(t)]dt\right)^{\frac{1}{2}}.
\end{align}
Since $\delta>0$ is arbitrary, we have shown that the first integral on the right-hand side of (\ref{velocity-add-zero}) converges to zero as $\eps \searrow 0$.

As for the second integral on the right-hand side, we use that the map $p \mapsto \frac{\psi(|p|)}{\mu(p)}p$ is Lipschitz continuous on $\real^{d}$, so that
\begin{align*}
\left|\volint\vphantom{\partial_{t}u_{\eps} B\cdot \eps \sigma(-\nabla u_{\eps})\left(\frac{1}{\mu(\phn)}\phn-\frac{\psi(|\xi|)}{\mu(\xi)}\xi\right)dx dt} \right.&\left.\vphantom{\volint}\partial_{t}u_{\eps} B\cdot \eps \sigma(-\nabla u_{\eps})\left(\frac{1}{\mu(\phn)}\phn-\frac{\psi(|\xi|)}{\mu(\xi)}\xi\right)dx dt\right|\\
&\leq C\|B\|_{L^{\infty}}\volint|\partial_{t}u_{\eps}| \eps \sigma(-\nabla u_{\eps})|\nu_{\eps}-\xi|dx dt\\
&\leq C\|B\|_{L^{\infty}}\left(\volint\eps(\partial_{t}u_{\eps})^{2}dx dt\right)^{\frac{1}{2}}\left(\volint\eps f(-\nabla u_{\eps})|\nu_{\eps}-\xi|^{2}dx dt\right)^{\frac{1}{2}}\\
&\leq \frac{C}{\sqrt{c_{g}}}\|B\|_{L^{\infty}}E_{\eps}[u_{\eps,0}]^{\frac{1}{2}}\left(\volint\eps f(-\nabla u_{\eps})|\nu_{\eps}-\xi|^{2}dx dt\right)^{\frac{1}{2}}.
\end{align*}
It follows with the help of Lemma \ref{control-by-epsre} that
\begin{align}\label{velocity-epsre}
\limsup_{\eps\searrow 0}\left|\volint\vphantom{\partial_{t}u_{\eps} B\cdot \eps \sigma(-\nabla u_{\eps})\left(\frac{1}{\mu(\phn)}\phn-\frac{\psi(|\xi|)}{\mu(\xi)}\xi\right)dx dt} \right.&\left.\vphantom{\volint}\partial_{t}u_{\eps} B\cdot \eps \sigma(-\nabla u_{\eps})\left(\frac{1}{\mu(\phn)}\phn-\frac{\psi(|\xi|)}{\mu(\xi)}\xi\right)dx dt\right|\nonumber\\
&\leq C\|B\|_{L^{\infty}}E_{0}^{\frac{1}{2}}\left(\int_{0}^{T}\re{u(t)}{\xi(t)}dt\right)^{\frac{1}{2}}.
\end{align}

Again making use of the Lipschitz continuity and Lemma \ref{dziuk}(i), the last integral on the right-hand side of (\ref{velocity-add-zero}) can be estimated as
\begin{align}\label{velocity-re}
\left|\cnght \volint\vphantom{V B \cdot \left(\frac{1}{\mu(\nu)}\nu- \frac{\psi(|\xi|)}{\mu(\xi)}\xi\right)\ggtv}\right.&\left.\vphantom{\cnght \volint} V B \cdot \left(\frac{1}{\mu(\nu)}\nu- \frac{\psi(|\xi|)}{\mu(\xi)}\xi\right)\ggtv\right|\nonumber\\
&\leq \cnght C \|B\|_{L^{\infty}}\volint |V||\nu-\xi|\ggtv\nonumber\\
&\leq \cnght C \|B\|_{L^{\infty}}\left(\volint V^{2}\ggtv\right)^{\frac{1}{2}}\left(\volint |\nu-\xi|^{2}\ggtv\right)^{\frac{1}{2}}\nonumber\\
&\leq  \frac{C\sqrt{\cnght}}{\sqrt{c_{\sigma}}} \|B\|_{L^{\infty}}\|V\|_{L^{2}(|\nabla \chi|)}\left(\int_{0}^{T}\re{u(t)}{\xi(t)} dt\right)^{\frac{1}{2}}.
\end{align}

Plugging in the results for each right-hand side integral into (\ref{velocity-add-zero}) leads to
\begin{align*}
\limsup_{\eps \searrow 0}\left|\volint\vphantom{g(-\nabla u_{\eps})\partial_{t}u_{\eps} B\cdot \eps \nabla u_{\eps}\,dx dt+ \cnght \volint \frac{1}{\mu(\nu)}VB\cdot \nu\ggtv}\right. & \left.\vphantom{\volint}g(-\nabla u_{\eps})\partial_{t}u_{\eps} B\cdot \eps \nabla u_{\eps}\,dx dt+ \cnght \volint \frac{1}{\mu(\nu)}VB\cdot \nu\ggtv\right|\\
&\leq C \|B\|_{L^{\infty}}\left(E_{0}^{\frac{1}{2}}+\|V\|_{L^{2}(|\nabla \chi|)}\right)\left(\int_{0}^{T}\re{u(t)}{\xi(t)} dt\right)^{\frac{1}{2}}.
\end{align*}
The proof of Theorem \ref{velocity-thm} is complete once we take into account that, by Lemma \ref{smooth-xi-approx}, the time-integrated relative entropy can be made arbitrarily small for suitable choices of vector fields $\xi$.

\subsection{Optimal energy dissipation inequality}

\begin{lem}\label{oed-lem}
	Under the assumptions of Theorem \ref{sil-thm}, including the energy convergence (\ref{energy-conv}), the inequality
	\begin{align*}
	E[u(\horiz)]+\cnght\int_{0}^{\horiz}\int_{\Omega}\frac{1}{\mu(\nu)}V^{2}\ggtv \leq E_{0}
	\end{align*}
	holds true for all $\horiz \in [0,T]$.
\end{lem}

The strategy to prove this lemma is to take $\eps \searrow 0$ in the optimal energy dissipation identity (\ref{aac-oed}) for the anisotropic Allen--Cahn equation. By the well-preparedness of the initial data (\ref{init-conv}), we have $\lim_{\eps \searrow 0}E_{\eps}[u_{\eps,0}]=E_{0}$. Furthermore, the $\Gamma$-convergence $E_{\eps}\overset{\Gamma}{\longrightarrow}E$ on $L^{1}(\Omega)$ (see \cite[Theorem 3.5]{Bouchitte}) yields $\liminf_{\eps \searrow 0}E_{\eps}[u_{\eps}(\horiz)]\geq E[u(\horiz)]$ since it follows from the conditions in Lemma \ref{compact-lem} that $u_{\eps}(t)\to u(t)$ in $L^{2}(\Omega)$ for all $t \in [0,T]$.

We remark that Bouchitt\'{e} \cite{Bouchitte} defines $\dom(E_{\eps})=\Lip(\Omega)$, so that the liminf inequality for our definition in (\ref{cahnhilliard}) does not immediately follow. However, the liminf inequality is the easier part of the $\Gamma$-convergence statement and can be shown by combining the Modica--Mortola argument in Theorem \ref{equi-thm}(i) above with the truncation of $W$ that can be found in \cite[Proof of Theorem 1.6]{LeoniMM}.

Therefore, the proof of the lemma reduces to proving the following inequality:
\begin{clm*}
	For every $\horiz \in [0,T]$, we have 
	\begin{equation}\label{oed-liminf-formulation}
	\liminf_{\eps \searrow 0}\int_{0}^{\horiz}\int_{\Omega}\eps g(-\nabla u_{\eps})(\partial_{t}u_{\eps})^{2}dxdt \geq \cnght \int_{0}^{\horiz}\int_{\Omega}\frac{1}{\mu(\nu)}V^{2}\ggtv.
	\end{equation}
\end{clm*}

Let $\zeta \in C(\Omega\times[0,T])$ and $\xi \in C(\Omega\times[0,T])^{d}$ such that $|\xi|\leq 1$ in $\Omega\times[0,T]$. We apply Young's inequality in the form $a^{2}\geq 2ab-b^{2}$ and add zero multiple times in order to replace $g$ with its asymptotic version and $\phn$ with $\xi$, similarly to the approach in (\ref{velocity-add-zero}). In this case, we obtain
\begin{align}\label{oed-add-zero}
\int_{0}^{\horiz}\int_{\Omega}&\eps g(-\nabla u_{\eps})(\partial_{t}u_{\eps})^{2}dxdt\nonumber \\
&\geq 2\int_{0}^{\horiz}\int_{\Omega}\eps g(-\nabla u_{\eps})\partial_{t}u_{\eps}\sigma(-\nabla u_{\eps})\zeta\,dxdt-\int_{0}^{\horiz}\int_{\Omega} \eps g(-\nabla u_{\eps})f(-\nabla u_{\eps})\zeta^{2}dxdt\nonumber \\
&=2\int_{0}^{\horiz}\int_{\Omega} \eps\left(g(-\nabla u_{\eps})-\frac{\sigma(\phn)}{\mu(\phn)}\right)\partial_{t}u_{\eps}\sigma(-\nabla u_{\eps})\zeta\,dxdt\nonumber\\
&\quad -\int_{0}^{\horiz}\int_{\Omega}\eps \left(g(-\nabla u_{\eps})-\frac{\sigma(\phn)}{\mu(\phn)}\right)f(-\nabla u_{\eps})\zeta^{2}dxdt\nonumber\\
&\quad +2\int_{0}^{\horiz}\int_{\Omega}\eps \left(\frac{\sigma(\phn)}{\mu(\phn)}-\psi(|\xi|)\frac{\sigma(\xi)}{\mu(\xi)}\right)\partial_{t}u_{\eps}\sigma(-\nabla u_{\eps})\zeta\,dxdt\nonumber\\
&\quad -\int_{0}^{\horiz}\int_{\Omega}\eps \left(\frac{\sigma(\phn)}{\mu(\phn)}-\psi(|\xi|)\frac{\sigma(\xi)}{\mu(\xi)}\right)f(-\nabla u_{\eps})\zeta^{2}dxdt\nonumber\\
&\quad + 2\int_{0}^{\horiz}\int_{\Omega}\eps\psi(|\xi|)\frac{\sigma(\xi)}{\mu(\xi)}\partial_{t}u_{\eps}\sigma(-\nabla u_{\eps})\zeta\,dxdt\nonumber\\
&\quad-\int_{0}^{\horiz}\int_{\Omega}\eps \psi(|\xi|)\frac{\sigma(\xi)}{\mu(\xi)}f(-\nabla u_{\eps})\zeta^{2}dxdt.
\end{align}

To show that the first integral on the right-hand side converges to zero, let $\delta >0$. By Lemma \ref{g-properties}(iii), there exists some $R>0$ such that $\left|g(p)-\frac{\sigma(p)}{\mu(p)}\right|<\delta$ for all $p \in \real^{d}$ with $|p|\geq R$. Arguing as in (\ref{eqn:ErrEstimate}), we find
\begin{align*}
\left|2\int_{0}^{\horiz}\int_{\Omega} \vphantom{\eps\left(g(-\nabla u_{\eps})-\frac{\sigma(\phn)}{\mu(\phn)}\right)\partial_{t}u_{\eps}\sigma(-\nabla u_{\eps})\zeta\,dxdt}\right.&\left.\vphantom{2\int_{0}^{\horiz}\int_{\Omega}}\eps\left(g(-\nabla u_{\eps})-\frac{\sigma(\phn)}{\mu(\phn)}\right)\partial_{t}u_{\eps}\sigma(-\nabla u_{\eps})\zeta\,dxdt\right|\\
&\leq 4R\, \frac{(\max_{S^{d-1}}\sigma)(\sup_{\real^{d}}g)}{\sqrt{c_{g}}}\|\zeta\|_{L^{\infty}}\sqrt{\eps T}\,E_{\eps}[u_{\eps,0}]^{\frac{1}{2}}\\
&\quad+2\sqrt{2}\,\delta\, \frac{\max_{S^{d-1}}\sigma}{\sqrt{c_{g}}\,\min_{S^{d-1}}\sigma}\,\|\zeta\|_{L^{\infty}}{E_{\eps}[u_{\eps,0}]}^{\frac{1}{2}}\left(\int_{0}^{\horiz}E_{\eps}[u_{\eps}(t)]dt\right)^{\frac{1}{2}}.
\end{align*}
Taking first $\eps \searrow 0$, then $\delta \searrow 0$ leads to
\begin{align}\label{oed-error-1}
\lim_{\eps \searrow 0}\left|2\int_{0}^{\horiz}\int_{\Omega} \eps\left(g(-\nabla u_{\eps})-\frac{\sigma(\phn)}{\mu(\phn)}\right)\partial_{t}u_{\eps}\sigma(-\nabla u_{\eps})\zeta\,dxdt\right|=0.
\end{align}

For the second line on the right-hand side of (\ref{oed-add-zero}), we argue analogously and obtain
\begin{align}\label{oed-error-1a}
\lim_{\eps \searrow 0}\left|-\int_{0}^{\horiz}\int_{\Omega}\eps \left(g(-\nabla u_{\eps})-\frac{\sigma(\phn)}{\mu(\phn)}\right)f(-\nabla u_{\eps})\zeta^{2}dxdt\right|=0.
\end{align}

Due to the Lipschitz continuity of the map $p \mapsto \psi(|p|)\frac{\sigma(p)}{\mu(p)}$, the third integral can be estimated as
\begin{align*}
\left|2\int_{0}^{\horiz}\int_{\Omega}\vphantom{\eps \left(\frac{\sigma(\phn)}{\mu(\phn)}-\psi(|\xi|)\frac{\sigma(\xi)}{\mu(\xi)}\right)\partial_{t}u_{\eps}\sigma(-\nabla u_{\eps})\zeta\,dxdt}\right.&\left.\vphantom{2\int_{0}^{\horiz}\int_{\Omega}}\eps \left(\frac{\sigma(\phn)}{\mu(\phn)}-\psi(|\xi|)\frac{\sigma(\xi)}{\mu(\xi)}\right)\partial_{t}u_{\eps}\sigma(-\nabla u_{\eps})\zeta\,dxdt\right|\\
&\leq 2 C\|\zeta\|_{L^{\infty}} \int_{0}^{\horiz}\int_{\Omega}\eps |\phn-\xi||\partial_{t}u_{\eps}|\sigma(-\nabla u_{\eps})\, dxdt\\
&\leq 2 C\|\zeta\|_{L^{\infty}} \left(\int_{0}^{\horiz}\int_{\Omega}\eps (\partial_{t}u_{\eps})^{2} dxdt\right)^{\frac{1}{2}}\left(\int_{0}^{\horiz}\int_{\Omega}\eps |\phn-\xi|^{2}f(-\nabla u_{\eps})\, dxdt\right)^{\frac{1}{2}}\\
&\leq 2 \frac{C}{\sqrt{c_{g}}}\|\zeta\|_{L^{\infty}} E_{\eps}[u_{\eps,0}]^{\frac{1}{2}}\left(\int_{0}^{\horiz}\int_{\Omega}\eps |\phn-\xi|^{2}f(-\nabla u_{\eps})\, dxdt\right)^{\frac{1}{2}},
\end{align*}
so that, by appealing to Lemma \ref{control-by-epsre}, we obtain
\begin{align}\label{oed-error-2}
\limsup_{\eps\searrow 0}\left|2\int_{0}^{\horiz}\int_{\Omega}\vphantom{\eps \left(\frac{\sigma(\phn)}{\mu(\phn)}-\psi(|\xi|)\frac{\sigma(\xi)}{\mu(\xi)}\right)\partial_{t}u_{\eps}\sigma(-\nabla u_{\eps})\zeta\,dxdt}\right.&\left.\vphantom{2\int_{0}^{\horiz}\int_{\Omega}}\eps \left(\frac{\sigma(\phn)}{\mu(\phn)}-\psi(|\xi|)\frac{\sigma(\xi)}{\mu(\xi)}\right)\partial_{t}u_{\eps}\sigma(-\nabla u_{\eps})\zeta\,dxdt\right|\nonumber\\&\leq C \|\zeta\|_{L^{\infty}}E_{0}^{\frac{1}{2}}\left(\int_{0}^{T}\re{u(t)}{\xi(t)}dt\right)^{\frac{1}{2}}.
\end{align}

For the fourth line of the right-hand side of (\ref{oed-add-zero}), a very similar computation yields
\begin{align}\label{oed-error-2a}
\limsup_{\eps\searrow 0}\left|-\int_{0}^{\horiz}\int_{\Omega}\vphantom{\eps \left(\frac{\sigma(\phn)}{\mu(\phn)}-\psi(|\xi|)\frac{\sigma(\xi)}{\mu(\xi)}\right)f(-\nabla u_{\eps})\zeta^{2}dxdt}\right.&\left.\vphantom{-\int_{0}^{\horiz}\int_{\Omega}}\eps \left(\frac{\sigma(\phn)}{\mu(\phn)}-\psi(|\xi|)\frac{\sigma(\xi)}{\mu(\xi)}\right)f(-\nabla u_{\eps})\zeta^{2}dxdt\right|\nonumber\\&\leq C\|\zeta\|_{L^{\infty}}^{2}\left(\int_{0}^{T}E[u(t)]dt\right)^{\frac{1}{2}}\left(\int_{0}^{T}\re{u(t)}{\xi(t)}dt\right)^{\frac{1}{2}}.
\end{align}

Finally, it follows from (\ref{velocity-xi-convergence}) and the equipartition of energy, Theorem \ref{equi-thm}(v), that
\begin{align}\label{oed-limit}
\lim_{\eps \searrow 0}\left( 2\int_{0}^{\horiz}\int_{\Omega}\vphantom{\eps \psi(|\xi|)\frac{\sigma(\xi)}{\mu(\xi)}\partial_{t}u_{\eps}\sigma(-\nabla u_{\eps})\zeta\,dxdt-\int_{0}^{\horiz}\int_{\Omega}\eps \psi(|\xi|)\frac{\sigma(\xi)}{\mu(\xi)}f(-\nabla u_{\eps})\zeta^{2}dxdt}\right.&\left.\vphantom{2\int_{0}^{\horiz}\int_{\Omega}}\eps \psi(|\xi|)\frac{\sigma(\xi)}{\mu(\xi)}\partial_{t}u_{\eps}\sigma(-\nabla u_{\eps})\zeta\,dxdt-\int_{0}^{\horiz}\int_{\Omega}\eps \psi(|\xi|)\frac{\sigma(\xi)}{\mu(\xi)}f(-\nabla u_{\eps})\zeta^{2}dxdt\right)\nonumber\\
&=2\cnght \int_{0}^{\horiz}\int_{\Omega} \psi(|\xi|)\frac{\sigma(\xi)}{\mu(\xi)}V\zeta \ggtv - \cnght\int_{0}^{\horiz}\int_{\Omega}\sigma(\nu)\psi(|\xi|)\frac{\sigma(\xi)}{\mu(\xi)}\zeta^{2}\ggtv.
\end{align}

It is now desirable to let $\xi \to \nu$ in $L^{2}(\ggtv)$ so that $\int_{0}^{T}\re{u(t)}{\xi(t)}dt \to 0$. This is possible by Lemma \ref{smooth-xi-approx} and Lemma \ref{dziuk}(i). By the dominated convergence theorem, we also have $\psi(|\xi|)\frac{\sigma(\xi)}{\mu(\xi)}\to \frac{\sigma(\nu)}{\mu(\nu)}$ in $L^{2}(\ggtv)$. Under this convergence, it follows from (\ref{oed-add-zero}) and the computations (\ref{oed-error-1})--(\ref{oed-limit}) that
\begin{align*}
\liminf_{\eps \searrow 0}\int_{0}^{\horiz}\int_{\Omega}&\eps g(-\nabla u_{\eps})(\partial_{t}u_{\eps})^{2}dxdt\\
&\geq 2\cnght \int_{0}^{\horiz}\int_{\Omega} \frac{\sigma(\nu)}{\mu(\nu)}V\zeta \ggtv - \cnght\int_{0}^{\horiz}\int_{\Omega}\sigma(\nu)\frac{\sigma(\nu)}{\mu(\nu)}\zeta^{2}\ggtv.
\end{align*}

In a last step we let $\zeta \to \frac{V}{\sigma(\nu)}$ in $L^{2}(\ggtv)$, which yields the lower bound as stated in the claim.

\subsection{Proof of Theorem \ref{sil-thm}}

We simply wrap up the proof of the theorem.

\begin{proof}
The compactness follows from Lemma \ref{compact-lem}. Under the energy convergence assumption, the distributional formulation for the time derivative (\ref{velocity-criterion}) follows from Lemma \ref{vel-lem}. The optimal energy dissipation relation (\ref{amcf-oed}) follows from Lemma \ref{oed-lem}. To obtain the distributional formulation for the curvature in (\ref{amcf-df}), we use $B\cdot \eps \nabla u_{\eps}$ as a test function in the distributional formulation of the anisotropic Allen--Cahn equation (\ref{aac-df}), where $B \in C^{1}(\Omega\times[0,T])^{d}$, where we recall that by Theorem \ref{aac-regular}, we have $B\cdot \eps \nabla u_{\eps} \in L^{2}\left(0,T;H^{1}(\Omega)\right)$, and that these functions are admissible test functions in (\ref{aac-df}). To pass to the limit as $\eps \searrow 0$ in the left hand-side, one applies Theorem \ref{velocity-thm}, and likewise for the right-hand side, apply Theorem \ref{curvature-thm}.
\end{proof}
 
\section{Weak-strong uniqueness for anisotropic mean curvature flow}\label{sec:wsu}
The goal of this section is to prove that, as long as a strong solution to anisotropic mean curvature flow (\ref{strong-amcf}) exists, any $BV$ solution with the same initial data coincides with the strong solution. One needs to require additional regularity for $(\sigma,\mu)$ to make sure that strong solutions will be sufficiently smooth, and we will also rely on the higher regularity of $\sigma$ in the proof of the weak-strong uniqueness statement. Thus, we assume that $\sigma,\mu \in C^{\infty}\big(\real^{d}\setminus\{0\}\big)$. Further, without loss of generality, we let $c_0 = 1.$ Following Hensel and Moser \cite[Definition 10]{HenselMoser}, we define strong solutions to (\ref{strong-amcf}) as follows:

\begin{dfn}
	Let $T>0$ be a finite time horizon. A family $\{\scA(t)\}_{t\in [0,T]}$ of open subsets of $\Omega$ is a strong solution to anisotropic mean curvature flow if 
	\begin{itemize}
		\item $\partial\scA(0)$ is an embedded $C^{\infty}$-submanifold,
		\item there exists a $C^{\infty}$-map $\Phi:\Omega\times[0,T]\to \Omega$ such that $\Phi(\cdot,t)$ is a diffeomorphism for all $t\in [0,T]$, furthermore $\Phi(\cdot,0)=\id_{\Omega}$ and 
		\begin{equation*}
		\Phi(\scA(0),t)=\scA(t), \qquad \Phi(\partial \scA(0),t)=\partial\scA(t)
		\end{equation*}
		for all $t \in [0,T]$, and
		\item $\{\partial\scA(t) \}_{t\in [0,T]}$ evolves by (\ref{strong-amcf}) in the classical sense.
	\end{itemize}
\end{dfn}

The setup in \cite[Definition 10 and Remark 15]{HenselMoser} also suggests that the $C^{\infty}$-regularity for $(\sigma,\mu)$, $\Phi$, and $\partial \scA(0)$ can be relaxed.

We are now ready to formulate the central theorem of this section: 
\begin{thm}\label{wsu}
Let $\{\mathscr{A}(t)\}_{t \in [0,T]}$ be a solution of anisotropic mean curvature flow (\ref{strong-amcf}), and further, let $\{A(t)\}_{t\in [0,T]}$ be time parametrized collection of sets with $\chi:=\chi_{A}$ a distributional solution of anisotropic mean curvature flow as in Definition \ref{amcf-sol}. If
\begin{equation*}
\left| \mathscr{A}(0)\triangle A(0)\right| = 0,
\end{equation*}
then 
\begin{equation*}
\left| \mathscr{A}(t)\triangle A(t)\right| = 0
\end{equation*}
for all $t \in [0,T]$.
\end{thm}

The proof for this theorem is modeled after \cite[Sections 2.2, 4]{HenselLaux}, where Hensel and the first author derive an analogous result for multiphase isotropic mean curvature flow. The key step in this argument is to find a gradient flow calibration (see below) for the smooth evolution $\left\{\mathscr{A}(t)\right\}_{t \in [0,T]}$. While the existence of a gradient flow calibration is nontrivial in the multiphase case (see \cite[Theorem 4]{HenselMoser} for a gradient flow calibration for multiphase mean curvature flow in $d=2$ with constant contact angle), such a calibration can always be constructed explicitly for a smooth two-phase evolution.

\subsection{Gradient flow calibrations}

In the remainder of this chapter, $C <\infty$ and $c>0$ denote positive constants (`large' and `small', respectively) that may depend on the pair of anisotropies $(\sigma,\mu)$, on the time horizon $T$, and on the smooth evolution $\left\{\mathscr{A}(t)\right\}_{t \in [0,T]}$. These constants need not be the same on every occurrence.

\begin{dfn}\label{calibrations}
	Let $\left\{\scA(t) \right\}_{t\in [0,T]}$ be a strong solution to anisotropic mean curvature flow consisting of nonempty open proper subsets of $\Omega$. A gradient flow calibration for $\left\{\scA(t) \right\}_{t\in [0,T]}$ is a triple
	\begin{equation*}
	(\xi, B, \tha) \in C_{1}^{2}(\Omega\times[0,T])^{d}\times C\big([0,T];C^{2}(\Omega)\big)^{d} \times C^{1}(\Omega\times[0,T];[-1,1])
	\end{equation*}
	satisfying
	\begin{itemize}
		\item the approximate evolution equations
		\begin{align}
		\left| \partial_{t}\xi + (B\cdot \nabla)\xi + (\nabla B)\trans\xi \right|(x, t) &\leq C \dist(x, \partial \scA(t))&\text{in } \Omega\times[0,T],\label{cal1}\\ 
		\left| \xi \cdot \left(\partial_{t}\xi + (B\cdot \nabla)\xi\right)\right|(x, t) &\leq C \dist^{2}(x, \partial \scA(t))&\text{in } \Omega\times[0,T],\label{cal2}\\
		\left| \partial_{t}\tha + (B\cdot \nabla)\tha\right|(x, t) &\leq C\dist(x, \partial \scA(t))&\text{in } \Omega\times[0,T], \label{cal3}
		\end{align}
		\item the compatibility condition
		\begin{equation}\label{cal4}
		\left| B \cdot \xi + \mu(\xi)\dvg(|\xi|\psi(|\xi|)D \sigma(\xi))\right|(x, t)\leq C \dist(x, \partial \scA(t))\quad \text{in } \Omega\times[0,T],
		\end{equation}
		\item and with $\nu_{\partial\scA(t)}$ denoting the outer unit normal of the set $\scA(t)$, we have the coercivity conditions
		\begin{align}
		\xi(x,t)&=\nu_{\partial \scA(t)}(x)&&\text{on } \bigcup_{t\in[0,T]}\left(\partial\scA(t)\times\{t\}\right),\label{cal5}\\
		|\xi(x, t)|&\leq 1-c \dist^{2}(x, \partial \scA(t)) &&\text{in } \Omega\times[0,T],\label{cal6}\\
		\tha(x,t)&>c \dist(x, \partial \scA(t))&&\text{in } \bigcup_{t\in[0,T]}\left(\overline{\scA(t)}^{c}\times\{t\}\right),\label{cal7}\\
		\tha(x,t)&<-c\dist(x, \partial \scA(t)) &&\text{in } \bigcup_{t\in[0,T]}\left(\scA(t)\times\{t\}\right).\label{cal8}
		\end{align}
	\end{itemize}
\end{dfn}

Intuitively, $\xi$ is an extension of the outer unit normal $\nu_{\partial\scA(t)}$ (with an additional coercivity property), whereas $B$ extends the normal velocity vector and $\tha$ is comparable to a signed distance function. The compatibility condition (\ref{cal4}) encodes the motion by anisotropic mean curvature. The space $C_{1}^{2}(\Omega\times[0,T])$ is defined as
\begin{equation*}
C_{1}^{2}(\Omega\times[0,T]):=\left\{f\in C(\Omega\times[0,T])\,\bigg|\,\partial_{t}f,\, \nabla f, \, D^{2}f \text{ are continuous on } \Omega\times[0,T]  \right\}.
\end{equation*} 

Let us collect the key inequalities for gradient flow calibrations that will be used in the proof of Theorem \ref{wsu}.
\begin{lem}\label{calibration-lem}
	Let $(\xi, B, \tha)$ be a gradient flow calibration for $\left\{\scA(t) \right\}_{t\in [0,T]}$. Then the following estimates hold true for all $(x,t)\in \Omega\times [0,T]$:
	\begin{enumerate}[(i)]
		\item 
		\begin{equation*}
		c\dist^{2}(x, \partial\scA(t)) \leq	1-\left|\xi\right|(x,t)\leq C \dist^{2}(x, \partial \scA(t)),
		\end{equation*}
		\item
		\begin{equation*}
		c \dist^{2}(x, \partial \scA(t)) \leq	1-\left|\xi\right|^{2}(x,t)\leq C \dist^{2}(x, \partial \scA(t)), 
		\end{equation*}
		\item 
		\begin{equation*}
		\left|\xi \cdot(\xi \cdot \nabla)B\right|(x,t)\leq C \dist(x, \partial \scA(t)).
		\end{equation*}
	\end{enumerate}
	If $\nu: \Omega\times[0,T]\to \real^{d}$ such that $|\nu|\equiv 1$, then
	\begin{enumerate}[(i)]
		\setcounter{enumi}{3}
		\item
		\begin{equation*}
		|\nu-\xi|^{2}\leq C\left(\sigma(\nu)-|\xi|\psi(|\xi|)D\sigma(\xi)\cdot\nu\right),
		\end{equation*}
		\item
		\begin{equation*}
		1-|\xi|\leq C\left(\sigma(\nu)-|\xi|\psi(|\xi|)D\sigma(\xi)\cdot\nu\right), \text{ and}
		\end{equation*}
		\item
		\begin{equation*}
		\left|\xi \cdot(\nu-\xi)\right|\leq C\left(\sigma(\nu) - |\xi|\psi(|\xi|)D\sigma(\xi)\cdot \nu \right).
		\end{equation*} 
	\end{enumerate}
\end{lem}

\begin{proof}
	\begin{enumerate}[(i)]
		\item The lower bound is precisely the coercivity condition (\ref{cal6}).

		For the upper bound, we observe that, due to (\ref{cal5}) and (\ref{cal6}), $|\xi|$ attains its maximum on $\partial \scA(t)$ for every $t \in [0,T]$. In particular, we have $\nabla |\xi|=0$ on $\partial \scA(t)$. A second-order Taylor expansion yields $1-|\xi|(x,t)\leq C\dist^{2}(x,\partial\scA(t))$. For this argument, we use the smoothness of $\xi$ (i.e., the continuity of $\nabla \xi$ and $D^{2}\xi$ on $\Omega\times [0,T]$) and the compactness of the domain $\Omega\times [0,T]$.
		\item This follows immediately from (i) since the condition $|\xi|\leq 1$ allows us to compute
		\begin{equation*}
		1-|\xi|\leq 1-|\xi|^{2}=(1-|\xi|)(1+|\xi|)\leq 2(1-|\xi|),
		\end{equation*}
		i.e., $1-|\xi|^{2}$ is bounded from below and above by a rescaled $1-|\xi|$.
		\item This inequality is a consequence of (\ref{cal1}) and (\ref{cal2}): We have
		\begin{align*}
		\left|\xi \cdot (\xi \cdot \nabla)B \right|(x,t) &= \left| \xi \cdot\left(\partial_{t}\xi + (B\cdot \nabla)\xi + (\nabla B)\trans\xi\right) - \xi \cdot \left(\partial_{t}\xi + (B\cdot \nabla)\xi\right)\right|(x,t) \\&\leq C\dist(x, \partial \scA(t))+C\dist^{2}(x, \partial \scA(t))
		\\&\leq C\dist(x, \partial \scA(t)).
		\end{align*}
		\item This is the statement of Lemma \ref{dziuk}(i).
		\item This is another elementary estimate: With the help of Lemma \ref{surface-tensions}(ii), (vi), one computes
		\begin{align*}
		\sigma(\nu)-|\xi|\psi(|\xi|)D\sigma(\xi)\cdot \nu&\geq \sigma(\nu)-|\xi|\psi(|\xi|)\sigma( \nu)\\
		&\geq (\min_{S^{d-1}}\sigma)\,(1-|\xi|\psi(|\xi|))\\
		&\geq (\min_{S^{d-1}}\sigma)\,(1-|\xi|).
		\end{align*}
		\item follows from the previous two estimates since
		\begin{equation*}
		\xi \cdot (\nu-\xi)=-\frac{1}{2}\left(1-2\,\nu\cdot \xi +|\xi|^{2}\right) + \frac{1}{2}\left(1-|\xi|^{2}\right)=-\frac{1}{2}|\nu- \xi|^{2} + \frac{1}{2}\left(1-|\xi|^{2}\right).\qedhere
		\end{equation*}
	\end{enumerate}
\end{proof}

\begin{lem}
	Every strong solution $\{\scA(t)\}_{t\in [0,T]}$ to anisotropic mean curvature flow (\ref{strong-amcf}) that consists of nonempty open proper subsets of $\Omega$ admits a gradient flow calibration $(\xi, B, \tha)$.
\end{lem}

\begin{proof}
	There exists a positive $\delta > 0$ such that, in the neighborhood 
	\begin{equation*}
	U:=\bigcup_{t \in [0,T]}\left(B_{\delta}(\partial \scA(t))\times\{t\}\right),
	\end{equation*} 
	the signed distance function $\sdist: U \to (-\delta, \delta)$ and the orthogonal projection $p: U\to \bigcup_{t \in [0,T]}\left(\partial \scA(t)\times\{t\}\right)$ with respect to $\partial \scA(t)$ are well-defined and regular (see \cite[Lemmas 14.16, 14.17]{GilbargTrudinger}). We use the sign convention that $\sdist(x,t)<0$ for $x \in \scA(t)\cap B_{\delta}(\partial \scA(t))$.

	Let $\zeta \in C_{c}^{\infty}((-\delta,\delta))$ be a cutoff function such that $\zeta(s)=1-s^{2}$ for $|s|\leq \frac{\delta}{2}$ and $s\zeta^{\prime}(s)\leq 0$ on $(-\delta,\delta)$. Furthermore, we define another truncation $f \in C^{\infty}(\real)$ such that $f(s)=s$ for $|s|\leq \frac{\delta}{2}$, $f(s)=\sgn(s)\delta$ for $|s|\geq \frac{3}{4}\delta$, and $f^{\prime}\geq 0$. Let now
	\begin{equation*}
	\xi(x,t):= \zeta(\sdist(x,t))\nabla\sdist(x,t).
	\end{equation*}
	Second, we define $\tha:= f\circ \sdist$ on $U$, which can be extended locally constantly to $\Omega\times [0,T]$.
	Third, let
	\begin{equation*}
	B(x,t)=-\mu(\xi(x,t))\dvg \left(D\sigma(\nabla \sdist(\cdot,t))\right)\big|_{p(x,t)}\xi(x,t) \quad \text{for } (x,t)\in U,
	\end{equation*}
	which can be extended by zero to $\Omega\times [0,T]$. It can be shown that $(\xi, B, \tha)$ is a gradient flow calibration for $\{\scA(t)\}_{t\in [0,T]}$ as in Definition \ref{calibrations}.
\end{proof}

\subsection{A Gr\"{o}nwall-type stability estimate}
We wish to prove Theorem \ref{wsu} by deriving a Gr\"{o}nwall-type estimate for a suitable quantity. A straightforward way to measure the difference between the calibrated evolution $\scA(t)$ and the weak solution $A(t)$ ($\chi_{A(t)}$ satisfying Definition \ref{amcf-sol}) at a given time $t \in [0,T]$ is the bulk error
\begin{equation*}
\be{\chi(t)}{\tha(t)}:=\int_{A(t)\triangle \mathscr{A}(t)}|\tha(x,t)|dx.
\end{equation*}
Since $|\tha(\cdot, t)|$ is strictly positive outside $\partial \scA(t)$, and hence almost everywhere in $\mathbb{R}^d$ this bulk error vanishes if and only if \newline$|\scA(t)\triangle A(t)|=0$. However, the available stability estimate for the bulk error is not strong enough to apply Gr\"{o}nwall's lemma immediately:
\begin{lem}\label{bulk-stability} Let $\{\mathscr{A}(t)\}_{t \in [0,T]}$ be a solution of anisotropic mean curvature flow (\ref{strong-amcf}) that is calibrated in the sense of Definition \ref{calibrations}. Further, let $\{A(t)\}_{t\in [0,T]}$ be time parametrized collection of sets with $\chi:=\chi_{A}$ a distributional solution of anisotropic mean curvature flow as in Definition \ref{amcf-sol}. Recalling the definition of $\re{\cdot}{\cdot}$ in (\ref{relentropy}) (with $c_0 = 1$), the following holds:
	\begin{enumerate}[(i)]
		\item 
		The bulk error $\be{\chi(\horiz)}{\tha(\horiz)}$ at an arbitrary time $\horiz \in [0,T]$ is given by
		\begin{align*}
		\be{\chi(\horiz)}{\tha(\horiz)}&=\be{\chi(0)}{\tha(0)}+\int_{0}^{\horiz}\int_{\Omega}\left(\chi_{A(t)}-\chi_{\mathscr{A}(t)}\right)\left(\partial_{t}\tha+(B\cdot\nabla)\tha\right)dxdt\\
		&\quad + \int_{0}^{\horiz}\int_{\Omega}\left(\chi_{A(t)}-\chi_{\mathscr{A}(t)}\right)\tha \dvg B\, dx dt \\&\quad+ \int_{0}^{\horiz}\int_{\redbd A(t)}\tha B\cdot (\xi-\nu)\,d\hausd dt \\&\quad+ \int_{0}^{\horiz}\int_{\redbd A(t)}\tha (V-B\cdot\xi)\,d\hausd dt.
		\end{align*}
		\item For every $\delta > 0$, there exists a constant $C(\delta)>0$, which also depends on the calibrated evolution, such that 
		\begin{align}\label{bulk-bad}
		\be{\chi(\horiz)}{\tha(\horiz)}&\leq \be{\chi(0)}{\tha(0)} + C(\delta) \int_{0}^{\horiz}\left(\be{\chi(t)}{\tha(t)}+\re{\chi(t)}{\xi(t)}\right)dt\nonumber \\
		&+ \delta\int_{0}^{\horiz}\int_{\redbd A(t)}|V-B\cdot \xi|^{2}d\hausd dt
		\end{align}
		for all $\horiz \in [0,T]$.
	\end{enumerate}
\end{lem}

\begin{proof}
	See \cite[Sections 4.2, 4.3]{HenselLaux}, where the argument is carried out for isotropic mean curvature flow. The same argument applies in our anisotropic setting since the definition of the bulk error functional has remained unchanged.
\end{proof}

To compensate for the additional error term on the right-hand side of (\ref{bulk-bad}), we make use of another stability estimate for the relative entropy:
\begin{lem}\label{entropy-stability} Let the hypotheses of Lemma \ref{bulk-stability} hold.
	For all $\horiz \in [0,T]$, we have
	\begin{align}\label{relen-stability}
	\re{\chi(\horiz)}{\xi(\horiz)}&\leq \re{\chi(0)}{\xi(0)}+{ C\int_{0}^{\horiz}\re{\chi(t)}{\xi(t)}dt}\nonumber\\
	&\quad - \frac{1}{4\max_{S^{d-1}}\mu}\int_{0}^{\horiz}\int_{\redbd A(t)}|V-B\cdot \xi|^{2}d\hausd dt.
	\end{align}
\end{lem}

A combination of the two stability estimates yields 
\begin{align*}
\be{\chi(\horiz)}{\tha(\horiz)}+\re{\chi(\horiz)}{\xi(\horiz)}&\leq \be{\chi(0)}{\tha(0)}+\re{\chi(0)}{\xi(0)}\\
&\quad+ C\int_{0}^{\horiz}\left(\be{\chi(t)}{\tha(t)}+\re{\chi(t)}{\xi(t)}\right)dt.
\end{align*}
In particular, it follows by Gr\"{o}nwall's lemma that $\be{\chi(\horiz)}{\tha(\horiz)}=0$ for all $\horiz \in [0,T]$ if $\be{\chi(\horiz)}{\tha(0)}+\re{\chi(\horiz)}{\xi(0)}=0$, which completes the proof of Theorem \ref{wsu}.

\subsection{Stability of the relative entropy}

The goal of this subsection is to prove Lemma \ref{entropy-stability}.

With the help of the divergence theorem and the surface energy (\ref{surface-energy}) with $c_0 = 1$, one can rewrite the relative entropy functional as
\begin{equation*}
\re{\chi(t)}{\xi(t)}=E[\chi(t)]-\int_{A(t)}\dvg\left(|\xi|\psi(|\xi|)\,D\sigma(\xi)\right)dx,
\end{equation*}
which using the distributional formulation of the time derivative from (\ref{amcf-df}) leads to
\begin{align*}
\re{\chi(\horiz)}{\xi(\horiz)}&=\re{\chi(0)}{\xi(0)}+E[A(\horiz)]-E[A(0)]\\
&\quad -\int_{0}^{\horiz}\int_{\redbd A(t)} \dvg\left(|\xi|\psi(|\xi|)\,D \sigma(\xi)\right)V d\hausd dt\\
&\quad -\int_{0}^{\horiz}\int_{\redbd A(t)} \partial_{t}\left(|\xi|\psi(|\xi|)\,D \sigma(\xi)\right)\cdot \nu d\hausd dt\\
&\leq\re{\chi(0)}{\xi(0)}- \int_{0}^{\horiz}\int_{\redbd A(t)}\frac{1}{\mu(\nu)}V^{2}d\hausd dt\\
&\quad -\int_{0}^{\horiz}\int_{\redbd A(t)} \dvg\left(|\xi|\psi(|\xi|)\,D \sigma(\xi)\right)V d\hausd dt\\
&\quad -\int_{0}^{\horiz}\int_{\redbd A(t)} \partial_{t}\left(|\xi|\psi(|\xi|)\,D \sigma(\xi)\right)\cdot \nu d\hausd dt.
\end{align*}
Note, in the the last step we used the optimal energy dissipation inequality (\ref{amcf-oed}) for the weak solution $\chi$. We introduce the notation $F(\xi)=|\xi|\psi(|\xi|)D\sigma(\xi)$ and  $M(\xi):=D_{\xi}F(\xi)=|\xi|\psi(|\xi|)\,D^{2}\sigma(\xi)+\left(\frac{\psi(|\xi|)}{|\xi|}+\psi^{\prime}(|\xi|)\right)D\sigma(\xi)\otimes \xi$, and complete squares three times to write
\begin{align}\label{relen-general}
\mathscr{E}&\hspace{-2.5pt}\left[\chi(\horiz) \big|\xi(\horiz)\right]\nonumber \\ &\leq \re{\chi(0)}{\xi(0)} + \frac{1}{2}\int_{0}^{\horiz}\int_{\redbd A(t)}\left(-\left|\frac{1}{\sqrt{\mu(\nu)}}V-\frac{1}{\sqrt{\mu(\nu)}}B\cdot \nu \right|^{2}\vphantom{-\left|\sqrt{\mu(\nu)}\dvg(F(\xi))+\frac{1}{\sqrt{\mu(\nu)}}V \right|^{2}+\left|\sqrt{\mu(\nu)}\dvg(F(\xi))+\frac{1}{\sqrt{\mu(\nu)}}B\cdot \nu\right|^{2}}\right.\nonumber\\
&\hspace{8.5pt}\left.\vphantom{-\left|\frac{1}{\sqrt{\mu(\nu)}}V-\frac{1}{\sqrt{\mu(\nu)}}B\cdot \nu \right|^{2}}-\left|\sqrt{\mu(\nu)}\dvg(F(\xi))+\frac{1}{\sqrt{\mu(\nu)}}V \right|^{2}+\left|\sqrt{\mu(\nu)}\dvg(F(\xi))+\frac{1}{\sqrt{\mu(\nu)}}B\cdot \nu\right|^{2}\right)d\hausd dt \nonumber\\
&\quad - \int_{0}^{\horiz}\int_{\redbd A(t)}\frac{1}{\mu(\nu)}VB\cdot \nu d\hausd dt - \int_{0}^{\horiz}\int_{\redbd A(t)}\dvg(F(\xi))B\cdot \nu d\hausd dt \nonumber \\
&\quad -\int_{0}^{\horiz}\int_{\redbd A(t)} \partial_{t}\left(|\xi|\psi(|\xi|)\,D \sigma(\xi)\right)\cdot \nu d\hausd dt. 
\end{align}

The second square can be trivially estimated. We will now deal with the other squared terms as well as the remaining integrals separately:

The first square will later be used to compensate for a term in the stability estimate of the bulk error. To this end, it will be useful to estimate the term with the help of the inequality $-a^{2}\leq-\frac{1}{2}(a+b)^{2}+b^{2}$ and (\ref{eqn:normalEntropyControl}):
\begin{align}\label{relen-sq1}
-\frac{1}{2}\int_{0}^{\horiz}\int_{\redbd A(t)}&
\left|\frac{1}{\sqrt{\mu(\nu)}}V-\frac{1}{\sqrt{\mu(\nu)}}B\cdot \nu \right|^{2}d\hausd dt\nonumber \\
&\leq -\frac{1}{4}\int_{0}^{\horiz}\int_{\redbd A(t)}
\left|\frac{1}{\sqrt{\mu(\nu)}}V-\frac{1}{\sqrt{\mu(\nu)}}B\cdot \xi \right|^{2}d\hausd dt\nonumber \\
&\quad + \frac{1}{2}\int_{0}^{\horiz}\int_{\redbd A(t)}
\left|\frac{1}{\sqrt{\mu(\nu)}}B\cdot (\nu-\xi) \right|^{2}d\hausd dt\nonumber \\
&\leq -\frac{1}{4\max_{S^{d-1}}\mu}\int_{0}^{\horiz}\int_{\redbd A(t)}
\left|V-B\cdot \xi \right|^{2}d\hausd dt\nonumber \\
&\quad + \frac{1}{2\min_{S^{d-1}}\mu}\|B\|_{C^{0}}^{2}\int_{0}^{\horiz}\int_{\redbd A(t)}
\left|\nu-\xi \right|^{2}d\hausd dt\nonumber \\
&\leq -\frac{1}{4\max_{S^{d-1}}\mu}\int_{0}^{\horiz}\int_{\redbd A(t)}
\left|V-B\cdot \xi \right|^{2}d\hausd dt\nonumber \\
&\quad + C\int_{0}^{\horiz}\re{\chi(t)}{\xi(t)}dt.
\end{align}

Similarly, for the third square, the inequality $a^{2}\leq 2(a-b)^{2}+2b^{2}$ (applied twice), the Lipschitz continuity of $\mu$, the compatibility condition (\ref{cal4}), and finally (\ref{eqn:normalEntropyControl}) yield
\begin{align}\label{relen-sq2}
\frac{1}{2}\int_{0}^{\horiz}\int_{\redbd A(t)}&\left|\sqrt{\mu(\nu)}\dvg(F(\xi))+\frac{1}{\sqrt{\mu(\nu)}}B\cdot \nu\right|^{2}d\hausd dt \nonumber \\
& \leq \frac{1}{2\min_{S^{d-1}}\mu}\int_{0}^{\horiz}\int_{\redbd A(t)}\left|\mu(\nu)\dvg(F(\xi))+B\cdot \nu\right|^{2}d\hausd dt \nonumber \\
& \leq \frac{1}{\min_{S^{d-1}}\mu}\int_{0}^{\horiz}\int_{\redbd A(t)}\left|\mu(\nu)\dvg(F(\xi))+B\cdot \xi\right|^{2}d\hausd dt \nonumber \\
& \quad + \frac{1}{\min_{S^{d-1}}\mu}\int_{0}^{\horiz}\int_{\redbd A(t)}\left|B\cdot (\nu-\xi)\right|^{2}d\hausd dt \nonumber \\
& \leq \frac{2}{\min_{S^{d-1}}\mu}\int_{0}^{\horiz}\int_{\redbd A(t)}\left|\mu(\xi)\dvg(F(\xi))+B\cdot \xi\right|^{2}d\hausd dt \nonumber \\
& \quad + \frac{2}{\min_{S^{d-1}}\mu}\int_{0}^{\horiz}\int_{\redbd A(t)}\left|(\mu(\nu)-\mu(\xi))M(\xi):\nabla \xi\right|^{2}d\hausd dt \nonumber \\
& \quad + \frac{1}{\min_{S^{d-1}}\mu}\int_{0}^{\horiz}\int_{\redbd A(t)}\left|B\cdot (\nu-\xi)\right|^{2}d\hausd dt \nonumber \\
&\leq C\int_{0}^{\horiz}\int_{\redbd A(t)}\min\{1, \dist^{2}(\cdot, \redbd \scA(t)) \} d\hausd dt \nonumber \\ 
&\quad + C\int_{0}^{\horiz}\int_{\redbd A(t)}|\nu-\xi|^{2}d\hausd dt \nonumber \\
&\leq C\int_{0}^{\horiz}\re{\chi(t)}{\xi(t)}dt.
\end{align}

It remains to show that 
\begin{align}\label{relen-rem-claim}
&\Bigg| \int_{0}^{\horiz}\int_{\redbd A(t)}\frac{1}{\mu(\nu)}VB\cdot \nu d\hausd dt +  \int_{0}^{\horiz}\int_{\redbd A(t)}\dvg(F(\xi))B\cdot \nu d\hausd dt\nonumber \\
&+\int_{0}^{\horiz}\int_{\redbd A(t)} \partial_{t}\left(|\xi|\psi(|\xi|)\,D \sigma(\xi)\right)\cdot \nu d\hausd dt \Bigg|\nonumber \\
&\qquad\qquad \leq C\int_{0}^{\horiz}\re{\chi(t)}{\xi(t)}dt.
\end{align}

The third integral on the left-hand side can be expanded in a way that resembles the approximate evolution equations (\ref{cal1}) and (\ref{cal2}):
\begin{align}\label{relen-rem-aee}
\int_{0}^{\horiz}&\int_{\redbd A(t)} \partial_{t}\left(|\xi|\psi(|\xi|)\,D \sigma(\xi)\right)\cdot \nu d\hausd dt = \int_{0}^{\horiz}\int_{\redbd A(t)} (\nu \otimes \partial_{t}\xi) : M(\xi)\, d\hausd dt\nonumber\\
&= \int_{0}^{\horiz}\int_{\redbd A(t)} \left((\nu-\xi) \otimes \left(\partial_{t}\xi + (B\cdot \nabla)\xi + (\nabla B)\trans\xi \right)\right) : M(\xi)\, d\hausd dt\nonumber\\
&\quad +\int_{0}^{\horiz}\int_{\redbd A(t)} \left(\xi \otimes \left(\partial_{t}\xi + (B\cdot \nabla)\xi\right)\right) : M(\xi)\, d\hausd dt\nonumber\\
&\quad -\int_{0}^{\horiz}\int_{\redbd A(t)} \left(\nu \otimes (B\cdot \nabla)\xi\right) : M(\xi)\, d\hausd dt\nonumber\\
&\quad -\int_{0}^{\horiz}\int_{\redbd A(t)} \left(\xi \otimes (\nu - \xi)\right) : \left((\nabla B)M(\xi)\trans\right) d\hausd dt.
\end{align}

Here, the first two integrals in the right-hand side expression are already in a suitable form: For the first integral this can be seen directly from (\ref{cal1}), whereas for the second term a computation yields 
\begin{align*}
\left(\xi \otimes \left(\partial_{t}\xi + (B\cdot \nabla)\xi\right)\right) : M(\xi)&=|\xi|\psi(|\xi|)\,\xi \cdot D^{2}\sigma(\xi)\left(\partial_{t}\xi + \left(B\cdot \nabla\right)\xi\right)\\
&\quad + \left(\frac{\psi(|\xi|)}{|\xi|}+\psi^{\prime}(|\xi|)\right)(D\sigma(\xi) \cdot \xi)\left(\xi \cdot (\partial_{t}\xi +(B\cdot \nabla)\xi)\right)\\
&=|\xi|\psi(|\xi|)\left(\partial_{t}\xi + \left(B\cdot \nabla\right)\xi\right) \cdot D^{2}\sigma(\xi)\xi\\
&\quad + \left(\frac{\psi(|\xi|)}{|\xi|}+\psi^{\prime}(|\xi|)\right)\sigma(\xi)\,\xi \cdot (\partial_{t}\xi+(B\cdot \nabla)\xi)\\
&= \left(\frac{\psi(|\xi|)}{|\xi|}+\psi^{\prime}(|\xi|)\right)\sigma(\xi)\,\xi \cdot (\partial_{t}\xi+(B\cdot \nabla)\xi),
\end{align*}
so that one can invoke (\ref{cal2}). Here we have used a fact that follows from the positive $0$-homogeneity of $D\sigma$, namely that $D^{2}\sigma(\xi)\xi = \frac{d}{ds}\Big|_{s=0}D\sigma(\e^{s}\xi)=\frac{d}{ds}\Big|_{s=0}D\sigma(\xi)=0$.

\vspace{10pt}

Furthermore, an application of the chain rule shows that 
\begin{equation*}
(\nu \otimes (B\cdot \nabla)\xi):M(\xi)=\nu \cdot (B\cdot\nabla)(F(\xi)),
\end{equation*} and we will from now on use this slightly shorter formulation for the third right-hand side integral of (\ref{relen-rem-aee}).

By using (\ref{relen-rem-aee}) and plugging in the computations for the three right-hand side integrals as well as the weak formulation of anisotropic mean curvature flow in (\ref{amcf-df}), we obtain
\begin{align}\label{relen-rem-notime}
&\Bigg| \int_{0}^{\horiz}\int_{\redbd A(t)}\frac{1}{\mu(\nu)}VB\cdot \nu d\hausd dt +  \int_{0}^{\horiz}\int_{\redbd A(t)}\dvg(F(\xi))B\cdot \nu d\hausd dt\nonumber \\
&+\int_{0}^{\horiz}\int_{\redbd A(t)} \partial_{t}\left(|\xi|\psi(|\xi|)\,D \sigma(\xi)\right)\cdot \nu d\hausd dt\Bigg|\nonumber \\
&\qquad\qquad \leq \Bigg| -  \int_{0}^{\horiz}\int_{\redbd A(t)}\nabla B : \left(\sigma(\nu)I_{d}-\nu \otimes D\sigma(\nu) \right)d\hausd dt \nonumber \\
&\qquad \qquad \quad +  \int_{0}^{\horiz}\int_{\redbd A(t)}\dvg(F(\xi))B\cdot \nu d\hausd dt\nonumber \\
&\qquad\qquad\quad-\int_{0}^{\horiz}\int_{\redbd A(t)} \nu \cdot (B\cdot \nabla)\left(F(\xi)\right) d\hausd dt\nonumber\\
&\qquad\qquad\quad- \int_{0}^{\horiz}\int_{\redbd A(t)}\left(\xi \otimes (\nu - \xi)\right) : \left((\nabla B)M(\xi)\trans\right) d\hausd dt\Bigg|\nonumber \\
&\qquad \qquad \quad+C\int_{0}^{\horiz}\re{\chi(t)}{\xi(t)}dt.
\end{align}

It is now advisable to expand the first right-hand side integral and to add zero in order to isolate two more errors controlled by the time-integrated relative entropy:
\begin{align}\label{relen-rem-curvterm}
-  \int_{0}^{\horiz}\int_{\redbd A(t)}&\nabla B : \left(\sigma(\nu)I_{d}-\nu \otimes D\sigma(\nu) \right)d\hausd dt\nonumber\\
&=-  \int_{0}^{\horiz}\int_{\redbd A(t)}\sigma(\nu)\dvg B\, d\hausd dt\nonumber\\
&\quad +  \int_{0}^{\horiz}\int_{\redbd A(t)}\nu \cdot \left(F(\nu)\cdot\nabla\right) B \,d\hausd dt\nonumber\\
&=-  \int_{0}^{\horiz}\int_{\redbd A(t)}\left(\sigma(\nu)-F(\xi)\cdot\nu\right)\dvg B\, d\hausd dt\nonumber\\
&\quad +  \int_{0}^{\horiz}\int_{\redbd A(t)}\left(\nu-\xi\right) \cdot \left((F(\nu)-F(\xi))\cdot\nabla\right) B \,d\hausd dt\nonumber\\
&\quad -  \int_{0}^{\horiz}\int_{\redbd A(t)} F(\xi)\cdot \nu \dvg B\, d\hausd dt\nonumber\\
&\quad +  \int_{0}^{\horiz}\int_{\redbd A(t)}\nu \cdot \left(F(\xi)\cdot\nabla\right) B \,d\hausd dt\nonumber\\
&\quad +  \int_{0}^{\horiz}\int_{\redbd A(t)}\xi \cdot \left((F(\nu)-F(\xi))\cdot\nabla\right) B \,d\hausd dt.
\end{align}

The first two integrals are controlled by $\int_{0}^{\horiz}\re{\chi(t)}{\xi(t)}dt$; in the case of the second integral this follows from the local Lipschitz continuity of $F$.

We approach the last integral in (\ref{relen-rem-curvterm}) by an application of Taylor's theorem and show that it cancels with the second to last integral in (\ref{relen-rem-notime}). Indeed, one computes that
\begin{align}\label{relen-rem-taylor}
\xi \cdot \left((F(\nu)-F(\xi))\cdot\nabla\right)& B - \left(\xi \otimes (\nu - \xi)\right) : \left((\nabla B)M(\xi)\trans\right) \nonumber\\
&=\xi \cdot \left((M(\xi)(\nu-\xi))\cdot\nabla\right) B - \xi \cdot (\nabla B)M(\xi)\trans(\nu - \xi) \nonumber\\
&\quad + \xi \cdot \left(\int_{0}^{1}(1-r)\left((\nu-\xi)\cdot \nabla\right)M(r\nu + (1-r)\xi)(\nu-\xi)dr \cdot \nabla \right)B\nonumber\\
&=\xi \cdot (\nabla B)\left(M(\xi)-M(\xi)\trans \right)(\nu - \xi) \nonumber\\
&\quad + \xi \cdot \left(\int_{0}^{1}(1-r)\left((\nu-\xi)\cdot \nabla\right)M(r\nu + (1-r)\xi)(\nu-\xi)dr \cdot \nabla \right)B\nonumber \\
&=\left(\frac{\psi(|\xi|)}{|\xi|}+\psi^{\prime}(|\xi|)\right)\xi \cdot (\nabla B)\left( D\sigma(\xi)\otimes \xi - \xi \otimes D\sigma(\xi)\right)(\nu - \xi) \nonumber\\
&\quad + \xi \cdot \left(\int_{0}^{1}(1-r)\left((\nu-\xi)\cdot \nabla\right)M(r\nu + (1-r)\xi)(\nu-\xi)dr \cdot \nabla \right)B\nonumber \\
&=\left(\frac{\psi(|\xi|)}{|\xi|}+\psi^{\prime}(|\xi|)\right)\left(\xi \cdot (\nabla B)D\sigma(\xi)\right)\xi \cdot (\nu - \xi) \nonumber\\
&\quad -\left(\frac{\psi(|\xi|)}{|\xi|}+\psi^{\prime}(|\xi|)\right)\left(D\sigma(\xi)\cdot(\nu-\xi)\right)\xi \cdot(\xi \cdot \nabla)B \nonumber\\
&\quad + \xi \cdot \left(\int_{0}^{1}(1-r)\left((\nu-\xi)\cdot \nabla\right)M(r\nu + (1-r)\xi)(\nu-\xi)dr \cdot \nabla \right)B,
\end{align}
where now each summand is controlled by $\sigma(\nu)-|\xi|\psi(|\xi|)D\sigma(\xi)\cdot \nu$ by Lemma \ref{calibration-lem} (iii) and (vi). Here we use that $\sigma \in C^{3}\big(\real^{d}\setminus\{0\}\big)$.

Plugging in (\ref{relen-rem-curvterm}) and (\ref{relen-rem-taylor}) into (\ref{relen-rem-notime}) yields
\begin{align}\label{relen-rem-dvg}
&\Bigg| \int_{0}^{\horiz}\int_{\redbd A(t)}\frac{1}{\mu(\nu)}VB\cdot \nu d\hausd dt +  \int_{0}^{\horiz}\int_{\redbd A(t)}\dvg(F(\xi))B\cdot \nu d\hausd dt\nonumber \\
&+\int_{0}^{\horiz}\int_{\redbd A(t)} \partial_{t}\left(|\xi|\psi(|\xi|)\,D \sigma(\xi)\right)\cdot \nu d\hausd dt\Bigg|\nonumber \\
&\qquad\qquad=\Bigg| -  \int_{0}^{\horiz}\int_{\redbd A(t)} F(\xi)\cdot \nu \dvg B\, d\hausd dt\nonumber\\
&\qquad\qquad\quad +  \int_{0}^{\horiz}\int_{\redbd A(t)}\nu \cdot \left(F(\xi)\cdot\nabla\right) B \,d\hausd dt\nonumber\\
&\qquad \qquad \quad +  \int_{0}^{\horiz}\int_{\redbd A(t)}\dvg(F(\xi))B\cdot \nu d\hausd dt\nonumber \\
&\qquad\qquad\quad-\int_{0}^{\horiz}\int_{\redbd A(t)} \nu \cdot (B\cdot \nabla)\left(F(\xi)\right) d\hausd dt\Bigg|\nonumber \\
&\qquad \qquad \quad+C\int_{0}^{\horiz}\re{\chi(t)}{\xi(t)}dt.
\end{align}

The four remaining surface integrals can be written as one integral involving the divergence of a matrix field since
\begin{equation*}
\dvg\left(B\otimes F(\xi)-F(\xi)\otimes B\right)=\dvg (F(\xi))B+(F(\xi)\cdot \nabla)B-\dvg (B)F(\xi) - (B\cdot \nabla)(F(\xi)).
\end{equation*}
As a result of an integration by parts and the symmetry relation $\dvg(\dvg(a\otimes b))= \dvg(\dvg(b\otimes a))$, one can then see that
\begin{align}\label{relen-rem-symm}
&-  \int_{0}^{\horiz}\int_{\redbd A(t)} F(\xi)\cdot \nu \dvg B\, d\hausd dt\nonumber\\
&+  \int_{0}^{\horiz}\int_{\redbd A(t)}\nu \cdot \left(F(\xi)\cdot\nabla\right) B \,d\hausd dt\nonumber\\
& +  \int_{0}^{\horiz}\int_{\redbd A(t)}\dvg(F(\xi))B\cdot \nu d\hausd dt\nonumber \\
&-\int_{0}^{\horiz}\int_{\redbd A(t)} \nu \cdot (B\cdot \nabla)\left(F(\xi)\right) d\hausd dt \nonumber \\
&\qquad \qquad = \int_{0}^{\horiz}\int_{\redbd A(t)}\nu \cdot \dvg\left(B\otimes F(\xi)-F(\xi)\otimes B\right) d\hausd dt\nonumber\\
&\qquad \qquad= \int_{0}^{\horiz}\int_{A(t)}\dvg\left(\dvg\left(B\otimes F(\xi)-F(\xi)\otimes B \right)\right)dx dt\nonumber \\
&\qquad\qquad=0.
\end{align}
Plugging in (\ref{relen-rem-symm}) into (\ref{relen-rem-dvg}) yields (\ref{relen-rem-claim}).

In a last step, plugging in (\ref{relen-sq1}), (\ref{relen-sq2}), and (\ref{relen-rem-claim}) into (\ref{relen-general}) yields
\begin{align*}
\re{\chi(\horiz)}{\xi(\horiz)}&\leq \re{\chi(0)}{\xi(0)}+{ C\int_{0}^{\horiz}\re{\chi(t)}{\xi(t)}dt}\nonumber\\
&\quad - \frac{1}{4\max_{S^{d-1}}\mu}\int_{0}^{\horiz}\int_{\redbd A(t)}|V-B\cdot \xi|^{2}d\hausd dt.
\end{align*}
This concludes the proof of Lemma \ref{entropy-stability}.

\section*{Acknowledgments}
The present paper is an extension of the third author's master's thesis at the  University of Bonn.
This project has received funding from the Deutsche Forschungsgemeinschaft (DFG, German Research Foundation) under Germany's Excellence Strategy -- EXC-2047/1 -- 390685813 and the Graduiertenkolleg (Research Training Group) 2339 ``Interfaces, Complex Structures, and Singular Limits in Continuum Mechanics''.

\addcontentsline{toc}{section}{References}
\frenchspacing
\bibliographystyle{abbrv}

\bibliography{amcf-arxiv.bbl}

\end{document}